\documentclass[12pt,a4paper]{amsart}
 
\usepackage{fullpage}
\usepackage[latin1]{inputenc} 
\usepackage{amsmath,amssymb,amsthm,bbm} 
\usepackage{enumitem}
\usepackage[overload]{textcase}
\usepackage{epsfig,psfrag,color,ifthen}
\usepackage[english,algoruled,lined,noresetcount,norelsize]{algorithm2e}
\usepackage{mathrsfs}
 
\setlist{itemsep=1pt,parsep=0pt,topsep=2pt,partopsep=0pt} 
\setenumerate{leftmargin=*} 
\setitemize{leftmargin=\parindent}

\def\itm#1{\rm ({#1})} 
\def\itmit#1{\itm{\it #1\,}} 
\def\rom{\itmit{\roman{*}}}
\def\abc{\itmit{\alph{*}}}

\let\subset\subseteq 
\let\eps\varepsilon 
\let\rho\varrho 
\let\theta\vartheta
\def\dcup{\mathbin{\dot{\cup}}}

\def\soft#1{\leavevmode\setbox0=\hbox{h}\dimen7=\ht0\advance
    \dimen7 by-1ex\relax\if t#1\relax\rlap{\raise.6\dimen7
    \hbox{\kern.3ex\char'47}}#1\relax\else\if T#1\relax
    \rlap{\raise.5\dimen7\hbox{\kern1.3ex\char'47}}#1\relax
    \else\if d#1\relax\rlap{\raise.5\dimen7\hbox{\kern.9ex
    \char'47}}#1\relax\else\if D#1\relax\rlap{\raise.5\dimen7
    \hbox{\kern1.4ex\char'47}}#1\relax\else\if l#1\relax
    \rlap{\raise.5\dimen7\hbox{\kern.4ex\char'47}}#1\relax
    \else\if L#1\relax\rlap{\raise.5\dimen7\hbox{\kern.7ex
    \char'47}}#1\relax\else\message{accent \string\soft
    \space #1 not defined!}#1\relax\fi\fi\fi\fi\fi\fi}

\newtheorem{theorem}{Theorem}
\newtheorem{lemma}[theorem] {Lemma}    
\newtheorem{corollary}[theorem] {Corollary}    
\newtheorem{conjecture}[theorem] {Conjecture}   
\newtheorem{problem}[theorem] {Problem}   
\newtheorem{proposition}[theorem] {Proposition}   
\newtheorem{claim}[theorem] {Claim}  
\newtheorem{fact}[theorem]{Fact}
\newtheorem{definition}[theorem] {Definition}  

\newcommand{\oldqed}{}
\def\endofClaim{\hfill\scalebox{.6}{$\Box$}}

\newenvironment{claimproof}[1][Proof]{
  \renewcommand{\oldqed}{\qedsymbol}
  \renewcommand{\qedsymbol}{\endofClaim}
  \begin{proof}[#1]
}{
  \end{proof}
  \renewcommand{\qedsymbol}{\oldqed}
}

\newcommand{\By}[2]{\overset{\mbox{\tiny{#1}}}{#2}} 
\newcommand{\ByRef}[2]{   \By{\eqref{#1}}{#2} }

\newcommand{\eqByRef}[1]{ \ByRef{#1}{=} }

\newcommand{\leByRef}[1]{ \ByRef{#1}{\le} }

\newcommand{\Prob}{\mathbb{P}}
\newcommand{\REALS}{\mathbb{R}}
\newcommand{\NATS}{\mathbb{N}}
\newcommand{\Exp}{\mathop{\mathbb{E}}} 
\newcommand{\Hy}{\mathcal{H}} 
\newcommand{\cG}{\mathcal{G}} 
\newcommand{\cR}{\mathcal{R}} 
\newcommand{\cJ}{\mathcal{J}} 
\newcommand{\cL}{\mathcal{L}} 
\newcommand{\cX}{\mathcal{X}}
\newcommand{\cS}{\mathcal{S}} 
\newcommand{\K}{\mathcal{K}} 
\newcommand{\Part}{\mathcal{P}} 
\newcommand{\Qart}{\mathcal{Q}} 
\newcommand{\Qb}{\mathbf{Q}} 
\newcommand{\entropy}{\mathbf{H}}
\newcommand{\paths}{\mathscr{P}}

\newcommand{\cross}{\textrm{Cross}}
\newcommand{\reld}{d^*}
\newcommand{\reldeg}{\overline{\deg}} 
\newcommand{\HOM}{\text{EMB}} 
\newcommand{\INJ}{\text{INJ}}
\newcommand{\CROSS}{\mathrm{CROSS}} 
\newcommand{\DIST}{\mathrm{DIST}} 
\newcommand{\SLICE}{\mathrm{SL}} 
\newcommand{\Hskel}{\Hy^{skel}}
\newcommand{\ex}{\mathrm{ex}}

\title{Tight cycles and regular slices in dense hypergraphs}

  \author[P. Allen]{Peter Allen} \address{
    Department of Mathematics, London School of Economics, Houghton Street,
London WC2A 2AE, U.K.
  }
  \email{p.d.allen@lse.ac.uk}

  \author[J. B\"ottcher]{Julia B\"ottcher} \address{
    Department of Mathematics, London School of Economics, Houghton Street,
London WC2A 2AE, U.K.
  }
  \email{j.boettcher@lse.ac.uk}

  \author[O. Cooley]{Oliver Cooley} \address{
    Graz University of Technology, Institute of Optimization and Discrete Mathematics,
    Steyrer\-gasse~30, 8010 Graz, Austria
  }
  \email{cooley@math.tugraz.at}

  \author[R. Mycroft]{Richard Mycroft} \address{
    School of Mathematics, University of Birmingham, Edgbaston, Birmingham B15
    2TT, UK
  }
  \email{r.mycroft@bham.ac.uk}

  \thanks{
    PA was partially supported by FAPESP (Proc.~2010/09555-7); JB by FAPESP
    (Proc.~2009/17831-7); PA and JB by CNPq (Proc.~484154/2010-9); OC by the DFG (TA~309/2-2);
    The cooperation of the authors was supported by a joint CAPES-DAAD project
    (415/ppp-probral/po/D08/11629, Proj.~no.~333/09).
    The authors are grateful to NUMEC/USP, N\'ucleo de Modelagem Estoc\'astica e
    Complexidade of the University of S\~ao Paulo, and Project MaCLinC/USP, for
    supporting this research.
  }
 
\date{\today} 

\begin{document} 

\begin{abstract} 
  We study properties of random subcomplexes of partitions returned by (a
  suitable form of) the Strong Hypergraph Regularity Lemma, which we call
  regular slices. We argue that these subcomplexes capture many important
  structural properties of the original hypergraph. Accordingly we advocate
  their use in extremal hypergraph theory, and explain how they can lead
  to considerable simplifications in existing proofs in this field. We also
  use them for establishing the following two new results.

  Firstly, we prove a hypergraph extension of the Erd\H{o}s-Gallai Theorem:
  for every $\delta>0$ every sufficiently large $k$-uniform hypergraph
  with at least $(\alpha+\delta)\binom{n}{k}$ edges contains a tight cycle of
  length $\alpha n$ for each $\alpha\in[0,1]$.

  Secondly, we find (asymptotically) the minimum codegree requirement for a
  $k$-uniform $k$-partite hypergraph, each of whose parts has $n$ vertices,
  to contain a tight cycle of length $\alpha kn$, for each $0<\alpha<1$.
\end{abstract} 

\maketitle


\section{Introduction}

The Szemer\'edi Regularity Lemma~\cite{SzeReg} is a powerful tool 
in extremal graph theory. It is probably fair to say that the majority of the
advances in the last decade in extremal graph theory either rely on, or at least
were inspired by, the Regularity Lemma. 
Finding the right extension of this result for hypergraphs turned out to be
a challenging endeavour, which culminated in the proof of the
Strong Hypergraph Regularity Lemma together with a corresponding Counting
Lemma (see \cite{Gowers,NRS,RS2,RS,RSk}), which provide an analogous machinery for extremal problems
in hypergraphs. The difficulty with these tools is their technical
intricacy, which leads to significant additional complexity in applications
of the regularity method in extremal hypergraph theory.

In this paper we argue that often much of this complexity can be avoided by
using instead of the complicated structure returned by the Strong
Hypergraph Regularity Lemma a more accessible structure, which we call a
\emph{regular slice}. We provide a lemma (which is a consequence of the Strong
Hypergraph Regularity Lemma) that asserts the existence of such
regular slices which inherit enough structure from the original hypergraph
to be useful for embedding problems. In addition we provide two
applications of this lemma concerning the existence of tight cycles in
dense hypergraphs.

In the remainder of this introduction we will first describe these
applications, and then provide more details on our lemma concerning regular
slices.

\subsection{Cycles in dense graphs and hypergraphs}

Dirac's Theorem~\cite{Dirac} asserts that any $n$-vertex graph~$G$ with minimum degree 
$\delta(G)\ge n/2$ contains a Hamilton cycle. Bondy~\cite{Bondy} extended
this result and showed that Hamiltonian $n$-vertex graphs with at least
$n^2/4$ edges are in fact either complete bipartite or pancyclic, that is,
they contain cycles of all lengths $\ell\le n$.  Cycles of shorter lengths
were also considered by Erd\H{o}s and Gallai~\cite{ErdGal}, whose
celebrated theorem states that for any integer $d \geq 3$, any $n$-vertex
graph with at least $(d-1)(n-1)/2 + 1$ edges contains a cycle of length at
least $d$.  This is the best possible bound which is linear in $d$. The analogous
problem for minimum degree conditions was settled by
Alon~\cite{Alon}, who proved that large graphs with minimum degree at least $n/k$
contain a cycle of length at least $n/(k-1)$. Nikiforov and
Schelp~\cite{NikiSchelp} considered exact cycle lengths. They showed that
for all $\alpha>0$ there is~$n_0$ such that for $n\ge n_0$ all $n$-vertex
graphs~$G$ with $\delta(G)\ge\alpha n$ contain all even $\ell$-cycles with
$4\le \ell\le\delta(G)+1$ (they also considered odd cycles for
non-bipartite~$G$; and even sharper results are obtained in~\cite{Allen}).

For hypergraphs much less is known. A \emph{tight cycle}\footnote{There
  are several other notions of cycles in hypergraphs. Since our results are
  about tight cycles we concentrate on these here.} in a $k$-uniform
hypergraph~$G$ is a cyclically ordered list of vertices such that every $k$
consecutive vertices form an edge in~$G$.  Only recently R\"odl, Ruci\'nski
and Szemer\'edi~\cite{RRSHam} established an approximate extension of
Dirac's Theorem for $k$-uniform hypergraphs (the $3$-uniform case was
resolved exactly by the same authors in~\cite{RRS3unif}): for all $\gamma>0$ every
sufficiently large $k$-uniform hypergraph such that every $(k-1)$-set of
vertices lies in at least $(\frac12+\gamma)n$ edges has a tight Hamilton
cycle. A hypergraph analogue of the Erd\H{o}s-Gallai Theorem is not yet
known.  Gy\H{o}ri, Katona and Lemons~\cite{GyKatLem} made a first step in
this direction. They showed that $k$-uniform hypergraphs~$G$ with more than
$(\alpha n-k)\binom{n}{k-1}$ edges contain a tight path on $\alpha n$
vertices. Note that this statement is vacuous for $\alpha>1/k$.

Our first result is an approximate Erd\H{os}-Gallai type result that
establishes (up to the error term $\delta$) the best possible linear
density bound for the containment of tight cycles of a given length.

\begin{theorem}\label{thm:HypErdGal}
  For every positive~$\delta$ and every integer $k \ge 3$, there is an
  integer~$n_*$ such that the following holds for all $\alpha\in[0,1]$.
  If~$G$ is a $k$-uniform hypergraph on $n\ge n_*$ vertices with
  $e(G)\ge(\alpha+\delta)\binom{n}{k}$, then~$G$ contains a tight cycle of length~$\ell$ for every
      $\ell\le\alpha n$ that is divisible by~$k$.
\end{theorem}

Partite versions of the graph embedding results mentioned above also have
been studied extensively. Moon and Moser~\cite{MooMos} showed that the
minimum degree condition in Dirac's Theorem can (almost) be halved when~$G$
is balanced bipartite. They proved that a bipartite graph~$G$ with~$n$
vertices in each partition class contains a Hamilton cycle if
$\delta(G)\ge(n/2)+1$. A corresponding (bi)pancyclicity result was
established by Schmeichel and Mitchem~\cite{SchMit} and balanced
$k$-partite graphs were considered by Chen, Faudree, Gould, Jacobson and Lesniak in~\cite{CheFauGouJacLes}.
For hypergraphs, R\"odl and Ruci\'nski~\cite{RSHamSurvey} proved that
for all $\gamma>0$ every sufficiently large $k$-uniform hypergraph which is
$k$-partite with~$n$ vertices in each part, such that each $(k-1)$-set of
vertices, one from each partition class, is contained in at least
$(\frac12+\gamma)n$ edges, contains a tight Hamilton cycle.

Our second result establishes asymptotically best possible codegree
bounds for the containment of tight cycles of a given length in the partite
setting.

\begin{theorem}\label{thm:PartiteCycle}
For every positive $\delta$ and every integer $k \ge 3$ there is an integer
 $n_{*}$ such that the following holds for each $\alpha\in[0,1]$. If $G$ is a $k$-uniform $k$-partite 
 hypergraph with parts of size $n\ge n_{*}$, such that any collection of $k-1$
 vertices, one in each of $k-1$ parts of $G$, lies in at least
$(\alpha+\delta)n$ edges of $G$, then
 \begin{enumerate}[label=\abc] 
 \item\label{thm:PartiteCycle:small} $G$ contains a tight cycle of length $\ell$ for every $\ell\le
   \alpha kn$ that is divisible by $k$, and
 \item\label{thm:PartiteCycle:big} if $\alpha\ge\frac{1}{2}$ then $G$ contains a tight cycle of length
   $\ell$ for every $\ell\le(1-\delta)kn$ that is divisible by $k$.
\end{enumerate}
\end{theorem}


We believe that with substantial additional work
part~\ref{thm:PartiteCycle:big} of this Theorem would also follow from the
proof in~\cite{RSHamSurvey}. However, our new tools allow for a much
simpler proof.

\subsection{Regular slices.} 
Briefly, given an $n$-vertex graph~$G$ the Regularity Lemma returns an
equipartition of the vertices of~$G$ into a constant number of parts such
that almost all of the bipartite subgraphs induced by pairs of parts
approximate random graphs. This so-called \emph{regular partition} can be represented by
a constant size weighted graph, the \emph{weighted reduced graph}, in
which the (weighted) density of copies of any small graph approximates the
density of copies in the original graph. 
In fact, we can use a regular partition and the reduced graph also to find
copies of large bounded degree subgraphs in~$G$, among other things. In this way we can reduce difficult extremal graph problems to relatively simple problems
on the weighted reduced graph. This
approach to extremal graph theory is called the regularity method.

It is then natural to ask for an analogue of regular partitions and reduced
graphs for hypergraphs. However, the partition returned by the Strong
Hypergraph Regularity Lemma is no longer given by a vertex partition and cannot be represented
by a weighted hypergraph.  Instead, the partition has the structure of a
weighted multi-complex with edges of sizes up to and including $k$, which
makes applications of the Strong Hypergraph Regularity Lemma significantly
harder. Whilst simpler forms of hypergraph regularity do exist
(see~\cite{WeakReg,WeakCount}), they are substantially less powerful
for hypergraph embedding.

The \emph{Regular Slice Lemma} (Lemma~\ref{lem:simpreg}) that we propose in
this paper for approaching problems in extremal hypergraph theory bypasses
these difficulties by taking a random subcomplex of the weighted
multi-complex returned by the Strong Hypergraph Regularity Lemma (after a
considerable number of suitable modifications). Related ideas were applied
already by Haxell, {\L}uczak, Peng, R{\"o}dl, Ruci{\'n}ski and
Skokan~\cite{3CycleRamsey}. The advantage of this random subcomplex, which we call a \emph{regular slice}, is that it does correspond to a vertex partition and a weighted reduced hypergraph. The disadvantage of this
approach is that a regular slice discards most of the hyperedges of the
original hypergraph. However, as our lemma asserts, a lot of information
about the original hypergraph is still captured: edge densities, and more
generally small subgraph densities, minimum degree and codegree conditions,
and edge densities in link hypergraphs are approximated in regular
slices. Since these are exactly the type of conditions which typically
appear in extremal hypergraph theory we believe that regular slices will be
useful for solving many problems in this area.

\subsection{Structure of the paper.} 
In the next section we give the definitions needed to work with hypergraphs,
then in Section~\ref{sec:LB} we discuss lower bounds for
Theorems~\ref{thm:HypErdGal} and~\ref{thm:PartiteCycle}.
Section~\ref{sec:slice} introduces the definitions required for and the statement of the Regular Slice Lemma. It also includes the statement of a Cycle Embedding Lemma, which enables us to apply the Regular Slice Lemma to prove our two main theorems.
In Sections~\ref{sec:EG} and~\ref{sec:partite} we assume these results to prove
Theorems~\ref{thm:HypErdGal} and~\ref{thm:PartiteCycle} respectively.
Section~\ref{sec:reg} deals with the definitions and machinery of hypergraph
regularity, while Section~\ref{sec:regslice} contains the proof of the Regular Slice Lemma, and in Section~\ref{sec:embedding}, we prove the Cycle Embedding Lemma. We
conclude with a discussion of open problems and possible further applications in
Section~\ref{sec:conclude}.
Finally, Appendix~\ref{app:count} contains the derivation of the version of the Hypergraph Counting Lemma we use, and Appendix~\ref{app:strReg} contains the proof of a variant of the Strong Hypergraph Regularity Lemma which is used in the proof of the Regular Slice Lemma.


\section{Preliminaries}

A \emph{hypergraph} consists of a vertex set $V$ and an edge set $E$, where each edge $e \in E$ is a subset of $V$.  We frequently identify a hypergraph $\Hy$ with
its edge set, writing $e \in \Hy$ to mean that $e$ is an edge of $\Hy$, and
$|\Hy|$ for $e(\Hy)$, the number of edges of $\Hy$.
A \emph{$k$-uniform hypergraph}, or \emph{$k$-graph}, is a hypergraph in which every edge has
size $k$; it is \emph{$\ell$-partite} for $\ell \ge k$ if there exists a partition of the vertex set into $\ell$ parts so that every edge has at most one vertex in each part.
Given a hypergraph $\Hy$, we denote by $\Hy^{(i)}$ the $i$-graph on
$V(\Hy)$ formed by the \emph{$i$-edges} of $\Hy$, i.e.~the edges of cardinality $i$, and write $e_i(\Hy) := |\Hy^{(i)}|$ for the number of such edges. The \emph{order} of a hypergraph is the
number of vertices $v(\Hy) := |V(\Hy)|$. The \emph{degree} of a set of $i$ vertices in a $k$-graph, where $1\le i \le k-1$, is the number of edges containing all $i$ vertices. In the case when $i=k-1$ we talk of the \emph{codegree}. If $\Hy$ is a hypergraph on $V$, and $V'
\subseteq V$, then the induced subgraph $\Hy[V']$ is the hypergraph on $V'$ whose 
edges are all $e \in \Hy$ with $e \subseteq V'$. A \emph{$k$-complex} $\Hy$ is a hypergraph in which all edges have size at most $k$,
and which is down-closed: that is, if $e \in \Hy$ is an edge, and $e' \subset e$,
then $e' \in \Hy$. Note that for any non-empty complex $\Hy$ we have $\emptyset \in \Hy$. We prefer the terms \emph{subgraph} and \emph{subcomplex} to subhypergraph, sub-$i$-graph, or sub-$k$-complex.

Throughout this paper we will maintain the convention of using normal
letters for uniform hypergraphs, and calligraphic letters for not necessarily uniform
hypergraphs (which will usually be complexes). The exceptions to this rule are the definition of $\Hy^{(i)}$ above and of $\Hy_X$ in Section~\ref{sec:slice}: both define a specific uniform subgraph of a not necessarily uniform hypergraph $\Hy$.
 
Let $G$ be a $k$-graph. A
\emph{tight path} in $G$ is an ordered list of distinct vertices of $G$, such that
each set of $k$ consecutive vertices induces an edge of $G$; a \emph{tight cycle} is
a cyclically ordered list with the same property. The \emph{length} of a tight path or cycle is the number of edges in the path or cycle; the number of vertices in a tight cycle is therefore equal to its length, whilst the number of vertices in a tight path is $k-1$ greater than its length. We will denote the tight cycle or path on $\ell$ vertices by $C_\ell$ or $P_\ell$ respectively.

Any $k$-graph $G$ naturally corresponds to a $k$-complex $\cG$, whose edges are all subsets $e' \subseteq e$ of edges $e \in G$. We refer to $\cG$ as the complex \emph{generated by the down-closure of $G$}. We can then work with the $k$-complex $\cG$ to obtain useful information about the $k$-graph $G$. We say that a set $S$ of $k+1$ vertices of $G$ is \emph{supported on $G$} if every subset $S' \subseteq S$ of size $k$ is an edge of $G$ (i.e.~$S$ forms a clique in $G$). Similarly, we say that a $(k+1)$-graph $G'$ on $V(G)$ is \emph{supported on $G$} if every edge of $G'$ is supported on $G$. Note in particular that if $\Hy$ is a $k$-complex, then $\Hy^{(i)}$ is supported on $\Hy^{(i-1)}$ for any $1 \leq i \leq k$.

We write $[r]$ to denote the set $\{1, 2, \dots, r\}$. For a set $A$, we often
write $\binom{A}{r}$ to denote the collection of subsets of $A$ of size $r$. We
use the notation $x=a\pm b$ to denote $a-b\leq x \leq a+b$. If $S(x,y)$ is a
statement, we say `$S(x,y)$ holds for $x \ll y$' if for any $y >
0$ there exists $x_0 > 0$ such that for any $0<x \leq x_0$ the statement
$S(x,y)$ holds; similar expressions with more constants are defined analogously.
Throughout this paper floors and ceilings are often omitted where they do not
affect the argument.

\section{Lower bounds}
\label{sec:LB}

\subsection{Lower bounds for Theorem~\ref{thm:HypErdGal}.}

Given a $k$-graph $H$, we define $\ex(n,H)$ to be the maximum number of edges an $n$-vertex $k$-graph can have without containing $H$ as a subgraph. Then Theorem~\ref{thm:HypErdGal} says that $\ex\big(n,P_{\alpha n}\big),\ex\big(n,C_{\alpha n}\big)\le\big(\alpha+o(1)\big)\binom{n}{k}$. We will now establish several lower bounds on $\ex\big(n,P_{\alpha n}\big)$ and $\ex\big(n,C_{\alpha n}\big)$.

Gy\H{o}ri, Katona and Lemons~\cite{GyKatLem} (who were interested in the case
$\alpha n=\ell$ where $\ell$ is constant) observed that
given any set system $\cS$ on $n$ vertices with sets of size at most $\alpha n-1$, no pair of which intersect in more than $k-2$ vertices, one can construct a $k$-graph $G$ without any tight path on $\alpha n=\ell$
vertices by taking every $k$-set contained in a set of $\cS$. In the event that
all members of $\cS$ have size exactly $\alpha n-1$ and every $(k-1)$-set in $[n]$ is in some member of $\cS$,
$\cS$ is a combinatorial design and the corresponding $G$ has $(\alpha
n-k)\binom{n}{k-1}/k \approx \alpha\binom{n}{k}$ edges. A celebrated recent result of Keevash~\cite{KeeDesign} states that for each fixed $\alpha n=\ell$ and $k$, these combinatorial designs exist for all sufficiently large $n$ satisfying a necessary divisibility condition. Gy\H{o}ri, Katona and Lemons (writing before~\cite{KeeDesign} appeared) used R\"odl's solution~\cite{RodlEH} of the Erd\H{o}s-Hanani
Conjecture to prove, for each $\delta>0$, the existence of hypergraphs with
$(1-\delta)\alpha\binom{n}{k}$ edges and no $\ell$-vertex tight path, provided that $n$ is sufficiently
large compared to $\alpha$. This provides the lower bound $\ex\big(n,P_{\alpha n}\big)\ge\big(1-o(1)\big)\alpha\binom{n}{k}$. Unfortunately both Keevash's and R\"odl's proofs only
work when $n$ is superexponentially large in $\ell$, and it is easy to
see that in fact no designs (or even approximations to designs) can exist if
$\ell\gg n^{1/2}$, so that this construction does not work in the constant $\alpha$ range where
Theorem~\ref{thm:HypErdGal} gives a non-trivial result.

A very similar construction gives $G$ without tight cycles on $\ell$ or
more vertices. Namely, if there exists a set system $\cS$ on $n-1$ vertices
with all sets of size $\ell-2$, covering each $(k-1)$-set exactly once,
then we can construct a graph $G$ whose edges are all $k$-sets contained in
members of $\cS$ together with all $k$-sets containing the $n$th vertex. We
have $e(G)=\frac{\ell-1}{k}\binom{n-1}{k-1}$, and it is easy to check
that $G$ contains no tight cycle on $\ell$ or more vertices. However, if we
only want to exclude tight cycles on exactly $\ell$ vertices, we obtain
\[\ex\big(n,C_\ell\big)\ge
\tfrac12\big(\tfrac{1}{k!}\big)^{\tfrac{\ell}{\ell-1}}n^{k-1+\tfrac{k-1}{\ell-1}}\]
by deleting an edge from each copy of $C_\ell$ in the random $k$-graph with edge probability
\[p=\big(\tfrac{2}{k!}n^{k-\ell}\big)^{\tfrac{1}{\ell-1}}\,.\]

A simple lower bound on both $\ex\big(n,P_{\alpha n}\big)$ and $\ex\big(n,C_{\alpha n}\big)$ is provided by the following construction. Let $G$
be the $k$-graph on vertex set $V=A\dcup B$, where $|A|=\alpha
n/k-1$ and $|B|=(1-\alpha /k)n+1$, and whose edge set is $E(G):=\{e\in
\binom{V}{k}:e\cap A\neq \emptyset\}$. It is easy to see that the longest tight
path in $G$ has at most $k|A|+k-1< \alpha n$ vertices (the longest tight cycle is even slightly shorter). For large $n$,
the number of edges in $G$ is 
\[ \left(1-(1-\tfrac{\alpha}{k})^k\right)\binom{n}{k} +
o(n^k) = \left(\sum_{i=1}^k (-1)^{i-1}
\binom{k}{i}\big(\tfrac{\alpha}{k}\big)^i\right) \binom{n}{k} + o(n^k)\,.\]
Hence $\ex\big(n,C_{\alpha n}\big),\ex\big(n,P_{\alpha n}\big)\ge\big(\alpha+O(\alpha^2)\big)\binom{n}{k}$, matching the first order asymptotics in Theorem~\ref{thm:HypErdGal} for small $\alpha$.

Moreover, for each $2\le r\le k$, if $\alpha n r/k >(1-\alpha)n+1$, then the $k$-graph $G$ on
vertex set $V=A\dcup B$, where $|A|=\alpha n-1$ and $|B|=(1-\alpha)n+1$, with edge set consisting of all $k$-sets either
contained in $A$ or with at least $r$ vertices in $B$, contains no $\alpha
n$-vertex tight cycle. The densest of these is the construction with $r=2$, which is permissible when $1-\alpha$ is small, giving $\ex\big(n,C_{\alpha n}\big)\ge\big(1-k\alpha^{k-1}(1-\alpha)+O(1/n)\big)\binom{n}{k}$. (Minor modifications of the construction give the same asymptotic bound for $\ex\big(n,P_{\alpha n}\big)$.) This tends to $1$ as $1-\alpha$ tends to zero, but does not match the first order asymptotics of Theorem~\ref{thm:HypErdGal}.

Finally, we note that the divisibility condition on $\ell$ in Theorem~\ref{thm:HypErdGal} is necessary, at least for
$\alpha \leq k!/k^k$, as can be seen by considering the complete $k$-partite $k$-graph with vertex classes of size $n/k$.

\subsection{Lower bounds for Theorem~\ref{thm:PartiteCycle}.}

Clearly Theorem~\ref{thm:PartiteCycle} part~\ref{thm:PartiteCycle:big} is asymptotically best possible. For part~\ref{thm:PartiteCycle:small}, consider the following construction. Given the $k$-partite
vertex set $V_1\cup\cdots\cup V_k$
with $|V_i|=n$ for each $i$ and $\alpha\le\frac{1}{2}$, we partition each $V_i$
into sets $V_i^0$ and $V_i^1$ of sizes respectively $(1-\alpha)n$ and $\alpha n$. We let a set of $k$ vertices $e$, with one vertex in each of
$V_1^{j_1},\ldots,V_k^{j_k}$, where $j_i\in\{0,1\}$ for each $i\in[k]$, be an edge of the $k$-graph $G$ if and only if
$j_1+\cdots+j_k$ is odd. It is easy to check that the partite codegree
condition of Theorem~\ref{thm:PartiteCycle} is almost satisfied, and similarly easy to check that a tight cycle
with one edge $e$ whose vertices are in $V_1^{j_1},\ldots,V_k^{j_k}$ must have
all its vertices in $V_1^{j_1},\ldots,V_k^{j_k}$, from which it follows that we
cannot have tight cycles on more than
$k\min(|V_1^{j_1}|,\ldots,|V_k^{j_k}|)=k\alpha n$ vertices in $G$. This shows
that Theorem~\ref{thm:PartiteCycle} part~\ref{thm:PartiteCycle:small} is
asymptotically best possible.

\section{Regular Slices} \label{sec:slice} In this section we state our Regular Slice Lemma. We first need to introduce the required notation.
Much of this notation is standard, and will also be needed in Section~\ref{sec:reg} for the Strong Hypergraph Regularity Lemma. We will then state a Cycle Embedding Lemma, which allows us to use the nice properties of regular slices to find tight cycles in a hypergraph satisfying the appropriate conditions.

\subsection{Multipartite hypergraphs and regular complexes}\label{subsec:regdef}

Let $\Part$ partition a vertex set $V$ into parts $V_1, \dots, V_s$. Then we say
that a subset $S \subseteq V$ is \emph{$\Part$-partite} if $|S \cap V_i| \leq 1$ for
every $i \in [s]$. Similarly, we say that a hypergraph $\Hy$ is \emph{$\Part$-partite} if
all of its edges are $\Part$-partite. In this case we refer to the parts of
$\Part$ as the \emph{vertex classes} of $\Hy$. As defined previously for $k$-graphs, we say that 
a hypergraph $\Hy$ is \emph{$s$-partite} if there is some partition $\Part$ of $V(\Hy)$ into $s$ parts for which $\Hy$ is $\Part$-partite.

Let $\Hy$ be a $\Part$-partite hypergraph.
Then for any $A \subseteq [s]$ we write $V_A$ for $\bigcup_{i \in A} V_i$. The
\emph{index} of a $\Part$-partite set $S \subseteq V$ is $i(S) := \{i \in [s] : |S \cap
V_i| = 1\}$. We write $\Hy_A$ to denote the collection of edges in $\Hy$ with
index $A$. So $\Hy_A$ can be regarded as an $|A|$-partite $|A|$-graph on vertex
set $V_A$, with vertex classes $V_i$ for $i \in A$. It is often convenient to refer to the subgraph induced by a set of vertex classes rather than with a given index; if $X$ is a $k$-set of vertex classes of $\Hy$ we write $\Hy_X$ for the $k$-partite $k$-uniform subgraph of $\Hy^{(k)}$ induced by $\bigcup X$, whose vertex classes are the members of~$X$. Note that $\Hy_X = \Hy_{\{i : V_i \in X\}}$. In a similar manner we write $\Hy_{X^<}$ for the $k$-partite hypergraph on vertex set $\bigcup X$ whose edge set is $\bigcup_{X' \subsetneq X} \Hy_X$. Note that if $\Hy$ is a complex, then $\Hy_{X^<}$ is a $(k-1)$-complex because $X$ is a $k$-set.

Let $i \geq 2$, let $H_i$ be any $i$-partite $i$-graph, and let $H_{i-1}$ be any $i$-partite
$(i-1)$-graph, on a common vertex set $V$ partitioned into $i$ common vertex classes. We denote by
$K_i(H_{i-1})$ the $i$-partite $i$-graph on $V$ whose edges are all
$i$-sets in $V$ which are supported on $H_{i-1}$ (i.e.~induce a copy of the complete $(i-1)$-graph~$K_i^{i-1}$
on~$i$ vertices in~$H_{i-1}$). The \emph{density of $H_i$ with respect to
$H_{i-1}$} is then defined to be \[ d(H_i|H_{i-1}):= \frac{|K_i(H_{i-1})\cap
H_i|}{|K_i(H_{i-1})|}\]
 if
$|K_i(H_{i-1})|>0$. For convenience we take
$d(H_i|H_{i-1}):=0$ if $|K_i(H_{i-1})| = 0$. So $d(H_i|H_{i-1})$ is the proportion of copies of
$K^{i-1}_i$ in $H_{i-1}$ which are also edges of $H_i$. When $H_{i-1}$ is clear from the
context,
we simply refer to $d(H_i | H_{i-1})$ as the \emph{relative density of $H_i$}.
More generally, if $\Qb:=(Q_1,Q_2,\ldots,Q_r)$ is a collection
of~$r$ not necessarily disjoint subgraphs of~$H_{i-1}$, we define $K_i(\Qb):=\bigcup_{j=1}^r K_i(Q_j)$
and \[d(H_i |\Qb):= \frac{|K_i(\Qb)\cap H_i|}{|K_i(\Qb)|}\] if
$|K_i(\Qb)|>0$. Similarly as before we take $d(H_i |\Qb):=0$ if
$|K_i(\Qb)| =0$. 
We say that $H_i$ is \emph{$(d_i,\eps,r)$-regular with respect
to~$H_{i-1}$} if we have $d(H_i|\Qb) = d_i \pm \eps $ for every $r$-set $\Qb$
of subgraphs of $H_{i-1}$ such that $|K_i(\Qb)| > \eps |K_i(H_{i-1})|$. We
often refer to $(d_i,\eps,1)$-regularity simply as
\emph{$(d_i,\eps)$-regularity}; also, we say simply that $H_i$ is \emph{$(\eps,r)$-regular with respect
to~$H_{i-1}$} to mean that there exists some $d_i$ for which $H_i$ is \emph{$(d_i,\eps,r)$-regular with respect
to~$H_{i-1}$}. Finally, given an $i$-graph $G$
whose vertex set contains that of $H_{i-1}$, we say that $G$ is
\emph{$(d_i,\eps,r)$-regular with respect to~$H_{i-1}$} if the $i$-partite subgraph of
$G$ induced by the vertex classes of $H_{i-1}$ is $(d_i,\eps,r)$-regular with respect to
$H_{i-1}$. Similarly as before, when $H_{i-1}$ is clear from the context, we refer to the relative density of this $i$-partite subgraph of $G$ with respect to $H_{i-1}$ as the \emph{relative density of~$G$}.

Now let $\Hy$ be an $s$-partite $k$-complex on
vertex classes $V_1, \dots, V_s$, where $s \geq k \geq 3$. Recall that this
means that if $e \in \Hy$ and $e' \subseteq e$ then $e' \in \Hy$. So if $e \in
\Hy^{(i)}$ for some $2 \leq i \leq k$, then the vertices of $e$ induce a copy of
$K^{i-1}_i$ in $\Hy^{(i-1)}$. This means that for any index $A \in \binom{[s]}{i}$ the
density $d(\Hy^{(i)}[V_A]|\Hy^{(i-1)}[V_A])$ can be regarded as the proportion of
`possible edges' of $\Hy^{(i)}[V_A]$ which are indeed edges. (Here a `possible edge' is a
subset of $V(\Hy)$ of index $A$ all of whose proper subsets are edges of~$\Hy$).
We therefore say that~$\Hy$ is \emph{$(d_k,\dots,d_2,\eps_k,\eps,r)$-regular} if
\begin{enumerate}[label=\abc]
 \item for any $2 \leq i \leq k-1$ and any $A \in
  \binom{[s]}{i}$, the induced subgraph $\Hy^{(i)}[V_A]$ is
  $(d_i,\eps)$-regular with respect to~$\Hy^{(i-1)}[V_A]$, and
 \item for any $A \in \binom{[s]}{k}$, the induced subgraph 
  $\Hy^{(k)}[V_A]$ is $(d_k, \eps_k, r)$-regular with respect to~$\Hy^{(k-1)}[V_A]$.
\end{enumerate}
So each constant $d_i$ approximates the relative density of each subgraph $\Hy^{(i)}[V_A]$ for $A \in \binom{[s]}{i}$ for which $\Hy^{(i)}[V_A]$ is non-empty.
For a $(k-1)$-tuple $\mathbf{d} = (d_k, \dots, d_2)$ we write
$(\mathbf{d},\eps_k,\eps,r)$-regular to mean
$(d_k,\dots,d_2,\eps_k,\eps,r)$-regular. 
A regular complex is the
correct notion of `approximately random' for hypergraph regularity.

\subsection{Regular slices and the reduced $k$-graph}\label{sec:rhgraph}
The Regular Slice Lemma says that any $k$-graph $G$ admits a
`regular slice'. This is a multipartite $(k-1)$-complex $\cJ$ whose vertex
classes have equal size, which is regular, and which moreover has the property that $G$ is regular with
respect to $\cJ$. The first two of these conditions are formalised in the
following definition: we say that a $(k-1)$-complex $\cJ$ is
\emph{$(t_0,t_1,\eps)$-equitable} if it has the following properties.
\begin{enumerate}[label=\abc]
  \item $\cJ$ is $\Part$-partite for some $\Part$ which partitions $V(\cJ)$ into
  $t$ parts, where $t_0\le t\le t_1$, of equal size. We
  refer to $\Part$ as the \emph{ground partition} of $\cJ$, and to the parts of $\Part$ as the
  \emph{clusters} of $\cJ$.
  \item There exists a \emph{density vector} $\mathbf{d}=(d_{k-1},\ldots,d_2)$
  such that for each $2\le i\le k-1$ we have $d_i\ge 1/t_1$ and
  $1/d_i\in\NATS$, and the $(k-1)$-complex $\cJ$ is $(\mathbf{d},\eps,\eps,1)$-regular.
\end{enumerate}
For any $k$-set $X$ of clusters of $\cJ$, we write $\hat{\cJ}_X$ for the
$k$-partite $(k-1)$-graph $\cJ^{(k-1)}_{X^<}$.
For reasons that will become apparent later, we sometimes refer to $\hat{\cJ_X}$ as a
\emph{polyad}.
The following well-known fact, a special case of the Dense Counting Lemma (see
\cite[Theorem 6.5]{KRS}), tells us approximately the number of edges in each
layer of $\cJ$, and also how many $k$-sets are supported on any polyad. Note that the definition of an equitable complex implies that $\eps\ll 1/t_1\le d_i$ in this fact.

\begin{fact} \label{fact:densecount} 
Suppose that $1/m_0 \ll \eps \ll 1/t_1, 1/t_0, \beta, 1/k$, and that
$\cJ$ is a $(t_0, t_1, \eps)$-equitable $(k-1)$-complex with density
vector $(d_{k-1}, \dots, d_2)$ whose clusters each have size $m \geq m_0$. Let $X$ be a set of
$k$ clusters of~$\cJ$. Then $$|K_k(\hat{\cJ}_X)| = (1 \pm \beta) m^k \prod_{i =
2}^{k-1} d_i^{\binom{k}{i}},$$ and for any proper subset $X' \subsetneq X$ we
have $$|\cJ_{X'}| = (1 \pm \beta) m^{|X'|} \prod_{i = 2}^{|X'|}
d_i^{\binom{|X'|}{i}}.$$ \qed
\end{fact}

Given a $(t_0,t_1,\eps)$-equitable $(k-1)$-complex $\cJ$ and a
$k$-graph $G$ on $V(\cJ)$, we say that \emph{$G$ is $(\eps_k,r)$-regular with respect to a $k$-set $X$ of clusters of $\cJ$} if there exists some $d$ such
that $G$ is $(d,\eps_k,r)$-regular with respect to the polyad $\hat{\cJ_X}$. We
also write $\reld_\cJ(X)$ for the relative density of $G$ with respect to
$\hat{\cJ_X}$, or simply $\reld(X)$ if $\cJ$ is clear from the context, which will usually be the case in applications. 
Note in particular that Fact~\ref{fact:densecount} then implies
that the number of edges of $G$ which are supported on $\hat{\cJ_X}$ is
approximately $\reld(X) m^k \prod_{i = 2}^{k-1} d_i^{\binom{k}{i}}$. We now give
the key definition of the Regular Slice Lemma.

\begin{definition}[Regular slice] Given $\eps,\eps_k>0$,
$r,t_0,t_1\in\NATS$, a $k$-graph $G$ and a $(k-1)$-complex $\cJ$ on $V(G)$, we
call $\cJ$ a $(t_0,t_1,\eps,\eps_k,r)$-regular slice for $G$ if
$\cJ$ is $(t_0,t_1,\eps)$-equitable and $G$ is $(\eps_k,r)$-regular with respect to all but at most
$\eps_k\binom{t}{k}$ of the $k$-sets of clusters of $\cJ$, where $t$ is the number of clusters of $\cJ$.
\end{definition}

  It will sometimes be convenient not to specify all of the parameters: we may
  write that $\cJ$ is $(\cdot, \cdot, \eps)$-equitable, or is a $(\cdot,\cdot,\eps,\eps_k,r)$-slice for 
$G$, if we do not wish to specify $t_0$ or $t_1$ (as will be the
  case if we specify instead the density vector $\mathbf{d}$ and the number of clusters $t$).

Given a regular slice $\cJ$ for a $k$-graph $G$, it will be important to know the
relative densities $\reld(X)$ for $k$-sets $X$ of clusters of $\cJ$. To keep track of these we make the following definition.

\begin{definition}[Weighted reduced $k$-graph]
 Given a $k$-graph $G$ and a $(t_0, t_1, \eps)$-equitable
 $(k-1)$-complex $\cJ$ on $V(G)$, we let $R_\cJ(G)$ be the complete weighted $k$-graph whose
 vertices are the clusters of $\cJ$, and where each edge $X$ is given weight $\reld(X)$
 (in particular, the weight is in $[0,1]$). When $\cJ$ is clear from the context we often simply write $R(G)$ instead of $R_\cJ(G)$. 
\end{definition}

In applications we will usually take $\cJ$ to be a regular slice for $G$. In
this case we will want to ensure that $G$ is regular with respect to all $k$-sets $X$ of clusters of $\cJ$
with $\reld(X) > 0$, which can be achieved by
the simple expedient of deleting all edges of $G$ lying in $k$-sets with respect to which $G$ is not regular; the definition of a regular slice implies that
there are few such edges. The reason that we do not specify this in the definition
of a regular slice, or specify that only $k$-sets with respect to which $G$ is regular are given
positive weight in the weighted reduced graph, is that we will also make use of this definition in the process
of proving that a certain equitable $(k-1)$-complex is in fact a regular slice.
Note that although different choices of $\cJ$ may well produce different
weighted reduced $k$-graphs (and only some of these will have `good'
properties), since $\cJ$ will always be clear from the context we will speak of
`the' reduced $k$-graph $R(G)$ of $G$.

In general, it is not very helpful to know that $\cJ$ is a regular slice for a
$k$-graph $G$;
the reduced graph of $G$ with respect to $\cJ$ does not necessarily resemble $G$ in the way that the reduced $2$-graph of a $2$-graph $H$ with respect to a Szemer\'edi partition resembles $H$. The Regular Slice Lemma states that there is a regular slice $\cJ$ with respect to which $R(G)$ does resemble $G$, in the sense that densities of small subgraphs (part~\ref{simpreg:a} of the Regular Slice Lemma) and degree conditions (part~\ref{simpreg:b}) are preserved. Furthermore, all vertices of $G$ are in some sense `represented' on $\cJ$ (part~\ref{simpreg:c})---this is useful for embedding spanning subgraphs.
In order to make this precise, we need the following further definitions.

Given hypergraphs $G$ and $H$, we write $n_H(G)$ for the number of labelled
copies of $H$ in $G$.
In order to extend this definition to weighted $k$-graphs
$G$, we let \[n_H(G):=\sum_{\phi:V(H)\to V(G)}\prod_{e\in
E(H)}\reld\big(\phi(e)\big)\] where $\phi$ ranges over all injective maps and
$\reld$ is the weight function on $E(G)$. We then define the \emph{$H$-density
of $G$} as \[d_H(G):=\frac{n_H(G)}{\binom{v(G)}{v(H)}\cdot v(H)!}\,. \] In
other words, $d_H(G)$ is the expectation of $\prod_{e\in E(H)}\reld (\phi(e))$
for an injective map $\phi:V(H)\to V(G)$ chosen uniformly at random. In the
special case that $H$ consists of a single edge, we will write simply $d(G)$,
and speak of the \emph{density of $G$} or the \emph{edge-density of $G$}.

Now let $G$ be a $k$-graph on $n$ vertices. Given a set $S \subseteq V(G)$ of size $j$
for some $1 \le j\le k-1$, and a subset $X \subseteq V(G)$, the \emph{relative degree $\reldeg(S;G, X)$ of $S$ in $X$
with respect to $G$} is defined to be
\[\reldeg(S;G, X):=\frac{\big|\{e\in G[S \cup X] :S\subset
e\}\big|}{\binom{|X \setminus S|}{k-j}}\,.\]
In other words, $\reldeg(S;G, X)$ is the proportion of $k$-sets of vertices of $G[X \cup S]$ extending $S$ which are in fact edges of $G$.
To extend this definition to weighted $k$-graphs $G$, we replace the
number of edges in $G[X \cup S]$ including $S$ with the sum of the weights of edges of $G[X \cup S]$
including $S$. Finally, if $\cS$ is a collection of $j$-sets in $V(G)$,
then $\reldeg(\cS; G, X)$ is defined to be the mean of $\reldeg(S;G, X)$ over
all sets $S\in\cS$.

Given a $k$-graph $G$ and distinct `root' vertices
$v_1,\ldots,v_\ell$ of $G$, and a $k$-graph $H$ equipped with a set of distinct
`root' vertices $x_1,\ldots,x_\ell$, we define the \emph{number of labelled rooted copies of $H$ in $G$}, written
$n_H(G;v_1,\ldots,v_\ell)$, to be the number of injective maps from $V(H)$ to
$V(G)$ which embed $H$ in $G$ and take $x_j$ to $v_j$ for each $1\le j\le\ell$. Then the \emph{density of
rooted copies of $H$ in $G$} is defined to be
\[d_H(G;v_1,\ldots,v_\ell):=\frac{n_H(G;v_1,\ldots,v_\ell)}{\tbinom{v(G)-\ell}{v(H)-\ell}\cdot\big(v(H)-\ell\big)!}\,.\]
This density has a natural probabilistic interpretation: choose uniformly at random an injective map $\psi : V(H) \to V(G)$ such that $\psi(x_j) = v_j$ for every $j \in [\ell]$. Then $d_H(G;v_1, \dots, v_\ell)$ is the probability that $\psi$ embeds $H$ in $G$. 
Next, we define $\Hskel$ to be the $(k-1)$-complex on $V(H) - \ell$ vertices which is
obtained from the complex $\Hy$ generated by the down-closure of $H$ by deleting the vertices
$x_1,\ldots,x_\ell$ (and all edges containing them) and deleting all
edges of size $k$. Given a $(t_0, t_1, \eps)$-equitable
$(k-1)$-complex $\cJ$ on $V(G)$, the \emph{number of rooted copies of $H$ supported by~$\cJ$}, written $n_H(G;v_1,\ldots,v_\ell,\cJ)$, is defined to be the number of labelled
rooted copies of $H$ in $G$ such that each vertex of $\Hskel$ lies in a distinct cluster of~$\cJ$ and the image of $\Hskel$ is in $\cJ$ (but note that we do \emph{not} require the edges involving
$v_1,\ldots,v_\ell$ to be contained in or supported by $\cJ$, and indeed
typically they will not be). We also define $n'_{\Hskel}(\cJ)$ to be the number of labelled copies of $\Hskel$ in $\cJ$ with each vertex of $\Hskel$ embedded in a distinct cluster of $\cJ$. Then the
\emph{density $d_H(G;v_1,\ldots,v_\ell,\cJ)$ of rooted copies of $H$ in $G$ supported by $\cJ$} is then defined by
$$d_H(G;v_1,\ldots,v_\ell,\cJ):= \frac{n_H(G;v_1,\ldots,v_\ell,\cJ)}{n'_{\Hskel}(\cJ)}.$$
Again we have a natural probabilistic interpretation: let $\psi : V(\Hskel) \to V(G)$ be an injective map chosen uniformly at random, and extend $\psi$ to a map $\psi' : V(H) \to V(G)$ by taking $\psi'(x_i) = v_i$ for each $i \in [\ell]$. Then $d_H(G;v_1,\ldots,v_\ell,\cJ)$ is the conditional probability that~$\psi'$ embeds $H$ in~$G$ given that~$\psi$ embeds $\Hskel$ in~$\cJ$ with each vertex of~$\Hskel$ embedded in a different cluster of~$\cJ$.

We can now state the Regular Slice Lemma. We remark that for many
applications it suffices to take $q = s =1$, with $\Qart$ being the
trivial partition of $V$ with one part.

\begin{lemma}[Regular Slice Lemma]\label{lem:simpreg}
Let $k \geq 3$ be a fixed integer. For all positive integers $q$, $t_0$ and $s$,
positive~$\eps_k$ and all functions $r: \NATS \rightarrow \NATS$ 
and $\eps: \NATS \rightarrow (0,1]$, there are integers~$t_1$ 
and~$n_0$ such that the following holds for all $n \ge n_0$ which are divisible 
by~$t_1!$. Let $V$ be a set of $n$ vertices, and suppose that~$G_1, \dots,
G_s$ are edge-disjoint $k$-graphs on $V$, and that $\Qart$ is a partition of $V$
into at most $q$ parts of equal size. Then there exists a $(k-1)$-complex
$\cJ$ on $V$ which is a
$(t_0,t_1,\eps(t_1),\eps_k,r(t_1))$-regular slice for each
$G_i$, such that the ground partition $\Part$ of $\cJ$ refines $\Qart$, and such
that $\cJ$ has the following additional properties.
\begin{enumerate}[label=\abc]
  \item\label{simpreg:a} For each $1 \leq i \leq s$, any $k$-graph $H$ with $v(H)\le 1/\eps_k$ and each
  set $X$ of at least $\eps_k t$ clusters of $\cJ$ (where $t$ is the total number of clusters of $\cJ$), we have
  \[\left|d_H\left(R(G_i)[X]\right)-d_H\left(G_i\left[\bigcup X\right]\right)\right|<\eps_k\,.\]
  \item\label{simpreg:b} For each $1 \leq i \leq s$, each $1 \le j \le k-1$,
  each set $Y$ of $j$ clusters of $\cJ$, and each set $X$ of clusters of $\cJ$
  for which $\bigcup X$ is the union of some parts of $\Qart$, we have
  \[\big|\reldeg(Y;R(G_i), X)-\reldeg(\cJ_Y;G_i, \bigcup X)\big|<\eps_k\,.\]
  \item\label{simpreg:c} For each $1\leq i\leq s$, each $1\le\ell\le 1/\eps_k$,
  each $k$-graph $H$ equipped with a set of distinct root vertices
  $x_1,\ldots,x_\ell$ such that $v(H)\le 1/\eps_k$, and any distinct vertices
  $v_1,\ldots,v_\ell$ in $V$, we have
  \[\big|d_H(G_i;v_1,\ldots,v_\ell,\cJ)-d_H(G_i;v_1,\ldots,v_\ell) \big|<\eps_k
  \,.\]
\end{enumerate}
\end{lemma}

To understand this lemma consider first the special case
of regularising only one graph $G=G_1$, i.e.~$s=1$, with $\Qart$ being the trivial partition with one
part, i.e.~$q=1$, and where (for properties~\ref{simpreg:a}
and~\ref{simpreg:c}) $H$ consists of a single edge (in the latter case, rooted at one vertex $v$). Then the result
is that $\cJ$ is a regular slice for $G$, with not too few clusters (bounded below by
$t_0$) but also not too many (bounded above by $t_1$), with very
strong regularity properties (typically $\eps$ is much smaller than any density
$d_i$) and $G$ is $(\eps_k,r)$-regular with respect to $\cJ$. This looks so far
very much like the `usual' hypergraph regularity (see Lemma~\ref{reglemma} later
for a statement) except that the statement that $G$ is regular with respect to
$\cJ$ says something only about a very small part of $G$. Then
property~\ref{simpreg:a} states that $R(G)$ `looks like' $G$ in that
edge densities on large sets agree. Property~\ref{simpreg:b} tells us that degree conditions on $G$ transfer to $R(G)$. Finally
property~\ref{simpreg:c} says: for every vertex $v$ of $G$, the fraction
of $(k-1)$-sets of $V(G)$ which make edges with $v$ is close to the fraction
of $(k-1)$-edges of $\cJ$ which make edges with $v$. That is, vertex degrees in $G$ are inherited when we consider only extensions supported on $\cJ$.

We note that it is often necessary to estimate (for example) the number of edges
of $G$ which have one vertex in each of some pairwise-disjoint sets $X_1,\ldots,X_k$ rather than
just the number of edges in a set $X$: provided that each of the $X_i$ is large,
it follows by use of the inclusion-exclusion principle and the ability conferred
by~\ref{simpreg:a} to estimate the number of edges in a large set $X$ that we
can also make such estimates.

For embedding subgraphs, it is important to distinguish dense regular $k$-sets (i.e. $k$-sets on which $G$ is dense and regular with respect to $\cJ$). In order to state our Cycle Embedding Lemma we therefore give the following variation of the definition of the reduced graph.

\begin{definition}[The $d$-reduced $k$-graph] \label{def:redgraph}
Let $G$ be a $k$-graph and let $\cJ$ be a
$(t_0,t_1,\eps,\eps_k,r)$-regular slice for $G$. Then for $d>0$ we
define the $d$-reduced $k$-graph $R_d(G)$ to be the $k$-graph whose vertices
are the clusters of $\cJ$ and whose edges are all $k$-sets $X$ of clusters of
$\cJ$ such that $G$ is $(\eps_k, r)$-regular with respect to $X$ and $\reld(X) \geq d$.
\end{definition}

The next lemma states
that for regular slices $\cJ$ from Lemma~\ref{lem:simpreg}, $H$-densities and degrees are also preserved by $R_d(G)$, allowing us to work with this
structure also.

\begin{lemma} \label{reducedd+d} 
Let $G$ be a $k$-graph and let $\cJ$ be a
$(t_0,t_1,\eps,\eps_k,r)$-regular slice for~$G$ with $t$
clusters. Also let $X$ be a set of clusters of $\cJ$. Then for any $k$-graph $H$ we have
\begin{equation*} \label{eq:reduceddensity} 
d_H\big(R_d(G)[X]\big)\ge d_H\big(R(G)[X]\big) - d - \frac{\eps_k e(H)
\binom{t}{k}}{\binom{|X|}{k}},
\end{equation*} 
and for any set $Y$ of at most $k-1$ clusters of $\cJ$ we have
\begin{equation*} \label{eq:reduceddegree} 
\reldeg(Y;R_{d}(G), X) \ge \reldeg(Y;R(G), X) - d - \zeta(Y),
\end{equation*}
where $\zeta(Y)$ is defined to be the proportion of $k$-sets of clusters $Z$
with $Y \subseteq Z \subseteq Y \cup X$ which are not $(\eps_k, r)$-regular with respect to $G$.
\end{lemma}

\begin{proof} 
Observe that we can transform $R(G)$ to $R_d(G)$ by editing the edge-weights of
$R(G)$ in three stages. First, for any edge $S$ of $R(G)$ with $\reld(S) \geq
d$, we increase the weight of $S$ from $\reld(S)$ to 1. Second, for any edge
$S$ of $R(G)$ with $\reld(S) < d$, we decrease the weight of $S$ from
$\reld(S)$ to zero. Finally, for any edge $S$ of $R(G)$ such that $G$ is not
$(\eps_k, r)$-regular with respect to $S$, we reduce the weight of $S$ to
zero. Note that the number of $k$-sets $S$ of the latter type is at most $\eps_k \binom{t}{k}$ since $\cJ$ is a $(t_0,t_1,\eps,\eps_k,r)$-regular slice
for~$G$.

To prove the first equation, we consider the effect of each of these changes on the quantity $d_H\big(R(G)[X]\big)$. 
Recall that this was defined to be the average of 
$\prod_{e \in H} \reld(\phi(e))$ taken over all the 
$\binom{|X|}{h}h!$ injections 
$\phi : V(H) \to X$, where $h := |V(H)|$. 
If the weight of any edge of $R(G)$ is increased, this average cannot decrease. 
Likewise, if the weights of any subset of edges of $R(G)$ are decreased from at most $d$ to zero, 
then this average will decrease by at most $d$ also. Finally, if the
weights of some set of $m$ chosen edges of $R(G)$ are reduced to zero, then
this reduces $\prod_{e \in H} \reld(\phi(e))$ by at most one for each of the at
most $e(H) m \binom{|X|-k}{h-k} (h-k)!k!$ injections $\phi$ for which $\phi(e)$
is a chosen edge for some $e \in H$. So $d_H\big(R(G)[X]\big)$ decreases by at
most $$\frac{e(H) m \binom{|X|-k}{h-k} (h-k)!k!}{\binom{|X|}{h}h!} = \frac{e(H)
m }{\binom{|X|}{k}}.$$ Combining these changes for the three steps described
above, we obtain the first equation.

Similarly, for the second equation we consider the effect of each change on the
quantity $\reldeg(Y;R(G), X)$, which was defined to be the average of
$\reld(Z)$ over all $k$-sets $Z$ with $Y \subseteq Z \subseteq Y \cup X$. As
before, increasing the weight of any edge of $R(G)$ cannot cause this average
to decrease, and decreasing the weights of any subset of edges of $R(G)$ by at
most $d$ will cause this average to decrease by at most $d$. Finally, since
$\reld(Z)$ must be between zero and one for any $Z$, reducing the weight of a
$\zeta$-proportion of the $k$-sets $Z$ to zero will reduce this average by at
most $\zeta$; combining these changes for the three steps gives the second equation.
\end{proof}

\subsection{Cycle Embedding Lemma}\label{sec:emb}

A standard and very useful result (originally proved by
{\L}uczak~\cite{LucRamsey}) in extremal graph theory is that for all $d>0$ and
sufficiently small $\eps>0$, the following holds. Given a graph $G$ and a partition of $V(G)$ into clusters of equal size, let $R_d(G)$ be the graph whose vertices correspond to clusters and where an edge indicates that the bipartite graph induced by the corresponding clusters is $\eps$-regular with density at least $d$. If there
is a connected component of $R_d(G)$ which contains a matching covering at least
$(\alpha+d)v(R_d(G))$ of the vertices of $R_d(G)$, then there are paths and even
cycles in $G$ of each length up to $\alpha v(G)$. In fact, a stronger result is
true, as observed by Hladk\'y, Kr\'a\soft{l} and Piguet (see~\cite{FracMatch}, where
this idea was also used): we need
only a fractional matching with weight $(\alpha+d)v(R_d(G))$. In this
subsection we state the corresponding result for tight paths and cycles in
$k$-graphs.

Let $G$ be an $n$-vertex $k$-graph. Then 
a \emph{matching} in~$G$ is a
set $M \subset E(G)$ of vertex-disjoint edges of~$G$, and the \emph{matching
  number}~$\nu(G)$ denotes the maximum size of a matching in~$G$. A \emph{fractional matching} $M$ in $G$ assigns a weight $w_e \in [0, 1]$ to each edge $e \in G$ so that for every vertex $v \in V(G)$ we have $\sum_{e :~v \in e} w_e \leq 1$. The \emph{weight} of $M$ is the sum of all the edge weights, which must lie between zero and $n/k$. We say that $M$ is $\emph{perfect}$ if it has weight $n/k$.

  Next, we define a \emph{tight walk} $W$ in $G$ to be a sequence of vertices of $G$ such that each set of $k$ consecutive vertices induces an edge of $G$.
For edges $e, f \in G$ we say that $W$ is a walk from $e$ to $f$ if $W$ begins with 
the vertices of $e$ (in some order) and concludes with the vertices of  
the vertices of $f$ (in some order). If such a walk exists then we say that $e$ and $f$ are 
\emph{tightly connected}. This gives an equivalence relation on the
edges of~$G$. To see this, observe that if $e$ and $f$ are edges of $G$ with 
$|e \cap f| = k-1$, then given any $(k-1)$-tuple in $e$ and
 $(k-1)$-set in $f$, there is a tight walk in $G$ whose first $k-1$ vertices are
 the chosen $(k-1)$-tuple in $e$ (in order) and whose last $k-1$ vertices are
 the chosen $(k-1)$-set in $f$ (in some order which we cannot choose). Applying this observation repeatedly establishes the transitivity of the `tightly connected' relation, and moreover shows that the observation still holds if we replace the assumption $|e \cap f| = k-1$ by the weaker assumption that $e$ and $f$ are tightly connected in $G$. 
A \emph{tight component} of~$G$ is an equivalence class of this relation, that is, an edge maximal set
$C \subset G$ such that each pair $e, f$ of edges in~$C$ are tightly
connected (recall we identify this edge set with the subgraph of~$G$ with vertex set $\bigcup C$ and edge set~$C$). A \emph{tightly connected matching} in $G$ is a matching in which all edges are tightly connected (that is, they all lie in the same tight component of $G$). Finally, a \emph{tightly connected fractional matching} is a fractional matching in which the same is true of all edges of non-zero weight.

We can now state our main embedding result, which we prove in Section~\ref{sec:embedding}.

\begin{lemma}[Cycle Embedding Lemma]\label{lem:emb}
  Let $k,r,n_0,t$ be positive integers, and $\psi,d_2,\ldots,d_k,\eps,\eps_k$ be
  positive constants such that $1/d_i\in\NATS$ for each $2\le i\le k-1$, and such
  that $1/n_0 \ll 1/t$,
\[\frac{1}{n_0} \ll
  \frac{1}{r},\eps\ll\eps_k,d_2,\ldots,d_{k-1}\quad\text{ and }\quad\eps_k\ll
  \psi,d_k,\frac{1}{k}\,.\] Then the following holds for all integers $n\ge
  n_0$. Let~$G$ be a $k$-graph on~$n$ vertices, and $\cJ$ be a $(\cdot,\cdot,\eps,\eps_k,r)$-regular slice for $G$ with $t$ clusters and density vector $(d_{k-1},\ldots,d_{2})$. Suppose that
  $R_{d_k}(G)$ contains a tightly connected fractional matching with weight~$\mu$. Then $G$ contains a tight cycle of length~$\ell$ for every
      $ \ell\le(1-\psi)k\mu n/t$ that is divisible by~$k$.
\end{lemma}

\section{Erd\H{o}s-Gallai for hypergraphs}\label{sec:EG} 
In this section we aim to demonstrate the value of the tools presented
in the previous section by proving Theorem~\ref{thm:HypErdGal}.

Our strategy is simple: by Lemmas~\ref{lem:simpreg} and~\ref{lem:emb} it is
enough to show that any $k$-graph $G$ with edge density $\alpha$ contains a
tightly connected matching with at least $\alpha v(G)/k$ edges. To obtain this, we show (Proposition~\ref{prop:comp}) that any $G$ with edge density $\alpha$ contains a tight component $G^*$ with $e_k(\cG^*)\ge\alpha e_{k-1}(\cG^*)$, and then that any $k$-graph which satisfies this inequality contains the desired matching (Lemma~\ref{lem:ratiomatching}).

We first justify our assertion that there is a dense tight component.

\begin{proposition}\label{prop:comp}
Let $G$ be a $k$-graph on $n$ vertices. Then there is a tight component $G^*$ of $G$ such that  
  \begin{equation*}
    e_k(\cG^*)\ge\frac{e_{k-1}(\cG^*)}{\binom{n}{k-1}} e(G),
  \end{equation*}
  where $\cG^*$ denotes the $k$-complex generated by the down-closure of $G^*$.
\end{proposition}
\begin{proof}
  Let $G_1,\dots,G_s$ be the tight components of~$G$, and let $\cG_i$ denote 
 the $k$-complex generated by the down-closure of $G_i$. Fix $\ell \in [s]$ 
  which maximises $e_k(\cG_\ell)/e_{k-1}(\cG_\ell)$. 
  By the definition of tight components we have
  $\sum_{i\in[s]}e_{k-1}(\cG_i)\le\binom{n}{k-1}$. Hence
  \begin{equation*}
      e(G)= \sum_{i\in[s]} e_k(\cG_i) 
      \le \frac{e_k(\cG_\ell)}{e_{k-1}(\cG_\ell)}\sum_{i\in[s]} e_{k-1}(\cG_i) 
      \le \frac{e_k(\cG_\ell)}{e_{k-1}(\cG_\ell)} \binom{n}{k-1} \,.
  \end{equation*} 
\end{proof}

We next state the lemma guaranteeing that a dense $k$-graph has a large matching.

\begin{lemma}\label{lem:ratiomatching}
  Let $k,r$ be any natural numbers and let
  $\cG$ be a $k$-complex in which
  \begin{equation}\label{eq:ratio}
   e_k(\cG)\geq (r-1)e_{k-1}(\cG)+1\,.
  \end{equation}
  Then $\nu\big(\cG^{(k)}\big)\ge r$.
\end{lemma}

The proof of this lemma proceeds inductively and uses a classical
concept from extremal set theory called compression.  Let~$\Hy$ be a hypergraph on vertex set $[n]$, and choose $i, j \in [n]$ with $i < j$.  Then the \emph{$ij$-compression}
$S_{ij}$ performs the following operation on $\Hy$. For every edge $e \in E(\Hy)$ such 
that $i \notin e$, $j \in e$ and $\{i\} \cup e \setminus \{j\} \notin E(\Hy)$, delete $e$ from~$E(\Hy)$ 
and replace it by $\{i\} \cup e \setminus \{j\}$. We denote the resulting
hypergraph by $S_{ij}(\Hy)$.  If $S_{ij}(\Hy) = \Hy$ for every $i < j$, then we say that $\Hy$ is \emph{fully-compressed}.

The next proposition sets out various properties of compressions of complexes which we shall use.
\begin{proposition}\label{compress}
Let $\Hy$ be a $k$-complex on vertex set~$[n]$. Then for any $1 \leq i < j \leq n$ we have that
\begin{enumerate}[label=\abc]
 \item\label{compress:a} $S_{ij}(\Hy)$ is a $k$-complex with $e_\ell(S_{ij}(\Hy)) = e_\ell(\Hy)$ for any $\ell$, 
 \item\label{compress:b} $\nu\big(S_{ij}(\Hy^{(k)})\big)\leq \nu(\Hy^{(k)})$, and
 \item\label{compress:c} if $\Hy$ is fully-compressed, then for any edge $e \in \Hy$ such that $j \in e$ and $i \notin e$ we have that $\{i\} \cup e \setminus \{j\} \in \Hy$.
\end{enumerate}
\end{proposition}
\begin{proof}
\ref{compress:a} is immediate from the definition, whilst~\ref{compress:b} follows from Lemma~2.1 in~\cite{HuangLohSud}, or is easy to prove directly. Finally~\ref{compress:c} follows since in this case $S_{ij}(\Hy)=\Hy$.
\end{proof}

The second proposition needed to prove Lemma~\ref{lem:ratiomatching} shows that if $G$ satisfies~\eqref{eq:ratio} and some edge of $\cG$ is contained in few higher-level edges of $\cG$, then removing this edge and all edges containing it from $\cG$ gives a subcomplex which also satisfies~\eqref{eq:ratio}.

  \begin{proposition}\label{deletelowdeg}
Let $k,r$ be any natural numbers and let
$\cG$ be a $k$-complex in which $e_k(\cG)\geq (r-1)e_{k-1}(\cG)+1$.
Fix $0 \leq j \leq k-1$, and suppose that 
$e \in E(\cG^{(j)})$ lies in fewer than $(k-j)r$ edges
of $\cG^{(j+1)}$. Let $\cG'$ be the $k$-complex obtained by
deleting $e$ and any edge of $\cG$ which contains $e$ from $\cG$. 
Then $e_k(\cG')\geq (r-1)e_{k-1}(\cG')+1$.
  \end{proposition}
  
To prove this we use the local LYM inequality, which can be found in e.g.~\cite{bollobas}.

\begin{theorem}[Local LYM inequality]\label{thm:kk}
  Let $\cG$ be an $i$-complex on $n$ vertices. 
 Then $e_{i-1}(\cG) \geq \frac{i}{n-i+1} e_i(\cG)$. \qed
\end{theorem}
  
  \begin{proof}[Proof of Proposition~\ref{deletelowdeg}]
Let $A = \{x \in V(\cG) : e \cup \{x\} \in \cG^{(j+1)}\}$. By assumption we have $|A| \leq (k-j)r-1$. Let $\Hy$ be the $(k-j)$-complex on vertex set $A$ with edge set $\{e' \setminus e : e \subseteq e' \in \cG\}$. Note that when we delete $e$ and all edges containing it from $\cG$, we delete exactly $d := e_{k-j}(\Hy)$ edges from~$\cG^{(k)}$ and exactly $d' := e_{k-j-1}(\Hy)$ edges from $\cG^{(k-1)}$.
By Theorem~\ref{thm:kk},
    \begin{equation*}\begin{split}
      \frac{d'}{d} 
      &= \frac{e_{k-j-1}(\Hy)}{e_{k-j}(\Hy)}
      \geq \frac{k-j}{|A|-k+j+1}
      \geq \frac{k-j}{(k-j)r-1-k+j+1} \\
      &= \frac{k-j}{(k-j)(r-1)}
      = \frac{1}{r-1}\,.
    \end{split}\end{equation*}
    Thus the complex $\cG'$ obtained from our deletions satisfies
    \begin{equation*}\begin{split}
      e_k(\cG')
      &=e_k(\cG)-d 
      \ge e_k(\cG)-(r-1)d' \\
      &\geq (r-1)e_{k-1}(\cG)+1-(r-1)d' 
      =(r-1)e_{k-1}(\cG')+1 \,
    \end{split}\end{equation*}
    where we used the assumption that $e_k(\cG)\geq (r-1)e_{k-1}(\cG)+1$
    for the second inequality.
  \end{proof}

The case $j=0$ of Proposition~\ref{deletelowdeg} immediately gives the following corollary, which we use in the proof of Lemma~\ref{lem:ratiomatching}.

\begin{corollary}\label{bigenough} 
  Let $k,r$ be any natural numbers and let
  $\cG$ be a $k$-complex in which $e_k(\cG)\geq (r-1)e_{k-1}(\cG)+1$. Then
  $|V(\cG)| \ge kr$. \qed
\end{corollary}

We now give the proof of Lemma~\ref{lem:ratiomatching}.

\begin{proof}[Proof of Lemma~\ref{lem:ratiomatching}]
 Let~$\cG$ be a $k$-complex on vertex set $[n]$ in which 
$e_k(\cG) \geq (r-1)e_{k-1}(\cG)+1$. First, we perform repeatedly the following two operations. If for some $\ell\in[k-1]$ there is an edge $e\in\cG^{(\ell)}$ which is contained in fewer than $(k-\ell)r$ edges of $\cG^{(\ell+1)}$, we delete it and all its supersets from $\cG$. If there are $1\le i<j\le n$ such that $S_{ij}(\cG)\neq\cG$ then we replace $\cG$ with $S_{ij}(\cG)$. Eventually we reach a complex $\Hy$ where neither operation is possible.

Observe that Proposition~\ref{compress} part~\ref{compress:a} and Proposition~\ref{deletelowdeg} together tell us that, since we started with a complex~$\cG$ satisfying~\eqref{eq:ratio}, $\Hy$ satisfies~\eqref{eq:ratio} also. Moreover, Proposition~\ref{compress} part~\ref{compress:b} together with the trivial fact that deleting edges from a complex does not increase the matching number of any level of the complex implies that $\nu\big(\Hy^{(k)}\big)\le\nu\big(\cG^{(k)}\big)$. By definition, $\Hy$ is fully-compressed. Now given $1\le\ell\le k-1$, let $e$ be an edge of $\Hy^{(\ell)}$. Because $e$ is in $\Hy$, there are at least $(k-\ell)r$ edges of $\Hy^{(\ell+1)}$ containing $e$, and in particular there is some $j\ge(k-\ell)r$ such that $e\cup\{j\}$ is in $\Hy^{(\ell+1)}$. Now Proposition~\ref{compress} part~\ref{compress:c} states that for any $i<j$ with $i\notin e$ we have $\{i\}\cup e\cup\{j\}\setminus\{j\}\in\Hy^{(\ell+1)}$, and thus we have
 \begin{equation}\label{eq:compressgives}
  \{i\}\cup e\in\Hy^{(\ell+1)}\quad\text{for each}\quad i\in\big[(k-\ell)r\big]\quad\text{such that}\quad i\notin e\,.
 \end{equation}

  We will now show by induction on~$\ell$ that the $\ell$-edges 
  \[e_{\ell,m}:=\big\{\, (k-\ell)r+m, (k-\ell+1)r+m, \ldots, (k-1)r+m \,\big\}\] 
  for $m=1,\ldots,r$ form a matching in $\Hy^{(\ell)}$.
  Note that for any $\ell\in[k]$ the sets
  $e_{\ell,m}$ for $m\in[r]$ are by definition pairwise-disjoint; it remains to show that
  all $e_{\ell,m}$ are edges of $\Hy$.

  The base case $\ell=1$ is trivial, since we are looking only for distinct
  vertices, and the singletons $(k-1)r+1,(k-1)r+2,\ldots,kr$ are in $\Hy$ by
  Corollary~\ref{bigenough}.

  Now for some $1\le\ell\le k-1$ and $m\in[r]$ suppose that $e_{\ell,m}$ is an edge of $\Hy^{(\ell)}$. Let $i=(k-\ell-1)r+m\le(k-\ell)r$. By definition we have $i\notin e_{\ell,m}$, and so by~\eqref{eq:compressgives} we have $e_{\ell+1,m}=\{i\}\cup e_{\ell,m}\in\Hy^{(\ell+1)}$, as desired.
  
  We conclude that the sets $e_{k,m}$ for $m\in[r]$ form a matching in $\Hy^{(k)}$, so $r\le\nu\big(\Hy^{(k)}\big)\le\nu\big(\cG^{(k)}\big)$ as desired.
\end{proof}

We note that Lemma~\ref{lem:ratiomatching} is tight. Indeed, for any $k$ and $r$, let $\K$ be the $k$-complex generated by the down-closure of the complete $k$-graph $K^{(k)}_{kr-1}$ on $kr-1$ vertices. Then
\[e_k(\K) = \binom{kr-1}{k} = (r-1)\binom{kr-1}{k-1} = (r-1)e_{k-1}(\K)\,,\]
and $\nu\big(\K^{(k)}\big)=r-1$.

We can now complete our proof of Theorem~\ref{thm:HypErdGal}, which we restate for the convenience of the reader.

\medskip
\noindent {\bf Theorem~\ref{thm:HypErdGal}.}
\emph{For every positive~$\delta$ and every integer $k \ge 3$, there is an
  integer~$n_*$ such that the following holds for all $\alpha\in[0,1]$.
  If~$G$ is a $k$-uniform hypergraph on $n\ge n_*$ vertices with
  $e(G)\ge(\alpha+\delta)\binom{n}{k}$, then~$G$ contains a tight cycle of length~$\ell$ for every
      $\ell\le\alpha n$ that is divisible by~$k$.}

\begin{proof}
  Given $k \geq 3$ and $\delta>0$, we choose $d_k=\delta/3$, $t_0=24k/\delta$ and $\psi = \delta/8$. We let
  $\eps_k\le\delta/12$ be sufficiently small for Lemma~\ref{lem:emb}. We choose
  functions $\eps(\cdot)$ tending to zero sufficiently rapidly,
  and $r(\cdot)$ growing sufficiently rapidly, so that for any $t \in
  \NATS$ and $d_2, \dots, d_{k-1} \geq 1/t$ we may apply
  Lemma~\ref{lem:emb} with $r(t)$, $\eps(t)$. Obtain
  $t_1$ and $n_0$ by applying Lemma~\ref{lem:simpreg} with inputs
  $t_0, \eps_k, r(\cdot)$ and $\eps(\cdot)$ (taking $q=s=1$); for the rest of
  this proof we write $r$ and $\eps$ for $r(t_1)$ and $\eps(t_1)$
  respectively. Then by Lemma~\ref{lem:simpreg}, for any $k$-graph $G$ on $n$
  vertices, where $n\ge  n_0$ is divisible by $t_1!$, there is a
  $(t_0, t_1,\eps,\eps_k, r)$-regular slice for $G$ such that
  $d\big(R(G)\big)\ge d(G)-\eps_k$. Finally, we choose $n_1 \geq n_0$ sufficiently large for us to apply Lemma~\ref{lem:emb} with
 $t_1$ in place of $t$, $n_1$ in place of $n_0$ and all other constants as above, and also such that
  $t_1!\binom{n_1}{k-1}<\delta\binom{n_1}{k}/12$.
  
  Set $n_* := n_1 + t_1!$, and let $G'$ be a $k$-graph on $n'\ge n_*$
  vertices with $e(G')\ge(\alpha+\delta)\binom{n'}{k}$. Delete at most
  $t_1!-1$ vertices from $G'$ to obtain a $k$-graph $G$ on $n$ vertices,
  where $n$ is divisible by $t_1!$. By choice of $n_1$ we have
  $e(G)\ge(\alpha+11\delta/12)\binom{n}{k}$. So Lemma~\ref{lem:simpreg}
   gives us a $(t_0,t_1, \eps, \eps_k, r)$-regular slice $\cJ$ for
   $G$ with $t$ clusters (where $t_0 \leq t \leq t_1$)
   such that $d\big(R(G)\big)\ge d(G)-\eps_k\ge \alpha+10\delta/12$. Finally by Lemma~\ref{reducedd+d} we have
  $d\big(R_{d_k}(G)\big)\ge d\big(R(G)\big)-d_k-\eps_k\ge \alpha+5\delta/12$.
  
  The total number of edges in $R_{d_k}(G)$ is therefore at least
  $(\alpha+5\delta/12)\binom{t}{k}$. It follows by
  Proposition~\ref{prop:comp} that there is a tight component $R$ of
  $R_{d_k}(G)$ such that the $k$-complex $\cR$ generated by the down-closure
  of $R$ satisfies
  \begin{align*}
    e_k(\cR) \ge
  \frac{e_{k-1}(\cR)}{\binom{t}{k-1}}\left(\alpha+\frac{5\delta}{12}\right)\binom{t}{k}&=
  \frac{t-k+1}{k}\left(\alpha+\frac{5\delta}{12}\right)e_{k-1}(\cR) \\
  &\ge \left(\alpha+\frac{4\delta}{12}\right)t \frac{e_{k-1}(\cR)}{k}
  \end{align*} 
  where the final inequality is guaranteed by $t \ge t_0 = 24k/\delta$.
  
  By Lemma~\ref{lem:ratiomatching} it follows that $R$ contains a matching of at
  least $(\alpha+\delta/4)t/k$ edges. Since this matching is contained in
  $R$ it is tightly connected. Lemma~\ref{lem:emb} then implies that $G$
  contains a tight cycle of length $\ell$ for every $\ell$ at most
  $(1-\psi)k\big((\alpha+\delta/4)t/k\big)n/t\ge\alpha n$ which is
  divisible by $k$ (where the inequality follows from $\psi=\delta/8$).
\end{proof}

\section{Cycles in $k$-partite hypergraphs}\label{sec:partite}
In this section we will prove Theorem~\ref{thm:PartiteCycle}. As in the
previous section, our strategy is to show that the partite codegree condition
implies the existence of a sufficiently large (this time fractional) matching in
a connected component of $R_d(G)$, after which applying Lemma~\ref{lem:emb} will
give the existence of the desired tight paths and cycles in $G$.
Unfortunately, while the minimum codegree of $G$ transfers to $R(G)$ by
Lemma~\ref{lem:simpreg}, as we saw in Lemma~\ref{reducedd+d} it does not
transfer perfectly to $R_d(G)$: some $(k-1)$-sets of clusters may lie in many
irregular $k$-sets of clusters, and so have small relative degree in $R_d(G)$.
To handle this, we follow the approach of Keevash and Mycroft~\cite{KeeMyc},
moving to non-uniform hypergraphs $\cR$ in which an edge of $\cR^{(k-1)}$
corresponds to a $(k-1)$-set of $R_d(G)$ which \emph{does} acquire the
desired codegree from $G$, and the presence of an edge of $\cR$ of size $j <
k-1$ implies that most of the $(j+1)$-supersets of this edge are also edges.
The hypergraph-theoretic core of our proof of Theorem~\ref{thm:PartiteCycle} is
Lemma~\ref{lem:PartConnMatch}; to prove this we use the following special case
of Lemma~7.2 of~\cite{KeeMyc}, which was proved by a straightforward
application of Farkas's Lemma.

\begin{lemma}\label{lem:FracMatch}
Let $\Hy$ be a $k$-partite hypergraph whose parts $X_1, \dots, X_k$ each have size $t$. Suppose that $\emptyset \in \Hy$ and that for any $0 \leq i \leq k-1$ and $j \in [k]$, any edge of $\Hy^{(i)}$ which does not intersect $X_j$ is contained in at least $t - it/k$ edges of $\Hy^{(i+1)}$ which do intersect $X_j$. Then $\Hy^{(k)}$ admits a perfect fractional matching. \qed
\end{lemma}

\begin{lemma}\label{lem:PartConnMatch}
  Given $0\le \alpha<1$ and $\beta>0$, let $\cR$ be a $k$-partite hypergraph
 with vertex classes $X_1, \ldots, X_k$ of size $t$ with the following properties.
\begin{enumerate}[label=\rom]
\item\label{pcm:i} $\emptyset \in \cR$, and $\{v\} \in \cR$ for any $v \in V(\cR) = \bigcup_{i \in [k]} X_i$.
\item\label{pcm:ii} For any $1 \le i \le k-2$ and $j \in [k]$, any edge of $\cR^{(i)}$ which has no vertex in $X_j$ is contained in at least $(1-\beta)t$ edges of $\cR^{(i+1)}$ which
intersect $X_j$.
\item\label{pcm:iii} Each edge of $\cR^{(k-1)}$ is contained in at least 
$(\alpha+(2^k+1)\beta)t$ edges of $\cR^{(k)}$.
\end{enumerate}
Then 
 \begin{enumerate}[label=\abc]
\item\label{pcm:a} $\cR^{(k)}$ contains a tightly connected matching of $\alpha t$ edges, and
\item\label{pcm:b} if $\alpha\ge\frac{1}{2}$ then $\cR^{(k)}$ contains a tightly connected perfect fractional matching.
 \end{enumerate}
\end{lemma}
\begin{proof}
First note that for any $1 \leq i \leq k$, if some edge $e \in \cR^{(i)}$ does not include any edge of $\cR^{(i-1)}$ as a subset, then deleting $e$ from $\cR$ yields a subgraph which also meets the conditions of the lemma. So we may assume that every edge in $\cR^{(i)}$ includes at least one edge of $\cR^{(i-1)}$ as a subset.

The edges $f \in \cR^{(k)}$ for which every subset of $f$ is an edge of $\cR$ are the most useful in the sense that for each subset of $f$ we can apply the degree conditions~\ref{pcm:ii} and~\ref{pcm:iii}. We say that such edges are \emph{excellent}. We will show that each tightly connected component of $\cR^{(k)}$ contains a matching of $\alpha t$ excellent edges, giving~\ref{pcm:a}.
\begin{claim}\label{clm:pcm:conn}
 Each edge of $\cR^{(k)}$ is tightly connected to an excellent edge of $\cR^{(k)}$.
\end{claim}
\begin{claimproof}
Let $e=\{v_1,\ldots,v_k\}$ be an edge of $\cR^{(k)}$ with $v_i\in X_i$ for each $i$. Since every edge in $\cR^{(i)}$ includes at least one edge of $\cR^{(i-1)}$ as a subset, we may
assume without loss of generality that the
subsets $\{v_1,\ldots,v_j\}$ are in $\cR$ for each $1\le j \le k-1$. Choose $w_k \in X_k$ such that $\{v_1,\ldots,v_j,w_k\}$ is in $\cR$ for each
$0\le j\le k-1$. This is possible because $\{v_1,\ldots,v_j\}$ is in $\cR$ for each
$j$, and so at most $(k-2)\beta t+(1-\alpha-(2^k+1)\beta)t<t$ vertices in $X_k$
do not satisfy the given condition.
Now for each $i=k-1,\ldots,1$ in
that order, we choose a vertex $w_i \in X_i$ such that $\{v_1,\ldots,v_j,S\}$ is an edge of $\cR$ for each $0\le j\le i-1$ and each subset $S$ of $\{w_i,\ldots,w_k\}$.
Again, since $2^{k-1}\beta t+(1-\alpha-(2^k+1)\beta)t<t$ this is always possible.
Then the sequence of edges of the form $\{v_1,\ldots,v_i,w_{i+1},\ldots,w_k\}$ for $0\le i\le
k$ is a tight walk between $e$ and $\{w_1,\ldots,w_k\}$ in $\cR$ (since by
construction each of these $k$-sets is in $\cR$) and also by
construction $\{w_1,\ldots,w_k\}$ is an excellent edge.
\end{claimproof}

 We next show that for any small vertex set $S$ and any excellent edge $e$, there is a tight walk from $e$ to another excellent edge outside $S$.

\begin{claim}\label{clm:getedge}
 Given any set $S$ of vertices of $\cR$ such that $|S\cap X_i|<\alpha t$ for each
 $1\le i\le k$, and an excellent edge $e$, there is a tight walk from $e$ to an excellent edge $e'$ of $\cR^{(k)}$ with $e'\cap S=\emptyset$. 
\end{claim}
\begin{claimproof}
Let $e = \{u_1, \dots, u_k\}$, where $u_i \in X_i$ for each $i$.
 For each $i=1,\ldots,k$ in that order, we wish to choose $v_i$ to be a vertex of
 $X_i\setminus S$ such that $\{v_1,\ldots,v_i,u_{i+1},\ldots,u_k\}$ is an
 excellent edge. This means that for each subset $T$ of
 $\{v_1,\ldots,v_{i-1},u_{i+1},\ldots,u_k\}$, we need to guarantee that
 $T\cup\{v_i\}$ is an edge of $R$. 
In total, at most
\[2^{k-1}\beta t+(1-\alpha-(2^k+1)\beta)t < t-\alpha t-\beta t\]
  vertices of $X_i$ are not suitable. Thus at least $(\alpha+\beta)t>|S \cap X_i|$
vertices are suitable, and it is always possible to choose $v_i\in X_i\setminus
S$ as desired. Now the sequence of edges of the form $\{v_1,\ldots,v_i,u_{i+1},\ldots,u_k\}$
for $0\le i\le k$ is a tight walk from $e$ to an excellent edge $e'=\{v_1,\ldots,v_k\}$ which does not intersect
$S$ as desired.
\end{claimproof}

Next we show that any tight component of $\cR^{(k)}$ has a matching with at
least $\alpha t$ edges. So let $C$ be a tight component of $\cR^{(k)}$; then $C$ contains at least one edge. Since by Claim~\ref{clm:pcm:conn} any
edge of $\cR^{(k)}$ is tightly connected to an excellent edge of $\cR^{(k)}$, $C$ contains an excellent edge
$e_1$.

We construct a matching in $C$ as follows. We let $M_1=\{e_1\}$. Now for each
$2\le i\le \alpha t$, let $e_i$ be the edge returned by applying
Claim~\ref{clm:getedge} with $e = e_{i-1}$ and $S$ consisting of all vertices covered by
$M_{i-1}$, and let $M_i=M_{i-1}\cup\{e_i\}$. Then $M_{\alpha t}$ is the
desired matching in $C$. (In fact, this process actually returns a tight path in $C$ on $\alpha t$ vertices. The matching $M_{\alpha t}$ consists of every $k$th edge of this path.)

To complete the proof we show that in the case $\alpha\ge\frac{1}{2}$ we have more: $\cR^{(k)}$ is tightly connected and admits a perfect fractional matching. For the latter we use Lemma~\ref{lem:FracMatch}. Indeed, any edge of $\cR^{(k-1)}$ is contained in at least $\alpha t \geq t/k=t-\tfrac{k-1}{k}t$ edges of $\cR^{(k)}$, the edge of $\cR^{(0)}$ (i.e.~$\emptyset$) is contained in the edges $\{v\} \in \cR^{(1)}$ for $v \in V(\cR)$, and for any $1 \leq i \leq k-2$, any edge of $\cR^{(i)}$ which does not intersect some part $X_j$ is contained in at least $(1-\beta)t \geq (k-1)t/k\ge t-\tfrac{i}{k}t$ edges of $\cR^{(i+1)}$ which do intersect $X_j$ (the first inequality follows from $\alpha + (2^{k-1}+1)\beta \le 1$). So $\cR$ meets the conditions of Lemma~\ref{lem:FracMatch}, from which we deduce that $\cR^{(k)}$ admits a perfect fractional matching.

To show that $\cR^{(k)}$ is tightly connected, it is by Claim~\ref{clm:pcm:conn} enough to show that any two
excellent edges are tightly connected. Given two excellent edges
$\{u_1,\ldots,u_k\}$ and $\{v_1,\ldots,v_k\}$, where $u_i, v_i \in X_i$ for each $i$, we have that both
$e := \{u_1,\ldots,u_{k-1}\}$ and $f := \{v_1,\ldots,v_{k-1}\}$ are edges of $\cR$: in particular
there are at least $2(\alpha+(2^k+1)\beta) t-t \ge (2^k + 1)\beta t$ 
vertices in $X_k$ which form edges with both $e$ and $f$. Of these, all but at
most $2\cdot 2^{k-1}\beta t$ form excellent edges with both $e$ and $f$, so we can choose a
vertex $w_k \in X_k$ so that $e \cup \{w_k\}$ and $f \cup \{w_k\}$ are both excellent edges of $\cR^{(k)}$. We repeat the same
process for each $i=k-1,\ldots,1$ in that order to find $w_i \in X_i$ which forms
excellent edges with each of $\{u_1,\ldots,u_{i-1},w_{i+1},\ldots,w_k\}$ and
$\{v_1,\ldots,v_{i-1},w_{i+1},\ldots,w_k\}$. We have thus constructed tight walks
from each of $\{u_1,\ldots,u_k\}$ and $\{v_1,\ldots,v_k\}$ to $\{w_1,\ldots,w_k\}$,
and thus the two edges are tightly connected, as desired.
\end{proof} 

We now prove Theorem~\ref{thm:PartiteCycle}, which we restate for the reader's convenience. We note that an ordering of
constants and functions:
\[ 1/n_1 \ll \eps(\cdot), 1/r(\cdot) \ll 1/t_1 \ll \eps_k \ll \beta, d_k, \psi \ll \delta,
1/k \] would be enough in the proof, but for convenience we will give values
explicitly.

\medskip \noindent {\bf Theorem~\ref{thm:PartiteCycle}.} 
\emph{For every positive $\delta$ and every integer $k \ge 3$ there is an integer
 $n_{*}$ such that the following holds for each $\alpha\in[0,1]$. If $G$ is a $k$-uniform $k$-partite 
 hypergraph with parts of size $n\ge n_{*}$, such that any collection of $k-1$
 vertices, one in each of $k-1$ parts of $G$, lies in at least
$(\alpha+\delta)n$ edges of $G$, then
 \begin{enumerate}[label=\abc] 
 \item $G$ contains a tight cycle of length $\ell$ for every $\ell\le
   \alpha kn$ that is divisible by $k$, and
 \item if $\alpha\ge\frac{1}{2}$ then $G$ contains a tight cycle of length
   $\ell$ for every $\ell\le(1-\delta)kn$ that is divisible by $k$.
\end{enumerate}}

\begin{proof}
  Given $0 < \delta < 1$ and $k \geq 3$, we choose \[\psi =
  \frac{\delta}{28}\,,\quad d_k =\frac{\delta}{7}\,,\quad
  \beta=\frac{\delta}{7(2^k+1)}\,,\quad t_0=1\,.\]
  We let \[\eps_k \leq \frac{\beta^{k-1}\delta}{7\cdot 2^{k-1}}\] be
sufficiently small for applying Lemma~\ref{lem:emb} with constants $\psi,
\eps_k$ and $d_k$. Next we choose functions $r: \NATS \to \NATS$ growing
sufficiently rapidly and $\eps: \NATS \to (0, 1]$ tending to zero sufficiently
rapidly so that for any $t \in \NATS$ we can apply Lemma~\ref{lem:emb}
with $\eps_k$ as chosen above, $r(t)$, $\eps(t)$ and any $d_2,
\dots, d_k \geq 1/t$. Now Lemma~\ref{lem:simpreg} returns $t_1$ and
$n_0$ such that for any $n\ge n_0$ divisible by $t_1!$, any $kn$-vertex
$k$-graph $G$ and any partition $\Qart$ of $V(G)$ into $k$ parts of size $n$, there is a
$\big(t_0,t_1,\eps(t_1),\eps_k,r(t_1)\big)$-regular slice for
$G$ which satisfies conditions~\ref{simpreg:a} and~\ref{simpreg:b} of
Lemma~\ref{lem:simpreg} and whose ground partition refines $\Qart$. Fix this
$t_1$; for the rest of the proof we will write $\eps$ and $r$ for
$\eps(t_1)$ and $r(t_1)$ respectively. Finally we choose $n_1 \geq \max
\left(n_0, 7t_1!/\delta\right)$ to be sufficiently large to apply
Lemma~\ref{lem:emb} with $t_1$ in place of $t$, $n_1$ in place of $n_0$ and all other constants as above.

Let $G'$ be a $k$-partite $k$-graph with parts $U'_1, \ldots, U'_k$ of size $n'\ge n_{*} :=
n_1 + t_1!$ such that any $(k-1)$-set of vertices, one in each of $k-1$
parts of $G'$, lies in at least $(\alpha+\delta)n'$ edges of $G'$. Choose $n$
divisible by $t_1!$ with $n'-t_1! \leq n \leq n'$ and a subset $U_i \subseteq
U'_i$ of size $n$ for each $i \in [k]$, and let $G$ be the subgraph of $G'$
induced by $\bigcup_{i \in [k]} U_i$. Let $\Qart$ denote the partition of
$V(G)$ into the classes $U_i$. Since at most $t_1!$ vertices were deleted
from each vertex class of $G'$, and $n' \geq n_1 \ge 7t_1!/\delta$, any
$\Qart$-partite $(k-1)$-set $S$ of vertices of $G$ lies in at least
$(\alpha+6\delta/7)n'$ edges of $G$, so $\reldeg(S;G, U_i)\ge\alpha+6\delta/7$,
where $U_i$ is the part of $\Qart$ which $S$ does not intersect. By
Lemma~\ref{lem:simpreg} we may choose a $\big(t_0,t_1,\eps,\eps_k,
r\big)$-regular slice $\cJ$ for $G$ which has the
properties~\ref{simpreg:a} and~\ref{simpreg:b} of Lemma~\ref{lem:simpreg} and
whose ground partition refines $\Qart$. Since each cluster of $\cJ$ has the
same size, and the same is true of the parts of $\Qart$, the number of clusters
of $\cJ$ must be divisible by $k$. So let $t$ be such that $kt$ is
the number of clusters. Then each cluster has size $m := n/t$, and since
$\cJ$ is $(t_0, t_1, \eps)$-equitable we have $t_0 \leq
kt \leq t_1$.

Let $\Qart_R$ denote the natural partition of the clusters of $\cJ$; so the
parts of $\Qart_R$ are $V_1, \dots, V_k$, where $V_i$ consists of all clusters
which are subsets of $U_i$. Observe that we have $|V_i| = t$, and that the
reduced $k$-graph $R_{d_k}(G)$ is $\Qart_R$-partite. We would like to
immediately apply Lemma~\ref{lem:PartConnMatch} to the $k$-complex generated by
the down-closure of $R_{d_k}(G)$, but unfortunately this complex may not
satisfy the required degree conditions. Instead we obtain a subcomplex which
does satisfy the necessary conditions.

To do this, consider a $\Qart_R$-partite set $\cX$ of $k-1$ clusters. There is
exactly one vertex class $V_i$ which $\cX$ does not intersect, and we say that
$\cX$ is \emph{good} if there are fewer than $\delta t/7$ clusters $X \in
V_i$ such that $G$ is not $(\eps_k,r)$-regular with respect to the $k$-set $\cX \cup \{X\}$ (for convenience, we henceforth refer to such $k$-sets simply
as \emph{irregular} $k$-sets). Next, for $\ell=k-2,\dots,0$ (in that order)
we say that a $\Qart_R$-partite set~$\cX$ of $\ell$ clusters is \emph{good} if
it is contained in at most $\beta t/2$ $\Qart_R$-partite sets of $\ell+1$
clusters which are \emph{bad}, i.e.~not good.

At the end of this process we have labelled each individual cluster as good or bad. 
Suppose that at least $\beta t/2$ clusters are bad. 
Then we can construct an irregular $k$-set in $R(G)$ by choosing any one
of these bad clusters, any of the at least $\beta t/2$ clusters with
which it forms a bad pair, and so on up to any of the at least $\delta
t/7$ irregular $k$-sets which contain the bad $(k-1)$-set so constructed. Since we
could construct a given irregular $k$-set in at most $k!$ ways, we conclude
that $R(G)$ contains at least $(\beta t/2)^{k-1}(\delta t/7)/k!$
irregular $k$-sets. Since $\eps_k \leq \beta^{k-1}\delta/(7\cdot 2^{k-1})$,
this is greater than $\eps_k\binom{t}{k}$, contradicting the fact that
$\cJ$ is an $\big(t_0,t_1,\eps,\eps_k,r\big)$-regular slice for $G$.

We conclude that fewer than $\beta t/2$ clusters are bad. So for each $1
\leq i \leq k$ we may choose a set $V_i'$ of $t' := (1 - \beta/2)t$
clusters from $V_i$, all of which are good. We now define a $k$-partite
hypergraph $\cR$ on $V' := V_1' \cup \cdots \cup V_k'$ as follows. For each $0
\leq j \leq k-1$ we take $\cR^{(j)}$ to consist of all good $j$-sets of
clusters (so in particular, $\emptyset \in \cR$ and $\{v\} \in \cR$ for any $v
\in V'$). For $\cR^{(k)}$ we instead take all edges of $R_{d_k}(G)$ which are
contained in $V'$.

We want to apply Lemma~\ref{lem:PartConnMatch} with
$t'$ and $\alpha + \delta/7$ in place of $t$
and $\alpha$ respectively. Observe that condition~\ref{pcm:i} is satisfied. For condition~\ref{pcm:ii} note that by definition of $\cR$, for any $1 \leq i \leq k-2$ and $j \in [k]$, any edge
$\cX \in \cR^{(i)}$ which does not intersect $V'_j$ is contained in at least
$|V_j'| - \beta t/2 \geq (1-\beta)t'$ edges of $\cR^{(i+1)}$ which
do intersect $V'_j$.

We now check condition~\ref{pcm:iii}. Consider any edge $\cX$ of $\cR^{(k-1)}$, so $\cX$ is a good set of
$k-1$ clusters, and let $j$ be such that $\cX$ does not intersect $V_j$. Recall
that $\reldeg(S;G, U_j)\ge\alpha+6\delta/7$ for any $\Qart$-partite
$(k-1)$-set of vertices of $G$ which does not intersect $U_j$; it follows that
$\reldeg(\cJ_\cX;G, U_j) \ge \alpha+6\delta/7$. By property~\ref{simpreg:b} of
Lemma~\ref{lem:simpreg} we have $|\reldeg(\cX;R(G), V_j) - \reldeg(\cJ_\cX;G,
U_j)| < \eps_k$, so $\reldeg(\cX; R(G), V_j) \ge \alpha+5\delta/7$. Since $\cX$
is good, the proportion of clusters $X$ in $V_j$ for which $\cX \cup \{X\}$ is
an irregular $k$-set is at most $\delta/7$, and so Lemma~\ref{reducedd+d}
implies that $\reldeg(\cX;R_{d_k}(G), V_j) \ge \reldeg(\cX; R(G), V_j) - d_k -
\delta/7 \geq \alpha+3\delta/7$, and therefore $\reldeg(\cX;R_{d_k}(G), V_j')
\ge \alpha + 2\delta/7$. As $(2^k+1)\beta = \delta/7$, it follows that any edge $\cX$
of $\cR^{(k-1)}$ lies in at least $(\alpha+\delta/7+(2^k+1)\beta)t'$ edges
of $\cR^{(k)}$.

So $\cR$ also satisfies condition~\ref{pcm:iii} of Lemma~\ref{lem:PartConnMatch}. It follows that $\cR^{(k)}$ has a tight component with a
matching containing $(\alpha+\delta/7) t' = (\alpha + \delta/7)(1 -
\beta/2)t \geq (\alpha + \delta/14)t$ edges. Since $\cR^{(k)}$ is a
subgraph of $R_{d_k}(G)$, by Lemma~\ref{lem:emb} $G$ contains a tight cycle of length $\ell$ for any $k < \ell \leq (1-\psi)(\alpha +
\delta/14)kn$ which is divisible by $k$. Since
$(1-\psi)(\alpha + \delta/14)kn \geq \alpha kn$ this completes the proof of~\ref{thm:PartiteCycle:small}.
If $\alpha \ge \frac{1}{2}$ then Lemma~\ref{lem:PartConnMatch} also shows
that $\cR^{(k)}$ admits a tightly connected perfect fractional matching, and~\ref{thm:PartiteCycle:big}
follows analogously.
\end{proof}

\section{Strong Hypergraph Regularity}\label{sec:reg}

In this section we introduce the Strong Hypergraph Regularity Lemma and some related tools.

\subsection{Families of partitions and the Strong Hypergraph Regularity Lemma}
In order to state the Strong Hypergraph Regularity Lemma we need some more notation.
Fix $k \geq 3$, and let $\Part$ partition a vertex set $V$ into parts $V_1,
\dots, V_t$.
For any $A \subseteq [t]$, we denote by $\cross_A$ the
collection of $\Part$-partite subsets $S \subseteq V$ of index $i(S) = A$.
Likewise, we denote by $\cross_j$ the union of $\cross_A$ for each $A \in
\binom{[t]}{j}$, so $\cross_j$ contains all $\Part$-partite subsets $S
\subseteq V$ of size $j$. Note that $\cross_A$ and $\cross_j$ are dependent on 
the choice of partition $\Part$, but this will always be clear from the context.
For each $2 \leq j \leq k-1$ and $A \in \binom{[t]}{j}$ let $\Part_A$ be a
partition of $\cross_A$. For consistency of notation we also define the trivial partitions $\Part_{\{s\}} := \{V_s\}$ for $s \in [t]$ and $\Part_\emptyset := \{\emptyset\}$. Let $\Part^*$ consist of the partitions $\Part_A$ for each $A\in\binom{[t]}{j}$ and each $0 \le j\le k-1$. We say that $\Part^*$ is a \emph{$(k-1)$-family of partitions
on $V$} if whenever $S, T \in \cross_A$ lie in the same part of $\Part_A$ and $B
\subseteq A$, then $S \cap \bigcup_{j\in B} V_j$ and $T \cap
\bigcup_{j\in B} V_j$ lie in the same part of $\Part_B$. In other words, given $A \in \binom{[t]}{j}$, if we specify one part of each $\Part_B$ with $B \in \binom{A}{j-1}$, then we obtain a subset of $\cross_A$ consisting of all $S\in\cross_A$ whose $(j-1)$-subsets are in the specified parts. Thus the partitions $\Part_B$ give a natural partition of $\cross_A$, and we are saying that $\Part_A$ must refine it.

We refer to the parts
of each member of $\Part^*$ as \emph{cells}. Also, we refer to $\Part$ as the \emph{ground partition} of $\Part^*$, and the parts
of $\Part$ (i.e.~the vertex classes $V_i$) as the \emph{clusters} of $\Part^*$. For
each $0 \leq j \leq k-1$ let $\Part^{(j)}$ denote the partition of $\cross_j$
formed by the parts (which we call \emph{$j$-cells}) of each of the partitions 
$\Part_A$ with $A \in \binom{[t]}{j}$ (so in particular $\Part^{(1)} = \Part$).

The cells of $\Part^*$ naturally form $(k-1)$-complexes on~$V$. 
Indeed, for any $0 \leq j \leq k-1$, any $A \in
\binom{[t]}{j}$ and any $Q' \in \cross_A$, let $C_{Q'}$ denote the cell of
$\Part_{A}$ which contains $Q'$. 
Then the fact that $\Part^*$ is a family of partitions implies that 
for any $Q \in \cross_k$ the union 
$\cJ(Q) := \bigcup_{Q' \subsetneq Q} C_{Q'}$ 
of cells containing subsets of $Q$ is a
$k$-partite $(k-1)$-complex.
We say that the $(k-1)$-family of partitions $\Part^*$ is \emph{$(t_0,
t_1, \eps)$-equitable} if 
\begin{enumerate}[label=\abc]
  \item\label{equitfam:a} $\Part$ partitions~$V$ into~$t$ clusters of equal size, where
  $t_0 \leq t \leq t_1$,
  \item\label{equitfam:b} for each $2 \leq j \leq k-1$, $\Part^{(j)}$ partitions
  $\cross_j$ into at most $t_1$ cells,
  \item\label{equitfam:c} there exists ${\mathbf{d}}=( d_{k-1},
  \dots, d_2)$ such that for each $2 \leq j \leq k-1$ we have $d_j \geq
  1/t_1$ and $1/d_j \in \NATS$, and for every $Q \in \cross_k$ the $k$-partite
  $(k-1)$-complex 
$\cJ(Q)$ 
is $({\bf d},\eps,\eps,1)$-regular.
\end{enumerate}
Note that conditions~\ref{equitfam:a} and~\ref{equitfam:c} imply that $\cJ(Q)$ is a $(1, t_1,
\eps)$-equitable $(k-1)$-complex (with the same density vector $\mathbf{d}$) for any $Q \in \cross_k$.

Next, for any $\Part$-partite set $Q$ with $2 \leq |Q| \leq k$, define $\hat{P}(Q; \Part^*)$
to be the $|Q|$-partite $(|Q|-1)$-graph on $V_{i(Q)}$ with edge set $\bigcup_{Q' \in
\binom{Q}{|Q|-1}} C_{Q'}$. We refer to $\hat{P}(Q; \Part^*)$ as a \emph{polyad}; when the family of partitions $\Part^*$ is clear from the context, we write simply $\hat{P}(Q)$ rather than $\hat{P}(Q; \Part^*)$. 
Note that the condition for $\Part^*$ to be a $(k-1)$-family of partitions can then be rephrased as saying that if $2 \leq |Q| \leq k-1$ then the cell $C_Q$ is supported on $\hat{P}(Q)$. 
Moreover, we will show in the proof of Lemma~\ref{lem:simpreg} (Claim~\ref{claim:numcell}) that if $\Part^*$ is $(t_0, t_1, \eps)$-equitable for sufficiently small $\eps$, then for any $2 \leq j \leq k-1$ and any $Q \in \cross_j$ the number of $j$-cells of $\Part^*$ supported on $\hat{P}(Q)$ is precisely equal to $1/d_j$.

Now let~$G$ be a $k$-graph on~$V$, and let $\Part^*$ be a $(k-1)$-family of
partitions on~$V$. Let $Q\in\cross_k$, so the polyad $\hat{P}(Q)$ is a $k$-partite $(k-1)$-graph. Recall
(see Section~\ref{subsec:regdef}) that $G$ is $(\eps_k,r)$-regular
with respect to $\hat{P}(Q)$ if there is some $d$ such that~$G$ is
$(d,\eps_k,r)$-regular with respect to $\hat{P}(Q)$. 
We say that~$G$
is \emph{$(\eps_k,r)$-regular with respect to~$\Part^*$} if there are at most
$\eps_k \binom{|V|}{k}$ sets $Q \in \cross_k$ for which $G$ is not
$(\eps_k,r)$-regular with respect to the polyad $\hat{P}(Q)$. That is,
at most an $\eps_k$-proportion of subsets of $V$ of size $k$ yield polyads with respect to which $G$ is not regular (though some subsets of $V$ of size $k$ do
not yield any polyad due to not being members of $\cross_k$). Similarly, we say that~$G$
is \emph{perfectly $(\eps_k,r)$-regular with respect to~$\Part^*$} if for every
set $Q \in \cross_k$, the graph $G$ is $(\eps_k,r)$-regular with respect to the
polyad $\hat{P}(Q)$, i.e.~there are no polyads with respect to which
$G$ is not regular.

Finally, we define the notion of a slice through a family of partitions. Indeed, 
if $\Part^*$ is a $(k-1)$-family of partitions, then a \emph{slice through $\Part^*$} 
is a $(k-1)$-complex $\cJ$ on $V$ such that for each $0 \leq i \leq k-1$ we may write
$\cJ^{(i)} = \bigcup_{A \in \binom{[t]}{i}} C(A),$
where each $C(A)$ is a cell in $\Part_A$. That is, a slice consists of a single cell 
from each $\Part_A$, but the requirement that $\cJ$ should be a $(k-1)$-complex requires that 
the choices of cells are `consistent', meaning that each chosen 3-cell is supported 
on the chosen 2-cells, and so forth. Observe in particular that any slice through 
$\Part^*$ must be $\Part$-partite. Moreover, if $\Part^*$ is 
$(t_0, t_1, \eps)$-equitable then any slice through $\Part^*$ must be 
$(t_0, t_1, \eps)$-equitable also, and the polyads of any slice through $\Part^*$ are also polyads of $\Part^*$. The reason for the term `regular slice' should 
now be apparent: we take a slice through a $(k-1)$-family of partitions $\Part^*$ 
so that the chosen slice has desirable regularity properties. 

We are now ready to state the Strong Hypergraph Regularity Lemma. The form
we consider, due to R\"odl and Schacht, is Lemma~23 in~\cite{RS}\footnote{In fact, their lemma allows for an initial family of partitions rather than just a partition $\Qart$ of the vertex set $V$.}.

\begin{lemma}[Strong Hypergraph Regularity Lemma]\label{reglemma}
Let $k \geq 3$ be a fixed integer. For all positive integers $q$, $t_0$ and $s$,
positive~$\eps_k$ and functions $r: \NATS \rightarrow \NATS$ 
and $\eps: \NATS \rightarrow (0,1]$, there exist integers~$t_1$ 
and~$n_0$ such that the following holds for all $n \ge n_0$ which are divisible 
by~$t_1!$. Let $V$ be a vertex set of size $n$, and suppose that~$G_1,
\dots, G_s$ are edge-disjoint $k$-graphs on $V$, and that $\Qart$ is a partition
of $V$ into at most $q$ parts of equal size. Then there exists a
$(k-1)$-family of partitions $\Part^*$ on~$V$ such that
\begin{enumerate}[label=\abc]
\item the ground partition of $\Part^*$ refines $\Qart$,
\item $\Part^*$ is $(t_0,t_1, \eps(t_1))$-equitable, and 
\item for each $1 \leq i \leq s$, $G_i$ is $(\eps_k,r(t_1))$-regular with
respect to~$\Part^*$. \qed
\end{enumerate}
\end{lemma}

Similar results were proved previously by R\"odl and Skokan~\cite{RSk} and
Gowers~\cite{Gowers}. In applications of Lemma~\ref{reglemma}, the regularity parameter $\eps_k$ of the graphs $G_i$ is typically much larger than the entries of the density vector of $\Part^*$, which may cause substantial technical difficulties. The next result, also due to  
R\"odl and Schacht (the form we state is Theorem~25 in~\cite{RS}), is a Regular Approximation Lemma, in which the regularity parameter may be taken to be much smaller than the densities of $\Part^*$. However, the price for this is that it is not our $k$-graph $G$ which is regularised, but instead a $k$-graph $G'$ which differs from $G$ only in a small number of edges. 

\begin{lemma}[Regular Approximation Lemma]\label{lem:RAL}
Let $k \geq 3$ be a fixed integer. For all positive integers $q$, $t_0$ and $s$,
positive~$\nu$ and functions $\eps: \NATS \rightarrow (0,1]$, there exist integers~$t_1$ 
and~$n_0$ such that the following holds for all $n \ge n_0$ which are divisible 
by~$t_1!$. Let $V$ be a vertex set of size $n$, and suppose that~$G_1,
\dots, G_s$ are edge-disjoint $k$-graphs on $V$, and that $\Qart$ is a partition
of $V$ into at most $q$ parts of equal size. Then there exist $k$-graphs $G'_1,\ldots,G'_s$
on $V$ and a
$(k-1)$-family of partitions $\Part^*$ on~$V$ such that
\begin{enumerate}[label=\abc]
\item the ground partition of $\Part^*$ refines $\Qart$,
\item $\Part^*$ is $(t_0,t_1, \eps(t_1))$-equitable,
\item for each $1 \leq i \leq s$ the graph $G'_i$ is perfectly
$(\eps(t_1),1)$-regular with respect to~$\Part^*$, and
\item for each $1\leq i\leq s$, we have $|G'_i\Delta G_i|\le\nu n^k$. \qed
\end{enumerate}
\end{lemma}

In fact, in order to prove Lemma~\ref{lem:simpreg} we need a slightly
strengthened version of the Strong Hypergraph Regularity Lemma, which we deduce from Lemma~\ref{lem:RAL}. 

Given families of partitions $\Part^*$ and $\tilde{\Part}^*$, such that the ground partition $\Part$ of $\Part^*$ is a refinement of the ground partition $\tilde{\Part}$ of $\tilde{\Part}^*$, we say that \emph{$\Part^*$ is generated from
$\tilde{\Part}^*$ by $\Part$} if every $\tilde{\Part}$-partite $j$-cell of $\Part^*$ is an
induced subgraph of a $j$-cell of $\tilde{\Part}^*$. Note that this condition does not determine a unique $\Part^*$, since the $j$-cells of $\Part^*$ which are not $\tilde{\Part}$-partite are not determined. However, if $\tilde{\Part}^*$ has many clusters, it does determine most $j$-cells of $\Part^*$. If $\tilde{\cJ}$ is a slice through $\tilde{\Part}^*$ and $\tilde{X}$ is a
set of clusters of $\tilde{\Part}$, then $\tilde{\cJ}[\tilde{X}]$ is the subset of elements of
$\tilde{\cJ}$ supported on $\tilde{X}$. Next, given a set $\tilde{X}$ of $\ell$ clusters of $\tilde{\Part}^*$, we say that a set $X$ of clusters of $\Part^*$ is \emph{$\tilde{X}$-consistent} if $X$ has precisely $\ell$ members, each of which is a subset of precisely one member of $\tilde{X}$. Finally, given such $X$ and $\tilde{X}$, for any slice $\cJ$ through $\Part^*$ we say that \emph{$\cJ[X]$
is contained in $\tilde{\cJ}[\tilde{X}]$}, and write $\cJ[X]\subseteq \tilde{\cJ}[\tilde{X}]$, if every
$j$-cell of $\cJ[X]$ is an induced subgraph of a cell of
$\tilde{\cJ}[\tilde{X}]$.

The Strengthened Regularity Lemma we require follows. The difference between this lemma and Lemma~\ref{reglemma} is the appearance
of a `coarse' family of partitions $\tilde{\Part}^*$ containing the `fine' family of partitions $\Part^*$
on which we guarantee regularity properties. This lemma guarantees that the neighbourhood of
some root vertices looks about the same on a part of
the coarse family of partitions as on any of the corresponding parts in the fine family of partitions.

\begin{lemma}[Strengthened Regularity Lemma]\label{lem:strReg}
  Let $k \geq 3$ be a fixed integer. For all positive integers $q$, $t_0$ and $s$,
  positive~$\eps_k$, functions $r: \NATS \rightarrow \NATS$, $\eps: \NATS \rightarrow (0,1]$ and monotone increasing functions 
  $p:\NATS\to\NATS$, there exist integers~$t_1^*$, $t_1$
  and~$n_0$ with $t_1=p(t_1^*)t_1^*$ such that the following holds
  for all $n \ge n_0$ which are divisible by~$t_1!$. Let $V$ be a vertex set of size $n$, and suppose that~$G_1, \dots, G_s$ are edge-disjoint $k$-graphs on $V$, and that $\Qart$ is a
  partition of $V$ into at most $q$ parts of equal size. Then there exist
  $(k-1)$-families of partitions $\Part^*$ and $\tilde{\Part}^*$ on~$V$ with ground
  partitions $\Part$ and $\tilde{\Part}$ respectively such that
  \begin{enumerate}[label=\abc]
    \item\label{strReg:refQ} $\Part$ refines $\Qart$,
    \item\label{strReg:genP} $\Part^*$ is generated from $\tilde{\Part}^*$ by $\Part$,
   \item\label{strReg:numP} $\Part$ has
    $t=p(t_1^*)\tilde{t}$ parts, where $\tilde{t}$ is the number of parts of $\tilde{\Part}$,
    \item\label{strReg:equit} $\tilde{\Part}^*$ and $\Part^*$ are $(t_0,t_1,
    \eps(t_1))$-equitable with equal density vectors, and all
    densities are at least $1/t_1^*$,
   \item\label{strReg:regP} for each $1 \leq i \leq s$, $G_i$ is
    $(\eps_k,r(t_1))$-regular with respect to~$\Part^*$,
    \item\label{strReg:spread} for each $1\leq i\leq s$, each $1\le\ell\le
    1/\eps_k$, each $k$-graph $H$ equipped with a set of distinct root vertices
  $x_1,\ldots,x_\ell$ such that $v(H)\le 1/\eps_k$, any distinct vertices
  $v_1,\ldots,v_\ell$ in $V$, any slices $\tilde{\cJ}$ through $\tilde{\Part}^*$ and $\cJ$
  through $\Part^*$, any $(v(H)-\ell)$-set of clusters $\tilde{X}$ of $\tilde{\Part}^*$ and any 
  $\tilde{X}$-consistent $(v(H)-\ell)$-set $X$ of clusters of $\Part^*$ such that
  $\cJ[X]\subseteq\tilde{\cJ}[\tilde{X}]$, we have
  \[d_H(G_i;v_1,\ldots,v_\ell,\tilde{\cJ}[\tilde{X}])=d_H(G_i;v_1,\ldots,v_\ell,\cJ[X])\pm\eps_k\,.\]
  \end{enumerate}
\end{lemma}

We will prove Lemma~\ref{lem:strReg} in Appendix~\ref{app:strReg}. However,
it is worth noting that we need it only for proving property~\ref{simpreg:c} of
Lemma~\ref{lem:simpreg}. The only use of this property in this paper is in the sketch proof of Theorem~\ref{thm:rrsham} given in Section~\ref{sec:conclude}, so the 
reader whose interest lies only in verifying Theorems~\ref{thm:HypErdGal} and \ref{thm:PartiteCycle} can safely forget Lemma~\ref{lem:strReg} and pretend that
Lemma~\ref{reglemma} provides the $(k-1)$-family of partitions~$\Part^*$ in the proof of
Lemma~\ref{lem:simpreg}.

\subsection{Tools for working with regularity.}
We will need various standard results in the proofs of Lemmas~\ref{lem:simpreg}
and~\ref{lem:emb}, which we present here.

We start with the Hypergraph Counting Lemma. The version we present here is slightly modified from a result
in~\cite{CFKO}, which in turn was derived from~\cite[Theorem~9]{RS2}. We
sketch in Appendix~\ref{app:count} how this modified version can be obtained. Similar
results were also proved previously by Gowers~\cite{Gowers} and by Nagle, R\"odl and
Schacht~\cite{NRS}. We remark that Fact~\ref{fact:densecount} is a special case of this lemma. Note though that in contrast to Fact~\ref{fact:densecount}, Lemma~\ref{lem:count} allows for a $k$-graph $G$ on top of the equitable complex $\cJ$, whose regularity $\eps_k$ with respect to $\cJ$ can be bigger than the entries of the density vector of $\cJ$.

\begin{lemma}[Counting Lemma, {\cite[Lemma 4]{CFKO}}]\label{lem:count}
  Let $k,s,r,m_0$ be positive integers, and let
  $d_2,\ldots,d_{k-1},\eps,\eps_k,\beta$ be positive constants such that $1/d_i
  \in\mathbb{N}$ for any $2 \leq i \leq k-1$ and
  \[\frac{1}{m_0}\ll \frac{1}{r}, \eps \ll
  \eps_k,d_2,\ldots,d_{k-1}\quad\text{and}\quad \eps_k \ll
  \beta, \frac{1}{s}\,.\]
  Then the following holds for all
  integers $m\ge m_0$. Let $H$ be a $k$-graph on $s$ vertices $1, \ldots, s$,
  and let $\mathcal{H}$ be the $k$-complex generated by the down-closure of
  $H$.
Also let $\cJ$ be a $(\cdot, \cdot, \eps)$-equitable $(k-1)$-complex with $s$
clusters $V_1, \dots, V_s$ each of size $m$ and with density vector
$\mathbf{d}$. 
Finally, let $G$ be a $k$-graph on $\bigcup_{i \in [s]} V_i$ which is supported
on $\cJ^{(k-1)}$ such that for any edge $e \in H$ the graph $G$ is $(\eps_k,r)$-regular with respect to the $k$-set of clusters $\{V_j : j \in e\}$. Then the number
of copies of $\mathcal{H}$ in $G$ such that $i$ is in $V_i$ for each $i$ is
$$\left(\prod_{e \in H} \reld(\{V_j : j \in e\}) \pm \beta \right)\left(\prod_{i
= 2}^{k-1} d_i^{e_i(\Hy)}\right) m^s\,.$$
\qed
\end{lemma}

A key property of regular complexes is that the
restriction of such a complex to a large subset of its vertex set is also a
regular complex, with the same relative densities at each level of the complex, albeit
with somewhat degraded regularity properties. The next lemma states this property formally. Its proof is as sketched for Lemma 4.1 in~\cite{lcycles} (the quantification there is slightly different but this does not affect the proof).

\begin{lemma}[Regular Restriction Lemma]\label{restriction} 
Suppose integers $k,m$ and real $\alpha,\eps,\eps_k,d_2,\ldots,d_k > 0$ are such that
$$\frac{1}{m}\ll \eps\ll \eps_k,d_2,\dots,d_{k-1}\quad\text{and}\quad  \eps_k\ll \alpha,\frac{1}{k}\,.$$ 
For any $r,s\in\NATS$ and $d_k>0$, set $\mathbf{d} = (d_k, \dots, d_2)$, and let $\cG$ be an $s$-partite $k$-complex whose vertex classes $V_1, \dots, V_s$ each have size $m$ and which is $(\mathbf{d}, \eps_k, \eps, r)$-regular. Choose any $V'_i \subseteq V_i$ with $|V'_i| \geq \alpha m$ for each $i \in [s]$. Then the induced subcomplex $\cG[V'_1 \cup \dots \cup V'_s]$ is $(\mathbf{d},\sqrt{\eps_k},\sqrt{\eps},r)$-regular. \qed
\end{lemma}

Given a copy of some subgraph $H' \subseteq H$ in $\cG^{(k)}$, how many ways are there to extend $H'$ to a copy of $H$ in $\cG^{(k)}$? The next lemma gives a lower bound on this number for almost all copies of $H'$ in $\cG^{(k)}$. To state this precisely we make the following definitions. 

Let $\cG$ be an $s$-partite $k$-complex whose vertex classes $V_1, \dots ,V_s$
are each of size~$m$, and let $\Hy$ be an $s$-partite $k$-complex whose vertex
classes $X_1, \dots, X_s$ each have size at most $m$. We say that an embedding of $\Hy$ in $\cG$ is
\emph{partition-respecting} if for any $i \in [s]$ the vertices of $X_i$ are
embedded within $V_i$. We denote the set of labelled partition-respecting copies
of $\Hy$ in $\cG$ by $\Hy_\cG$. The Extension Lemma~\cite[Lemma 5]{CFKO}
states that if $\Hy'$ is an induced subcomplex of $\Hy$, and $\cG$ is regular with $\cG^{(k)}$ reasonably dense, then almost all partition-respecting copies of
$\Hy'$ in $\cG$ can be extended to a large number of copies of $\Hy$ in
$\cG$. We present a slightly modified form of this lemma,  implied
by~\cite[Lemma~4.6]{lcycles}.
Again, we explain in Appendix~\ref{app:count} how the modification can be obtained.

\begin{lemma}[Extension Lemma,~{\cite[Lemma 4.6]{lcycles}}]\label{lem:ext}
Let $k,s,r,b,b',m$ be positive integers, where $b'<b$, and let $c,
\beta,d_2,\ldots,d_k,\eps,\eps_k$ be positive constants such that $1/d_i
\in\NATS$ for any $2 \leq i \leq k-1$ and 
\[\frac{1}{m}\ll \frac{1}{r},
\eps \ll c \ll \eps_k,d_2,\ldots,d_{k-1}\quad\text{and}\quad \eps_k \ll
\beta,d_k,\frac{1}{s},\frac{1}{b}\,.\]
Suppose that~$\Hy$ is an $s$-partite $k$-complex on~$b$ vertices
with vertex classes $X_1, \dots, X_s$ and let~$\Hy'$ be an induced subcomplex
of~$\Hy$ on~$b'$ vertices.
Suppose that $\cG$ is an $s$-partite $k$-complex with vertex classes $V_1,\ldots,V_s$, all of size $m$, such that $\bigcup_{i=0}^{k-1}\cG^{(i)}$ is $(\cdot, \cdot, \eps)$-equitable with density vector
$(d_{k-1},\ldots,d_2)$.  Suppose further that for each $e\in\Hy^{(k)}$ with index $A \in \binom{[s]}{k}$, the $k$-graph $\cG^{(k)}[V_A]$ is $(d,\eps_k,r)$-regular with respect to $\cG^{(k-1)}[V_A]$ for some $d\ge d_k$. Then
all but at most $\beta |\Hy'_\cG|$ labelled partition-respecting copies of~$\Hy'$
in~$\cG$ extend to at least $cm^{b-b'}$ labelled partition-respecting
copies of~$\Hy$ in~$\cG$. \qed
\end{lemma}

\section{Proof of the Regular Slice Lemma}\label{sec:regslice}
In this section we prove Lemma~\ref{lem:simpreg}.
We begin with an outline of the proof, considering the case of regularising only one $k$-graph $G=G_1$. Let $\Part^*$ be an equitable $(k-1)$-family of partitions obtained from Lemma~\ref{lem:strReg}. 
As previously noted, we will take a `slice' through $\Part^*$ so that the $i$-cells of this `slice' are consistent.
This can be done by the following procedure: for each pair of clusters we choose one
of the $2$-cells of $\Part^*$ on these clusters. We then throw out all other
$2$-cells and everything in higher levels of $\Part^*$ which is not supported on
our chosen $2$-cells. Now for each triple of clusters we choose one of the
$3$-cells of $\Part^*$ on these clusters, and so on.

Since the family of partitions $\Part^*$ is $(t_0, t_1, \eps)$-equitable,
the procedure described will always output a $(t_0,
t_1, \eps)$-equitable $(k-1)$-complex~$\cJ$. Because $G$ is also
regular with respect to $\Part^*$, we could hope that $\cJ$ will be a regular slice
for $G$. But in order for
this to be true, it is necessary that we do not accidentally choose the
$(k-1)$-cells of disproportionately many polyads $\hat{P}(Q)$ such that
$G$ is not regular with respect to $\hat{P}(Q)$. This already suggests
our proof method: we will follow the above procedure, and when we are required
to pick an $i$-cell for some $i$-set of clusters, we will choose uniformly at
random from the $1/d_i$ possibilities. It is easy to see that the expected
fraction of irregular $k$-sets of clusters of $\cJ$ is then equal to the
fraction of irregular polyads in $\Part^*$. So by linearity of expectation there
exists a regular slice for $G$, an idea which was previously observed in the $k=3$ case by Haxell,
{\L}uczak, Peng, R{\"o}dl, Ruci{\'n}ski and Skokan~\cite{3CycleRamsey}. However,
a straightforward application of McDiarmid's inequality shows that the fraction
of irregular sets is actually sharply concentrated, so that the $(k-1)$-complex
$\cJ$ obtained from the random procedure is very likely to be a regular slice. A similar argument together with the union bound guarantees that $\cJ$ is very likely to satisfy both properties~\ref{simpreg:a} and~\ref{simpreg:b} of Lemma~\ref{lem:simpreg}.

So far the properties of $\Part^*$ we used would also be guaranteed by Lemma~\ref{reglemma}. However, this will not be true for the argument proving property~\ref{simpreg:c}.
The obstacle to obtaining property~\ref{simpreg:c} is as follows: for any $H$ we are asking for
roughly $n^{1/\eps_k}$ distinct inequalities to be satisfied, one for
each choice of root vertices. Although each individual one of these inequalities is very likely to be true, their success probability only depends on the regularity $\eps_k$ and the number of clusters $t$, but not on $n$. Hence a union bound fails miserably. In order to circumvent this, we make use of the coarse $(k-1)$-family of partitions $\tilde{\Part}^*$ given by Lemma~\ref{lem:strReg}. We will illustrate our strategy in the case that $H$ is a $k$-edge, rooted at a single vertex. Each vertex $v$ of $G$ defines a $(k-1)$-graph, called the \emph{link of $v$}, whose edges are those $(k-1)$-sets which together with $v$ form edges of $G$. Then in this case, property~\ref{simpreg:c} states that for each $v\in V(G)$, the density of the link of $v$ is close to the density of the link of $v$ supported on $\cJ$.

Observe that most $(k-1)$-sets in $V(G)$ are $\tilde{\Part}$-partite, and the $\tilde{\Part}$-partite $(k-1)$-sets are partitioned by $\tilde{\Part}^*$ into its $(k-1)$-cells. These $(k-1)$-cells are of approximately equal size by Fact~\ref{fact:densecount}, and it follows that for each $v$ the density of the link of $v$ is close to the average, over $(k-1)$-cells $C$ of $\tilde{\Part}^*$, of the density of the link of $v$ on $C$. Now each $(k-1)$-cell $C$ of $\tilde{\Part}^*$ is partitioned by $\Part$ into a large number of $(k-1)$-cells of $\Part^*$, and Lemma~\ref{lem:strReg}\ref{strReg:spread} guarantees that for each $v$ the link density of $v$ on each of these parts is close to the link density of $v$ on $C$. It follows that if $\cJ$ is a slice through $\Part^*$ whose $(k-1)$-cells are chosen with about equal frequency from the $(k-1)$-cells of $\tilde{\Part}^*$ then $\cJ$ satisfies this case of property~\ref{simpreg:c}. We can again prove that this is the likely outcome using McDiarmid's inequality.

\smallskip

The form of McDiarmid's inequality we use is the following.

\begin{theorem}[McDiarmid's inequality, \cite{McDiarmid}]\label{thm:azuma} 
  Let $f:\REALS^n\to\REALS$ be such that there
  exists a vector $(c_1,\ldots,c_n)\in\REALS^n$ with the following property. For
  each $\mathbf{x}$, if $\mathbf{x}'$ differs from $\mathbf{x}$ only in 
  coordinate $i$ then $\big|f(\mathbf{x})-f(\mathbf{x}')\big|\le c_i$. Now if
  $(X_1,\ldots,X_n)$ is a vector of independent random variables, then for each $a> 0$ we have
  \[\Pr\Big(\big|f(X_1,\ldots,X_n)-\Exp f(X_1,\ldots,X_n)\big|\ge a\Big)\le
  2\exp \Big(\frac{-a^2}{2\sum_{i=1}^n c_i^2}\Big)\,. \] \qed
\end{theorem}

We will also make several uses of the following probabilistic statement which shows that $\Prob(E\mid C\cap S)$ and $\Prob(E\mid S)$ are nearly the same when $S$ `almost implies' $C$.

\begin{proposition} \label{prop:condprob}
Let $E, C$ and $S$ be events in some probability space such that $\Prob(S) > 0$ and $\Prob(C^c \mid S) \leq \eps$ for some constant $\eps \leq \tfrac{1}{2}$. Then $\Prob(E \mid C \cap S) = \Prob(E \mid S) \pm 2\eps$.
\end{proposition}

\begin{proof}
Observe that $\Prob(C \cap S) = \Prob(S)(1-\Prob(C^c \mid S)) > 0$, so $\Prob(E \mid C \cap S)$ is well-defined, and
\begin{align*}
(1&+2\eps)\Prob(E \mid S)
\geq
\frac{\Prob(E \cap S)}{(1-\eps) \Prob(S)}
\geq
\frac{\Prob(E \cap S)}{\Prob(S) - \Prob(C^c \cap S)}\\
&\geq
\Prob(E \mid C \cap S)
=
\frac{\Prob(E \cap C \cap S)}{\Prob(C \cap S)}
\geq
\frac{\Prob(E \cap S) - \Prob(E \cap C^c \cap S)}{\Prob(S)} \\
&\geq
\frac{\Prob(E \cap S)}{\Prob(S)} - \frac{\Prob(C^c \cap S)}{\Prob(S)}
\geq
\Prob(E \mid S) - \eps.
\end{align*}
\end{proof} 

\begin{proof}[Proof of Lemma~\ref{lem:simpreg}]
Given integers $q$, $t_0$ and $s$,
a constant $\eps_k$, and functions $r:\NATS\to\NATS$ and $\eps:\NATS\to(0,1]$, we define
further constants as follows. Without loss of generality we assume that $\eps_k
\leq 1$, and define 
\begin{equation}\label{simpreg:etas}
t_0^*:=\max \left(t_0,
\frac {128k^2}{\eps_k^{2+2/\eps_k+k}},
\frac{4096k!^4k2^{2k}q^ks^2}{\eps_k^{4}}\right)\,.
\end{equation}
Next we choose a constant
\begin{equation}\label{simpreg:epsks}
  \eps_k^* \le
  \frac{\eps_k^{2k+1}}{40s\cdot 2^k}\,.
\end{equation}
so that $\eps_k^*$ is sufficiently small to apply Lemma~\ref{lem:count} with $\eps_k^*$, $\eps_k/100$ and $1/\eps_k$ in place of $\eps_k$, $\beta$ and $s$ respectively.
We choose $p:\NATS\to\NATS$ to be any monotone increasing function such that for
all $x\in\NATS$ we have
\begin{equation}\label{simpreg:p}
  2\exp\Big(-2^{-x-9}\eps_k^2p(x)^2\prod_{j=2}^{k-1}
\big(\tfrac{1}{x}\big)^{2\binom{x}{j}}\Big)<2^{-xp(x)}
\prod_{j=2}^{k-1}\big(\tfrac{1}{x}\big)^{\binom{x}{j}}\,.
\end{equation}
It may not be immediately obvious that this is possible, but observe that if $p(x) > \max\big(2x/c_1(x),2/c_2(x)\big)$ then we have
\[2\exp\big(-c_1(x)p(x)^2\big)<2\exp\big(-2xp(x)\big)<c_2(x)2^{-xp(x)}\,,\]
from which~\eqref{simpreg:p} follows by appropriate choice of $c_1,c_2$.
We choose strictly monotone functions
$r^*:\NATS\to\NATS$ and $\eps^*:\NATS\to(0,1]$ such that 
\begin{equation}\label{simpreg:epssrs}
  \eps^*(t_1) \le \min \left(\eps(t_1), \frac{\eps_k}{10\cdot
2^{2^{1/\eps_k}}}, \frac{1}{2t_1^2} \right)\text{ and }r^*(t_1) \geq
r(t_1)
\end{equation}
for any $t_1 \in \NATS$. Moreover, we make these choices so that for
any $t_1 \in \NATS$ we can apply Lemma~\ref{lem:count} to count graphs on up to $1/\eps_k$ vertices, with 
$r^*(t_1)$, $\eps^*(t_1)$, $\eps_k^*$ and $\eps_k/100$ in place of 
$r$, $\eps$, $\eps_k$ and $\beta$ respectively, with $1/t_1$ in place of each $d_i$, and so that we can apply
Fact~\ref{fact:densecount} with $\eps_k/10$ in place of $\beta$ to any
$(t_0^*,t_1,\eps^*(t_1))$-equitable $(k-1)$-complex with
sufficiently large clusters.
Let $n^*_0$, $t_1^*$ and $t_1$ be obtained by applying
Lemma~\ref{lem:strReg} with inputs $q, t_0^*, s, \eps_k^*$ and functions
$r^*(\cdot)$, $\eps^*(\cdot)$ and $p(\cdot)$. Note that we
have $t_1=p(t_1^*)t_1^*$ by Lemma~\ref{lem:strReg}. Finally, let $n_0 \geq n^*_0$ be sufficiently
large to apply Lemma~\ref{lem:count} and Fact~\ref{fact:densecount} with $m_0 =
n_0/t_1$ and all other constants as before. For the remainder of the proof we
write $\eps$, $\eps^*$, $r$ and $r^*$ to denote $\eps(t_1)$,
$\eps^*(t_1)$, $r(t_1)$ and $r^*(t_1)$ respectively.

Let $V$ be a set of $n \geq n_0$ vertices, where $n$ is divisible by $t_1!$,
and let $\Qart$ partition $V$ into at most $q$ parts of equal size. Let
$G_1, \dots, G_s$ be edge-disjoint $k$-graphs on the vertex set $V$. Then we
must show that there exists a $(k-1)$-complex $\cJ$ on $V$ which is a
$(t_0,t_1,\eps,\eps_k,r)$-regular slice for each $G_i$, whose ground
partition $\Part$ refines $\Qart$, and which satisfies
properties~\ref{simpreg:a},~\ref{simpreg:b} and~\ref{simpreg:c} of the lemma. We
start by applying Lemma~\ref{lem:strReg} (with the inputs stated above), which
yields partitions $\Part^*$ and $\tilde{\Part}^*$ with the properties stated in that
lemma. In particular, $\Part^*$ and $\tilde{\Part^*}$ are both $(t_0^*,t_1,\eps^*)$-equitable with the same density vector $\mathbf{d} = (d_{k-1}, \dots, d_2)$, the ground partition $\Part$ of $\Part^*$ refines $\Qart$, and each $G_i$ is $(\eps^*_k,r^*)$-regular
with respect to $\Part^*$. 
 
\begin{claim}\label{claim:numcell} For each $2 \le i \le k-1$ and each $\Part$-partite set $Q \in \binom{V}{i}$, the number of $i$-cells of $\Part^*$ supported on the polyad $\hat{P}(Q, \Part^*)$ is precisely equal to $1/d_i$. Moreover, the same statement holds with $\tilde{\Part}$ and $\tilde{\Part^*}$ in place of $\Part$ and $\Part^*$ respectively. 
\end{claim}

\begin{claimproof} We prove the claim for $\Part^*$; the proof for $\tilde{\Part}^*$ is identical. Recall that part~\ref{equitfam:c} of the definition of an equitable family of partitions implies that $1/d_i$ must be an integer, and that each $i$-cell of $\Part^*$ supported on $\hat{P}(Q)$ is $(d_i,\eps^*,1)$-regular with respect to $\hat{P}(Q)$.
In particular, the number of $i$-sets in each of these $i$-cells is $(d_i\pm\eps^*)\big|K_i\big(\hat{P}(Q)\big)\big|$. 
Suppose for a contradiction that for some polyad $\hat{P}(Q)$, the number of $i$-cells of $\Part^*$ supported on $\hat{P}(Q)$ is at least $\tfrac{1}{d_i}+1$.
Since these $i$-cells are pairwise-disjoint and cover $K_i\big(\hat{P}(Q)\big)$, we conclude that
\[\big|K_i\big(\hat{P}(Q)\big)\big|\ge\big(\tfrac{1}{d_i}+1\big)(d_i\pm\eps^*)\big|K_i\big(\hat{P}(Q)\big)\big|\,.\]By Fact~\ref{fact:densecount} this number is non-zero, hence $1\ge\big(\tfrac{1}{d_i}+1\big)(d_i\pm\eps^*)$, which is a contradiction, since we have $d_i\ge 1/t_1$ and $\eps^*\le 1/(2t_1^2)$. A similar argument shows that the number of $i$-cells of $\Part^*$ supported on $\hat{P}(Q)$ is bigger than $\tfrac{1}{d_i}-1$, giving the desired result.
\end{claimproof}

We now construct our regular slice $\cJ$, whose clusters are the clusters of $\Part^*$ (that is, the parts of $\Part$), by applying Algorithm~\ref{alg:simp} to $\Part^*$. We claim
that with positive probability $\cJ$ satisfies the conclusions of
Lemma~\ref{lem:simpreg}.
 
\begin{algorithm}
  \caption{Taking a slice}\label{alg:simp}
  let $\cJ = \{\emptyset\} \cup \{\{v\} : v \in V\}$\;
  \ForEach{$2 \le i \le k-1$}{
   \ForEach{$X \in \binom{\Part}{i}$}{
    let $\mathcal{C}:=\{C\colon\, C \text{ is an $i$-cell of $\Part^*$ and is supported on } \cJ_{X^<}^{(i-1)}\}$\;
    choose $C\in\mathcal{C}$ uniformly at random\;
    let $\cJ := \cJ \cup C$\;
   }
  }
  return $\cJ$\;
\end{algorithm}

It is immediate from the algorithm that the output $\cJ$ is a $(k-1)$-complex which is a slice through $\Part^*$. 
Since $\Part^*$ is $(t_0^*, t_1, \eps^*)$-equitable with density vector $\mathbf{d} = (d_{k-1}, \dots, d_2)$, $\cJ$ is also $(t_0^*, t_1, \eps^*)$-equitable with the same density vector. Let $t$ denote the number of clusters of $\cJ$, so
$t_0^* \leq t \leq t_1$, and let $m := n/t$, so $m$ is the common size of each cluster. Note that by Claim~\ref{claim:numcell} every slice through $\Part^*$ has probability precisely $\prod_{i=2}^{k-1} d_i^{\binom{t}{i}}$ of being selected as $\cJ$. 

To show that $\cJ$ is likely to be a regular slice for each $G_i$, we must bound the
number of \emph{irregular polyads}, that is, polyads $\hat{\cJ}_X$ for which some $G_i$ with $1\le i\le s$ is not $(\eps^*_k,
r^*)$-regular with respect to $\hat{\cJ}_X$. Since each $G_i$ is $(\eps^*_k,r^*)$-regular with respect to $\Part^*$, the number of $\Part$-partite sets $Q \in \binom{V}{k}$ for which $G_i$ is not $(\eps_k^*,r^*)$-regular with respect to $\hat{P}(Q)$ is at most $\eps_k^*\binom{n}{k}$. On the other hand, since $\eps^*$ is chosen small enough so that we can apply Fact~\ref{fact:densecount} with $\eps_k/10$ in place of $\beta$, each
polyad $\hat{P}(Q)$ of $\Part^*$ supports
\[\big(1 \pm \tfrac{\eps_k}{10}\big) m^k \prod_{i=2}^{k-1} d_i^{\binom{k}{i}}\]
$\Part$-partite members of $\binom{V}{k}$. It follows that the number of irregular polyads in $\Part^*$ is at most
\begin{equation}
\label{eq:irregpoly}
\frac{s\eps_k^*\binom{n}{k}}{\big(1 - \tfrac{\eps_k}{10}\big)m^k\prod_{i=2}^{k-1} d_i^{\binom{k}{i}}}\le 2 s \eps^*_k \binom{t}{k} \prod_{i=2}^{k-1}
d_i^{-\binom{k}{i}}\,.
\end{equation}
By Claim~\ref{claim:numcell} the probability that a given polyad of $\Part^*$ is chosen
for $\cJ$ is precisely $\prod_{i=2}^{k-1} d_i^{\binom{k}{i}}$. So by linearity of expectation, the expected number of $k$-sets of clusters $X$ such that $\cJ_X$ is an irregular polyad is at
most $2s \eps^*_k \binom{t}{k}$. By Markov's inequality, we conclude that
with probability at least $1/2$ this number is at most $4s \eps^*_k
\binom{t}{k}$, which by~\eqref{simpreg:epsks} is at most $\eps_k
\binom{t}{k}$. In particular, $\cJ$ is a $(t_0, t_1, \eps, \eps_k,
r)$-regular slice for each $G_i$ with probability
at least $1/2$. We stress however that later in this proof we will make use of the stronger property that $\cJ$ contains at most $4s\eps^*_k\binom{t}{k}$ polyads with respect to which some $G_i$ is not $(\eps^*_k,r^*)$-regular.

It remains to show that properties~\ref{simpreg:a},~\ref{simpreg:b}
and~\ref{simpreg:c} each hold with high probability. In each case, this
amounts to verifying that some random variables are likely to all
be close to certain values, and our proof in each case follows the pattern of
first showing that the expectations of these random variables are indeed close
to the desired values, then showing that the random variables are sufficiently
concentrated to apply a union bound. We begin with property~\ref{simpreg:a}.

The following Claim~\ref{lem:expect} shows that for any $i$ the reduced
$k$-graph $R(G_i)$ of $G_i$ (with respect to $\cJ$) has approximately the same
$H$-density as $G_i$ for each small $H$ in expectation, and that this remains
true when considering the subgraph induced on a large subset of the clusters.
For this, we say that a copy of $H$ in $G_i$ is \emph{$\Part$-crossing} if it
has at most one vertex in any part of $\Part$.

If it happens to be the case that the density of $H$-copies in $G$ equals
the density of $\Part$-crossing $H$-copies, and $\eps$ and $\eps_k$ are zero,
this claim is a triviality. The number of $\Part$-crossing copies of $H$ in 
$G_i[\bigcup X]$ then equals the number obtained by applying
Lemma~\ref{lem:count} to $\Part^*$, which in this case means that $d_H\big(R(G_i)[X]\big)$
equals $d_H\left(G \left[\bigcup X \right]\right)$ precisely.

Of course none of these assumptions
are true, and consequently errors are introduced: the proof of the following
claim amounts to showing that since the assumptions are `almost' true, the
errors are small. It is convenient in the proof to use probabilistic language:
the probability that a random injective map from $V(H)$ to $V(G)$ is a graph embedding is exactly the density of $H$-copies in $G$.

\begin{claim}\label{lem:expect}
Let $H$ be a labelled $k$-graph on $h \leq 1/\eps_k$ vertices, $X$ be a set of
at least $\eps_k t$ clusters of $\cJ$, and $G = G_i$ for some $i \in [s]$. Then 
$$\left| \Exp_\cJ d_H\big(R(G)[X]\big) - d_H\left(G \left[\bigcup X \right]\right)\right| \leq \eps_k/2.$$
\end{claim}

\begin{claimproof}
Let the vertices of $H$ be labelled with the integers $1, \dots, h$.
Choose an injective map $\phi : V(H) \to \bigcup X$ uniformly at random. Say
that $\phi$ is \emph{$\Part$-crossing} if no two vertices of $H$ are mapped to the same
part of $\Part$, and let $\CROSS$ be the event that $\phi$ is $\Part$-crossing.
Also let $\HOM$ denote the event that $\phi$ is an embedding of $H$ into $G$.
We prove the claim by calculating the conditional probability $\Prob(\HOM
\mid \CROSS)$ (that is, the proportion of $\Part$-crossing $h$-tuples which form $\Part$-crossing copies of $H$ in $G[\bigcup X]$) in two different ways.

First, note that we have $\Prob(\HOM) =
d_H(G[\bigcup X])$ by definition. Recall that $m = n/t$ is the size of each
cluster of $\cJ$. Then there are $h!\binom{|X|m}{h}$ $h$-tuples in $\bigcup X$. At
most $h!\binom{|X|m}{h-1}hm$ of these $h$-tuples are not $\Part$-crossing. 
Since $t \geq t_0^* \geq 20/\eps_k^4$ by~\eqref{simpreg:etas} and $|X|\ge\eps_k t$ we have
$h \leq 1/\eps_k \leq \eps_k t \leq |X|m/2$. 
So the
probability of the complement $\CROSS^c$ is at most
\[\Prob(\CROSS^c) \leq \frac{h!\binom{|X|m}{h-1}hm}{h!\binom{|X|m}{h}} = \frac{h^2m}{|X|m - h+1} \leq
\frac{2h^2}{|X|} \leq \frac{2}{\eps_k^3 t} \leq \frac{\eps_k}{10}\,.\] 
We may therefore apply Proposition~\ref{prop:condprob} with $\HOM, \CROSS$ and the certain event in place of $E, C$ and $S$ respectively, which gives
\begin{equation}\label{expect:dHHomCross} 
\Prob(\HOM \mid \CROSS) = \Prob(\HOM) \pm \frac{\eps_k}{5} = d_H\left(G\left[\bigcup X\right]\right) \pm \frac{\eps_k}{5}\,.
\end{equation}

We now turn to the second way of evaluating $\Prob(\HOM\mid\CROSS)$, this time in terms of $\Exp_\cJ d_H\big(R(G)[X]\big)$. For any 
$h$-tuple $\mathcal{C} := (C_1, \dots, C_h)$ of distinct clusters of $\cJ$, we say that a copy of $H$ in $G$ is \emph{$\mathcal{C}$-crossing} if vertex $j$ of $H$ lies in the cluster $C_j$ for each $j \in [h]$. Similarly, we say that $\phi$ is $\mathcal{C}$-crossing if $\phi(j) \in C_j$ for each $j \in [h]$, an event which we denote by $\CROSS_\mathcal{C}$.

Fix some $h$-tuple $\mathcal{C} := (C_1, \dots, C_h)$, and let $\Hy$ be the $k$-complex generated by the down-closure of $H$. For any slice $\cJ$ though $\Part^*$ and any $e \in \Hy$ let $\cJ_e$ denote the cell of $\cJ$ on the clusters $\{C_j \colon j \in e\}$, and define
 $\cJ(H,\mathcal{C}):= \bigcup_{e \in \Hy} \cJ_e$ (so if $\phi$ is $\mathcal{C}$-crossing, then $\cJ(H, \mathcal{C})$ is the subcomplex of $\cJ$ consisting of all cells of $\cJ$ on sets of clusters to which $\phi$ maps an edge of $\Hy$). We can then extend $\cJ(H, \mathcal{C})$ to a slice through $\Part^*$ by appropriate choices of cells on each remaining set of clusters, and by Claim~\ref{claim:numcell}, the number of such extensions does not depend on~$\cJ$. We therefore have
 \begin{equation}\label{eq:expect:exc} 
\Exp_\cJ
 \prod_{e \in H} \reld(\{C_j : j \in e\})=\Exp_{\cJ(H,\mathcal{C})}
 \prod_{e \in H} \reld(\{C_j : j \in e\})\,.
 \end{equation}
 Let $q_{\cJ(H,\mathcal{C})}$ denote the number of $\mathcal{C}$-crossing copies of $H$ in $G$ which are supported on
 $\cJ(H,\mathcal{C})$. 
Then for each fixed $\cJ(H,\mathcal{C})$ we have
 \[\Prob(\HOM\text{ and } \phi(H)\text{ is supported on }\cJ(H,\mathcal{C})\mid \CROSS_\mathcal{C})=\frac{q_{\cJ(H,\mathcal{C})}}{m^h}\,.\]
Observe that for each $\mathcal{C}$-crossing copy of $H$ in $G$ there is precisely one possibility for $\cJ(H,\mathcal{C})$ on which this copy is supported. Moreover, by Claim~\ref{claim:numcell} there are precisely $\prod_{j=2}^{k-1}d_j^{-e_j(\Hy)}$ different choices for $\cJ(H,\mathcal{C})$, each of which is equally likely to occur. Hence we have 
 \begin{equation}\begin{split}\label{eq:expect:hcc}
  \Prob(\HOM \mid \CROSS_\mathcal{C})&=\frac{1}{m^h}\sum_{\cJ(H,\mathcal{C})}q_{\cJ(H,\mathcal{C})}\\
  &=\frac{1}{m^h}\Big(\prod_{j=2}^{k-1}d_j^{-e_j(\Hy)}\Big)\Exp_{\cJ(H,\mathcal{C})}q_{\cJ(H,\mathcal{C})} \,.
 \end{split}\end{equation}

If $\cJ$ has the property that $G$ is $(\eps_k^*, r^*)$-regular with respect to
$\hat{\cJ_Z}$ for every $k$-set $Z \subseteq \mathcal{C}$, then, because we chose $\eps^*_k$, $\eps^*$ and $r^*$ such that we can apply Lemma~\ref{lem:count} with $\beta=\eps_k/100$, the number of $\mathcal{C}$-crossing copies of $H$ in $G$ which are supported on $\cJ(H,\mathcal{C})$ is
\[ q_{\cJ(H,\mathcal{C})}=\left(\prod_{e\in H} \reld\big(\{C_j:j\in e\} \big) \pm
\frac{\eps_k}{100} \right) \left(\prod_{j=2}^{k-1}d_j^{e_j(\mathcal{H})}\right)
m^{h}\,.\]

If on the other hand $\cJ$ does not satisfy this property, we also need some bound on 
$q_{\cJ(H,\mathcal{C})}$. This number is obviously at least zero, and is at most the number 
of copies of $H$ in the $k$-graph whose edges are all $k$-tuples supported on 
$\bigcup_{Z \in \binom{\mathcal{C}}{k}} \hat{\cJ_Z}$. Since this $k$-graph is 
automatically $(\eps^*_k,r^*)$-regular with respect to $\cJ$, we can apply Lemma~\ref{lem:count} 
to find that the number of $\mathcal{C}$-crossing copies of $H$ in $G$
supported on $\cJ(H,\mathcal{C})$ satisfies
\[q_{\cJ(H,\mathcal{C})}\le \big(1+\tfrac{\eps_k}{100}\big)\left(\prod_{j=2}^{k-1}d_j^{e_j(\mathcal{H})}\right) m^{h}\le 2\left(\prod_{j=2}^{k-1}d_j^{e_j(\mathcal{H})}\right) m^{h}\,,\]
and hence
\[ q_{\cJ(H,\mathcal{C})}=\left(\prod_{e\in H} \reld\big(\{C_j:j\in e\} \big) \pm
2 \right) \left(\prod_{j=2}^{k-1}d_j^{e_j(\mathcal{H})}\right)
m^{h}\,.\]

Now let $\zeta(\mathcal{C})$ be the fraction of slices $\cJ$ through $\Part^*$ for which $G$ is not $(\eps_k^*, r^*)$-regular with respect to
$\hat{\cJ_Z}$ for some $k$-set $Z \subseteq \mathcal{C}$. Putting the above two estimates for $q_{\cJ(H,\mathcal{C})}$ together with~\eqref{eq:expect:hcc}, we have
\begin{align*}
\Prob(\HOM\mid \CROSS_\mathcal{C})&=\Big(\Exp_{\cJ(H,\mathcal{C})}\prod_{e \in H} \reld\big(\{C_j : j \in e\}\big)\Big)\pm\left( 2 \zeta(\mathcal{C}) +
\frac{\eps_k}{100}\right)\\
 &\eqByRef{eq:expect:exc} \Big(\Exp_\cJ \prod_{e\in H}
\reld\big(\{C_j:j\in e\} \big)\Big) \pm \left( 2 \zeta(\mathcal{C}) +
\frac{\eps_k}{100}\right)\,.
\end{align*}

We now condition on the event $\CROSS$. Observe that there is then precisely one $h$-tuple $\mathcal{C} = (C_1, \dots, C_h)$ for which the event $\CROSS_\mathcal{C}$ occurs (i.e. $\phi$ is $\mathcal{C}$-crossing). We take $\mathcal{C}$ to be this $k$-tuple, so $\mathcal{C}$ is now a random variable. Moreover, since $\phi$ was chosen uniformly at random and each cluster has equal size, the $k$-tuple $\mathcal{C}$ is chosen uniformly at random among all possibilities, which implies that 
\begin{equation}\label{eq:phom}
\Prob(\HOM \mid \CROSS) = \Big( \Exp_\mathcal{C} \Exp_\cJ \prod_{e\in H} \reld\big(\{C_j:j\in e\} \big)\Big) \pm \left( 2 \Exp_\mathcal{C} \zeta(\mathcal{C}) + \frac{\eps_k}{100}\right).
\end{equation}
However, $\Exp_\mathcal{C} \zeta(\mathcal{C})$ is simply the probability that
our uniformly random $\mathcal{C}$ has the property that 
$G$ is not $(\eps^*_k, r^*)$-regular with respect to $\hat{\cJ_Z}$ for some 
$k$-set $Z \subseteq \mathcal{C}$.
Taking a union
bound, this is at most $\binom{h}{k}$ multiplied by the probability that
$G$ is not $(\eps^*_k, r^*)$-regular with respect to $\hat{\cJ_Z}$ for a $k$-set
$Z$ of clusters of $X$ chosen uniformly at random. Since there are at least
$\binom{\eps_k t}{k}$ $k$-sets $Z$ of clusters in $X$, and 
by~\eqref{eq:irregpoly}
 there are at most $2\eps_k^* s \binom{t}{k} \prod_{j=2}^{k-1}
d_j^{-\binom{k}{j}}$ irregular polyads of $\Part^*$, each of which is chosen for
$\cJ$ with probability $\prod_{j=2}^{k-1} d_j^{\binom{k}{j}}$, we find that
\begin{equation}\label{eq:badpolyads}
\Exp_\mathcal{C} \zeta(\mathcal{C}) \leq \binom{h}{k}\frac{2\eps_k^* s
\binom{t}{k} }{\binom{\eps_k t}{k}}
\leByRef{simpreg:etas} \frac{2\eps_k^* s
t^k h^k}{(\eps_k
t/2)^k }\leByRef{simpreg:epsks}\frac{\eps_k}{20}\,.
\end{equation}
Finally, since $\mathcal{C}$ is chosen uniformly at random, by definition we have 
$d_H\big(R(G)[X]\big) = \Exp_\mathcal{C} \prod_{e \in H} \reld(\{C_j : j \in e\})$, and so
\begin{equation}\label{eq:expecteddensity} 
 \Exp_\cJ d_H\big(R(G)[X]\big) = \Exp_\mathcal{C} \Exp_\cJ 
 \prod_{e \in H} \reld(\{C_j : j \in e\})\,.
\end{equation}
Combining equations~$(\ref{eq:phom})$,~$(\ref{eq:badpolyads})$ and~$(\ref{eq:expecteddensity})$ we conclude
that 
\[\Exp_\cJ d_H\big(R(G)[X]) = \Prob(\HOM \mid \CROSS) \pm \frac{\eps_k}{5}\]
and so by~\eqref{expect:dHHomCross} we have
\[\Exp_\cJ d_H\big(R(G)[X])= d_H\left(G\left[\bigcup X\right]\right) \pm \frac{\eps_k}{2}\,.\]
\end{claimproof}

We next show that the random variable $d_H\big(R(G_i)[X]\big)$ is concentrated
about its mean using McDiarmid's inequality. We will require enough concentration of $d_H\big(R(G_i)[X]\big)$ to make use of a
union bound over all possible $H$ and $X$.

\begin{claim}\label{lem:concentrate} 
  Let $H$ be a $k$-graph on $h \leq 1/\eps_k$ vertices, $X$ a set of
  at least $\eps_k t$ clusters of $\cJ$, and $G = G_i$ for some $i \in [s]$. Then the probability
  that $d_H\big(R(G)[X]\big)$ deviates from $\Exp d_H\big(R(G)[X]\big)$ by more
  than $\eps_k/2$ is at most
  \[2^{-\binom{1/\eps_k}{k}}2^{-2t}\,.\]
\end{claim}
\begin{claimproof} We want to apply McDiarmid's inequality, which holds for a random variable which can be written as $f(\mathbf{x})$ for a vector $\mathbf{x}$ whose entries are independent random variables. Hence we first argue that $d_H\big(R(G)[X]\big)$ is of this form.
  We put an arbitrary order on the $j$-cells of $\Part^*$. Let $Y$ be a $j$-set of clusters. Then observe that
  when Algorithm~\ref{alg:simp} chooses the $j$-cell of $\cJ$ for $Y$, we have already chosen the cells of $\cJ^{(j-1)}$, so, by Claim~\ref{claim:numcell}, $\cJ_Y$
  is chosen uniformly at random from the exactly $d_j^{-1}$ $j$-cells supported on $\cJ^{(j-1)}[Y]$. Equivalently, we choose a number $p$
  from $1$ to $d_j^{-1}$ uniformly at random and take the $p$th $j$-cell
  supported on $\cJ^{(j-1)}[Y]$ in $\Part^*$. It follows that we can write
  \[d_H\big(R(G)[X]\big)=f(\mathbf{x})\] where $\mathbf{x}$ is a vector of
  integers whose first $\binom{|X|}{2}$ coordinates are chosen independently
  uniformly at random in $[d_2^{-1}]$, whose next $\binom{|X|}{3}$ coordinates are chosen independently and uniformly at random in $[d_3^{-1}]$,  and
  so on.
  
  We now want to bound $\big|f(\mathbf{x})-f(\mathbf{x}')\big|$ for two vectors $\mathbf{x},\mathbf{x}'$ differing in only one coordinate. In other words, we want to bound the change $c_Y$ of $d_H\big(R(G)[X]\big)$ when we change the choice of the $j$-cell on the set $Y\in\binom{X}{j}$ in Algorithm~\ref{alg:simp}. Observe that this change can only affect weighted copies of $H$ which use the $j$-cell $Y$. Moreover, $t \geq t_0^*
  \geq 2/\eps_k^3$ by definition of $t_0^*$, and therefore $\eps_k
  t \geq 2/\eps_k^2 \geq 2h^2$. Consequently we have
  \begin{equation}\label{eq:conc:1}
    c_Y\le\frac{|X|^{h-j}}{\binom{|X|}{h}\cdot h!} \leq
    \frac{t^{h-j}}{\binom{\eps_k t}{h}\cdot h!} \leq
    \frac{ t^{h-j}}{(\eps_k t - h)^h} 
    \leq
     2\eps_k^{-h} t^{-j}
  \end{equation}
  We therefore have
  \begin{equation}\label{eq:conc:2}
   \sum_{j=2}^{k-1}\sum_{Y\in\binom{X}{j}}c_Y^2\leByRef{eq:conc:1} \sum_{j=2}^{k-1}\binom{t}{j}\big(2\eps_k^{-h} t^{-j}\big)^2\le
   \sum_{j=2}^{k-1}4\eps_k^{-2h} t^{-j}\le 4k\eps_k^{-2h}t^{-2}
  \end{equation}

  Hence using Theorem~\ref{thm:azuma}, the probability that
  $f(\mathbf{x})$ differs from $\Exp f(\mathbf{x})$ by more than $\eps_k/2$
  is at most
\begin{align*}
  2\exp\Big(\frac{-\eps_k^2/4}{8k\eps_k^{-2h}t^{-2}}\Big) &= 
  2\exp\Big(\frac{-\eps_k^{2+2h}t^2}{32k}\Big) \leq
  2\exp\Big(\frac{-4t}{\eps_k^k} \Big) \\
  & < \exp\Big( \frac{-1}{\eps_k^k} \Big)\exp \Big(-2t\Big) <
  2^{-\binom{1/\eps_k}{k}}2^{ -2t},
\end{align*}
  where the first inequality holds since $h\le 1/\eps-k$ and $t \geq t_0^* \geq 128k^2/\eps_k^{2+2/\eps_k+k}$ by definition of $t_0^*$.
\end{claimproof}

We can now prove that property~\ref{simpreg:a} holds with high probability by taking a union bound over the $s$ different $k$-graphs $G_i$, the at most
  $2^{t}$ choices of $X$ and the $2^{\binom{1/\eps_k}{k}}$ choices of $H$.
  Applying Claims~\ref{lem:expect} and~\ref{lem:concentrate} we deduce that with
  probability at least $1 - s2^{-t}$ we
  have that $d_H\big(G_i[\bigcup X]\big)=d_H\big(R(G_i)[X]\big) \pm \eps_k$ for
  all $1 \leq i \leq s$, all $k$-graphs $H$ with at most $1/\eps_k$ vertices and all sets $X$ of at least $\eps_k t$ clusters.

\smallskip

Next we show that~\ref{simpreg:b} holds with high probability also.
This is achieved through a similar but simpler argument as
for~$\ref{simpreg:a}$, so we will be brief. Note first that since each part
of $\Qart$ has equal size, any set $X$ of clusters for which $\bigcup X$ is the union of some parts of $\Qart$ has size at least $t/q$. So fix some $1 \leq j \leq k-1$, a set
$X$ of at least $t/q$ clusters of $\cJ$, a $j$-set $Y$ of clusters
of $\cJ$, and some $k$-graph $G$ on $V$. We first prove an analogue of
Claim~\ref{lem:expect}, that the expected value of $\reldeg(Y;R_\cJ(G), X)$
conditioned upon the choices of all cells on subsets of $Y$, is within
$\eps_k/2$ of $\reldeg(\cJ_Y; G, \bigcup X)$.
To prove this we choose uniformly at random an edge $\{v_1, \dots, v_j\}$ of
$\cJ_Y$ (which is fixed since we are conditioning on the choice of cell for
$Y$), and also choose uniformly at random vertices $v_{j+1}, \dots, v_k$ from $\bigcup
X$ so that the vertices $v_i$ with $1\le i\le k$ are all distinct. Let $e = \{v_1, \dots, v_k\}$.
By definition $\reldeg(\cJ_Y; G, \bigcup X) = \Prob(e \in G)$. Let
$\CROSS$ be the event that $e$ is $\Part$-partite. Using the fact that
the set $\{v_1,\ldots,v_j\}$ is automatically $\Part$-partite, 
we have, similarly as
in the proof
of Claim~\ref{lem:expect}, that
\begin{equation*}
 \Prob(\CROSS) \geq 1 - \frac{\binom{|X|m-j}{k-j-1}km}{\binom{|X|m-j}{k-j}}\ge 1-\frac{k^2m}{|X|m-k+1}\ge 1-\frac{\eps_k}{10}\,,
\end{equation*}
and therefore 
\begin{equation}\label{eq:simpreg:degexpect}\reldeg(\cJ_Y; G, \bigcup X) = \Prob(e \in G \mid \CROSS) \pm
\frac{\eps_k}{4}\,.
\end{equation}
For each $i \in [j]$ let $C_i$ be the cluster containing $v_i$. Now fix any $(k-j)$ distinct clusters 
$C_{j+1}, \dots, C_{k}$, and define $\mathcal{C}$ to be the $k$-tuple $(C_1, \dots, C_k)$. 
Recall that $\cJ[\mathcal{C}]$ was defined to be the subcomplex of $\cJ$ consisting of all 
cells of $\cJ$ which are supported on the clusters of $\mathcal{C}$. By definition of 
relative density and Fact~\ref{fact:densecount}, the number of edges in $E_\mathcal{C}$
corresponding to any given choice of $\cJ[\mathcal{C}]$ is 
\[\reld_{\cJ[\mathcal{C}]}(\mathcal{C}) \big|K_k(\hat{\cJ}_\mathcal{C})\big| = \big(1 \pm \tfrac{\eps_k}{10}\big)
\reld_{\cJ[\mathcal{C}]}(\mathcal{C}) \prod_{i=2}^{k-1} d_i^{\binom{k}{i}} m^k\,,\]
Let $E_{\mathcal{C}}$ consist of the edges $f \in G$ which have precisely one vertex in each cluster of $\mathcal{C}$ and which satisfy $f \cap Y \in \cJ_Y$. Then for each fixed $\mathcal{C}$, summing over all possibilities for $\cJ[\mathcal{C}]$ we have
\begin{equation*}\begin{split}
\big|E_\mathcal{C}\big|&=\sum_{\cJ[\mathcal{C}]}\big(1 \pm \tfrac{\eps_k}{10}\big)\reld_{\cJ[\mathcal{C}]}(\mathcal{C})
\prod_{i=2}^{k-1} d_i^{\binom{k}{i}} m^k\\
&=\big(1 \pm \tfrac{\eps_k}{10}\big)\Exp_{\cJ[\mathcal{C}]}\reld_{\cJ[\mathcal{C}]}(\mathcal{C})
\prod_{i=2}^{j} d_i^{\binom{j}{i}} m^k\\
&=\big(1 \pm \tfrac{\eps_k}{10}\big)\Exp_{\cJ}\reld_{\cJ}(\mathcal{C})
\prod_{i=2}^{j} d_i^{\binom{j}{i}} m^k\,. 
\end{split}\end{equation*}
The second line comes from observing that by Claim~\ref{claim:numcell},
$E_\mathcal{C}$ is partitioned into $\prod_{i=2}^{k-1}
d_i^{-\left(\binom{k}{i} - \binom{j}{i}\right)}$
 sets according to the choice
of $\cJ[\mathcal{C}]$, and the distribution over choices of $\cJ[\mathcal{C}]$ given by $\cJ$ is uniform, which also gives the third line. The total number of possibilities for $e$ which give $\mathcal{C}$ is $|\cJ_Y|m^{k-j}$, and by Fact~\ref{fact:densecount}
we have
\[|\cJ_Y|m^{k-j} = \big(1 \pm \tfrac{\eps_k}{10}\big)
\prod_{i=2}^{j} d_i^{\binom{j}{i}} m^k\,.\]
We conclude that, writing $\CROSS_\mathcal{C}$ for the event that $\phi(v_i) \in C_i$ for each $i \in [k]$, we have
 \[\Prob(e \in G \mid \CROSS_\mathcal{C}) =\frac{\big|E_\mathcal{C}\big|}{|\cJ_Y|m^{k-j}}= \big(1 \pm \tfrac{\eps_k}{4}\big) \Exp_\cJ
[\reld_\cJ(\mathcal{C})]\,,\] 
We now condition on the event $\CROSS$. This implies that there is precisely one choice of 
distinct clusters $C_{j+1}, \dots, C_k$ of $X \setminus Y$ for which the event 
$\CROSS_{\mathcal{C}}$ occurs; we now take $\mathcal{C}$ to be given by this choice. 
Since $v_{j+1}, \dots, v_k$ were chosen uniformly at random, and each cluster has equal size, 
it follows that the $(k-j)$-set $\{C_{j+1}, \dots, C_k\}$
is chosen uniformly at random from all $(k-j)$-sets in $X \setminus Y$. This implies that
$$\Prob(e \in G \mid \CROSS) = \Exp_\mathcal{C} \Exp_\cJ
[\reld_\cJ(\mathcal{C})] \pm\eps_k/4.$$ 
Furthermore, by definition we have $\Exp_\mathcal{C} (\reld_{\cJ}(\mathcal{C})) = \reldeg(Y; R_{\cJ}(G), X),$ so 
\[\Exp_\cJ [\reldeg(Y; R_\cJ(G), X)] = \Prob(e \in G \mid \CROSS) \pm
\tfrac{\eps_k}{4}\eqByRef{eq:simpreg:degexpect}\reldeg(\cJ_Y;G,\bigcup X)\pm\tfrac{\eps_k}{2}\,,\] proving
the analogue of Claim~\ref{lem:expect}.

\smallskip

Next, we prove an analogue of Claim~\ref{lem:concentrate}, namely that for any
$k$-graph $G$ on $V$, any $1 \leq j \leq k-1$, any $j$-set $Y$ of clusters
of $\cJ$ and any set $X$ of at least $t/q$ clusters of $\cJ$, with high probability the random variable $\reldeg(Y; R_\cJ(G), X)$ conditioned on the choice of all cells on subsets of $Y$ is
within $\eps_k/2$ of its expectation conditioned on the choice of all cells
on subsets of $Y$.
This argument is very similar to the proof of Claim~\ref{lem:concentrate}. As there, we can write $\reldeg(Y; R_\cJ(G), X)$ conditioned on the choice of all cells on subsets of $Y$ as a function $f'(\mathbf{x})$ where the entries of $\mathbf{x}$ are integers corresponding to the choice of $i$-cell on each $i$-set in $X\cup Y$ not contained in $Y$, for each $2\le i\le k-1$. Again, for each $Z\subset X\cup Y$ of size between $2$ and $k-1$ not contained in $Y$, we let $c_Z$ bound the difference $\big|f'(\mathbf{x})-f'(\mathbf{x}')\big|$ for pairs of vectors differing only on the entry corresponding to $Z$. The value of $f'(\mathbf{x})$ is the average value of $\reld_\cJ(\mathcal{C})$ for sets $\mathcal{C}$ of clusters in $X$ containing $Y$, and changing the cell on $Z$ affects only those containing $Y\cup Z$, so we have
\[c_Z\le \frac{\binom{|X\setminus Z|}{k-|Y\cup Z|}}{\binom{|X\setminus Y|}{k-|Y|}}\le 2k!|X|^{-|Z\setminus Y|}\,,\]
where the last inequality is obtained by using $|X|\ge t/q$ and the choice of $t\ge t_0^*\ge 10q^kk^k\eps_k^{-1}$ in~\eqref{simpreg:etas}. Note that $|Z\setminus Y|\ge 1$ by assumption, so summing over all choices of $Z$ (where we write $|Z\setminus Y|=j$) we have
\[\sum_Z c_Z^2\le\sum_{j=1}^{k-1}2^{|Y|}\binom{|X|}{j}4k!^2|X|^{-2j}\le 4k!^2k2^{|Y|}|X|^{-1}\le \frac{4k!^2k2^kq}{t}\,,\]
where the final inequality uses $|X|\ge t/q$ and $|Y|\le k$.

By Theorem~\ref{thm:azuma} we find that the probability that
$\reldeg(Y;R(G), X)$ fails to be within $\eps_k/2$ of $\Exp_\cJ [\reldeg(Y; R(G), X)]$ (and therefore also the probability that $\reldeg(\cJ_Y; G, \bigcup
X)$ differs from $\reldeg(Y;R(G), X)$ by more than $\eps_k$)
is at most
\[ 2\exp\Big(\frac{-\eps_k^2 t}{32k!^2k2^kq}\Big)\leByRef{simpreg:etas}2^{-2\sqrt{t}}
\leByRef{simpreg:etas} 2^{-\sqrt{t}}/(2^qs)\,.\]

For property~\ref{simpreg:b} there are only at most $2^q$ possibilities
for $X$ (since we required that $\bigcup X$ is a union of parts of $\Qart$), and
at most $\binom{t}{1}+\cdots+\binom{t}{k-1} \leq t^{k-1}$
possibilities for $Y$. So we can take a union bound over all $G=G_1,\ldots,G_s$, all
sets $Y$ of between $1$ and $k-1$ clusters of $\cJ$, and all sets $X$ of
clusters such that $\bigcup X$ is a union of parts of $\Qart$, to deduce that
property~\ref{simpreg:b} holds with probability at least $1 - t^{k-1}
2^{-\sqrt{t}}$.

\smallskip

To complete the proof, we need to show that property~\ref{simpreg:c} holds with high
probability; the argument for this splits into two parts. The `probabilistic' part is to show that $\cJ$ takes about
the same fraction of \emph{every} slice (regular or otherwise) through
$\tilde{\Part}^*$. Let us briefly sketch how this helps us. We want to show that the density of rooted copies of $H$ is about equal to the density of rooted copies of $H$ supported on $\cJ$. The former density is easily seen to be close to the average over slices $\tilde{\cJ}$ through $\tilde{\Part}^*$ of the density of rooted copies of $H$ supported on $\tilde{\cJ}$. Now Lemma~\ref{lem:strReg}\ref{strReg:spread} implies that if $\tilde{\cJ}$ is any slice through $\tilde{\Part^*}$, and $X$ is any $\tilde{\Part}$-consistent collection of $\tilde{t}$ clusters of $\Part$, then the density of rooted copies of $H$ supported on $\tilde{\cJ}$ is close to the density of rooted copies of $H$ supported on $\tilde{\cJ}[X]$. Finally, since $\cJ$ contains about the same number of subcomplexes of the form $\tilde{\cJ}[X]$ for each $\tilde{\cJ}$, the density of rooted copies of $H$ supported on $\cJ$ is close to the average over $\tilde{\cJ}$ of the density of rooted copies of $H$ supported on $\tilde{\cJ}$, which is what we want to show.

The proof of the following `probabilistic part' follows the same pattern as we saw for parts~\ref{simpreg:a} and~\ref{simpreg:b}, showing that a certain random variable has the desired expectation, and then using McDiarmid's inequality to establish concentration. We remark that the latter depends on the fact that the number of clusters of $\cJ$ is much greater than the number of slices through $\tilde{\Part}^*$ by choice of $p(\cdot)$.

\begin{claim}\label{clm:equalchoose}
  With probability at least $1-2^{-t}$, for each slice $\tilde{\cJ}$ through
  $\tilde{\Part}^*$, the number of sets $X$ of $\tilde{t}$ clusters of $\Part^*$, one
  in each cluster of $\tilde{\Part}$, such that $\cJ[X]\subseteq\tilde{\cJ}$, is
  \begin{equation}\label{equalchoose}(1\pm\tfrac{\eps_k}{16})p(t_1^*)^{\tilde{t}}\prod_{j=2}^{k-1}d_j^{\binom{\tilde{t}}{j}}
  \,.
  \end{equation}
\end{claim}

\begin{claimproof}
  Let $\tilde{\cJ}$ be a slice through $\tilde{\Part}^*$. If $X$ is a fixed set of $\tilde{t}$
  clusters of $\Part^*$, one in each cluster of $\tilde{\Part}^*$, then, by Lemma~\ref{lem:strReg}\ref{strReg:genP},
  $\cJ[X]\subseteq \tilde{\cJ}$ exactly if for each $2\le j\le k-1$ and
  set $Y$ of $\binom{X}{j}$, we happened to choose the one $j$-cell on $Y$ which
  is a subset of the corresponding $j$-cell in $\tilde{\cJ}$. Conditioning on having
  done this for $j'$-cells for $j'<j$, the probability of doing so for any given
  $Y$ is $d_j$ by Claim~\ref{claim:numcell}, and these choices for different members of $\binom{X}{j}$ are
  independent. It follows that the probability that $\cJ[X]\subseteq\tilde{\cJ}$ is
  \[\prod_{j=2}^{k-1}d_j^{\binom{\tilde{t}}{j}}\,.\]
  The total number of choices of $X$ is $p(t_1^*)^{\tilde{t}}$, since
  each cluster of $\tilde{\Part}^*$ is split into $p(t_1^*)$ clusters of
  $\Part^*$ by Lemma~\ref{lem:strReg}\ref{strReg:numP}. By linearity of expectation, the expected number of sets $X$ such
  that $\cJ[X]\subseteq\tilde{\cJ}$ is
  \[p(t_1^*)^{\tilde{t}}\prod_{j=2}^{k-1}d_j^{\binom{\tilde{t}}{j}}\]
  as desired.
  \
  Similarly as for parts~\ref{simpreg:a} and~\ref{simpreg:b}, we can write the number of $\tilde{\Part}$-consistent sets $X$ of $\tilde{t}$ clusters of $\Part^*$ such that $\cJ[X]\subseteq \tilde{\cJ}$ in the form $f''(\mathbf{x})$, where $\mathbf{x}$ represents the choice of $j$-cell on $Y$ for each $Y\in\binom{\Part}{j}$ and each $2\le j\le k-1$. As before, we let $c_Y$ be the maximum change in $f''(\mathbf{x})$ which can be obtained by changing the choice of cell on $Y$. Observe that if $Y$ contains two or more clusters from the same part of $\tilde{\Part}$ then we have $c_Y=0$, while otherwise changing the cell on $Y$ can only affect those $\Part^*$-consistent sets $X$ such that
  $Y\subseteq X$, of which there are exactly $p(t_1^*)^{\tilde{t}-|Y|}$. So we have
  \[\sum_{Y}c_Y^2\le \sum_{j=2}^{k-1}\binom{\tilde{t}}{j}p(t_1^*)^j\big(p(t_1^*)^{\tilde{t}-j}\big)^2\le
  2^{\tilde{t}}p(t_1^*)^{2\tilde{t}-2}\]
  and hence by Theorem~\ref{thm:azuma}, the probability that the number of sets $X$ with
  $\cJ[X]\subseteq\tilde{\cJ}$ deviates by more than
  \[\tfrac{\eps_k}{16}
  p(t_1^*)^{\tilde{t}}\prod_{j=2}^{k-1}d_j^{\binom{\tilde{t}}{j}}\]
  from its expectation is at most
  \[2\exp\Big(\frac{-\eps_k^2
  p(t_1^*)^{2\tilde{t}}\prod_{j=2}^{k-1}d_j^{2\binom{\tilde{t}}{j}}}{2\cdot
  16^2\cdot 
  2^{\tilde{t}}p(t_1^*)^{2\tilde{t}-2}}\Big)
  \leq
  2\exp\left(-2^{-t_1^*-9}\eps_k^2p(t_1^*)^2\prod_{j=2}^{k-1} 
  \left(\frac{1}{t_1^*}\right)^{2\binom{t_1^*}{j}}
  \right)\,,\]
where the inequality follows from Lemma~\ref{lem:strReg}\ref{strReg:numP} and~\ref{strReg:equit}, which imply that $d_j\ge 1/t_1^*$ for each $j$ and $\tilde{t} \leq t_1^*$.
Since we chose $p(\cdot)$ to satisfy~\eqref{simpreg:p} (which we apply with $t_1^*$ in place of $x$), and we have $t\le t_1=p(t_1^*)t_1^*$ by Lemma~\ref{lem:strReg}\ref{strReg:numP}, 
  this
  probability is smaller than
  \[2^{-p(t_1^*)t_1^*} \prod_{j=2}^{k-1} \left(\frac{1}{t_1^*}\right)^{\binom{t_1^*}{j}} \leq 2^{-t}\prod_{j=2}^{k-1}d_j^{\binom{\tilde{t}}{j}}\,.\]
  We can therefore take a union bound over the (by Claim~\ref{claim:numcell}) $\prod_{j=2}^{k-1}d_j^{-\binom{\tilde{t}}{j}}$ choices of $\tilde{\cJ}$ to obtain
  the desired conclusion.
\end{claimproof}
 
We want to show that if the likely event of Claim~\ref{clm:equalchoose} holds, then using Lemma~\ref{lem:strReg}\ref{strReg:spread} we can deduce property~\ref{simpreg:c}. So fix $G=G_i$ for some $i \in [s]$. Now for any fixed $1\le\ell\le 1/\eps_k$, any fixed $k$-graph $H$ equipped with a set of distinct root vertices
  $x_1,\ldots,x_\ell$ such that $h=v(H)\le 1/\eps_k$ and any fixed set of distinct vertices
  $v_1,\ldots,v_\ell$ in $V$, we would like to show that the density  $d_H(G_i;v_1,\ldots,v_\ell)$ of rooted
  $H$-copies in $G$ is within $\eps_k$ of the density $d_H(G_i;v_1,\ldots,v_\ell,\cJ)$ of rooted $H$-copies in
  $G$ supported on $\cJ$. Recall that the vertex set of $\Hskel$ consists of all vertices of $H$ except for $x_1, \dots, x_\ell$. Choose uniformly at random a map $\psi: V(\Hskel) \to V(G)$, and define $\psi': V(H) \to V(G)$ by taking $\psi'(x_j) = v_j$ for any $j \in [\ell]$ and $\psi'(x) = \psi(x)$ for any $x \in V(\Hskel)$ (note carefully that, unlike for the previous cases, we do not insist that $\psi$ is injective, and that even if $\psi$ is injective the same may not be true of $\psi'$). Let $\INJ$ denote the event that $\psi'$ is injective, and let $\HOM$ denote the event that $\psi'$ is an embedding of $H$ into $G$ (so $\HOM$ is a subset of $\INJ$). Next, let $\CROSS_{\Part}$ denote the event that $\psi$ is $\Part$-crossing, meaning as before that each vertex of $\Hskel$ is mapped to a distinct part of $\Part$, and similarly let $\CROSS_{\tilde\Part}$ denote the event that $\psi$ is $\tilde{\Part}$-crossing. Finally, for any $(k-1)$-complex $\cL$ on $V$ let $\SLICE_{\cL}$ denote the event that $\psi$ is an embedding of $\Hskel$ into $\cL$. 
Then by definition we have
\begin{equation} \label{simpreg:cdens} 
d_H(G_i;v_1,\ldots,v_\ell) = \Prob(\HOM \mid \INJ),
\end{equation}
and
\begin{equation} \label{simpreg:cdensJ}
d_H(G_i;v_1,\ldots,v_\ell,\cJ) = \Prob(\HOM \mid \SLICE_{\cJ} \cap \CROSS_\Part),
\end{equation}
and our aim is to show that these two conditional probabilities differ by at most~$\eps_k$. We shall frequently refer to events of the form $\SLICE_{\cL} \cap \CROSS_{\tilde\Part}$ for some $(k-1)$-complex $\cL$ on $V$, so for brevity we denote this event by $\SLICE_{\cL}^*$. 

Our first goal is to show that conditioning on $\CROSS_{\tilde\Part}$ instead of on $\INJ$ in~\eqref{simpreg:cdens} and on $\CROSS_\Part$ in~\eqref{simpreg:cdensJ} has an insignificant effect on the probabilities expressed in these equations, since these events are all highly probable (even after conditioning on $\SLICE_\cJ$). Indeed, similarly as before we have 
  \begin{equation}\label{simpreg:ccrossc}
   \Prob(\CROSS_{\tilde{\Part}}^c)
\le \frac{\binom{n}{h-\ell-1}h(n/\tilde{t})}{\binom{n}{h-\ell}}
= \frac{h(h-\ell)n/\tilde{t}}{n-h+\ell+1}\le \frac{2h^2}{t_0}
\leByRef{simpreg:etas} \frac{\eps_k}{100}\,.
  \end{equation}    
Now observe that $\psi$ can be obtained by selecting a uniformly-random image $\psi(x)$ for each $x \in V(\Hskel)$ in turn. The event $\INJ$ will occur unless for some $x$ we select $\psi(x)$ to be one of the at most $h-\ell$ previously-chosen images or one of the vertices $v_1, \dots, v_\ell$. Taking a union bound, it follows that for any $X \in \cX$ we have 
\begin{equation} \label{simpreg:cinjcbound}
\Prob(\INJ^c) 
\leq (h-\ell)\frac{h}{n} 
\leq \frac{h^2}{n} 
\leq \frac{\eps_k}{100}.
\end{equation}
By~\eqref{simpreg:ccrossc} we may apply Proposition~\ref{prop:condprob} with $\HOM$ and $\CROSS_{\tilde\Part}$ in place of $E$ and $C$ respectively, and by~\eqref{simpreg:cinjcbound} we may apply the same proposition with $\HOM$ and $\INJ$ in place of $E$ and $C$ respectively (in each case we take $S$ to be the event which always occurs and $\tfrac{\eps_k}{100}$ in place of $\eps$). This gives our approximation for~\eqref{simpreg:cdens}, namely
\begin{equation} \label{simpreg:capprox}
\Prob(\HOM \mid \INJ) = 
\Prob(\HOM) \pm \tfrac{\eps_k}{50} = 
\Prob(\HOM \mid \CROSS_{\tilde{\Part}}) \pm \tfrac{2\eps_k}{50}.
\end{equation}
We turn now to~\eqref{simpreg:cdensJ}, beginning with the next claim. We define $\cX$ to be the set of all $\tilde\Part$-consistent sets $X \in \binom{\Part}{\tilde{t}}$. (Recall that this means that $X$ is a set of $\tilde{t}$ clusters in $\Part$, one contained in each part of~$\tilde{\Part}$).

\begin{claim}\label{simpreg:cclaim}
Let $\cL$ be a slice through $\Part^*$. Then 
\begin{enumerate}[label=\abc]
\item $\Prob(\SLICE_{\cL}\mid\CROSS_{\Part})
= \big(1\pm\tfrac{\eps_k}{100}\big)\prod_{i=2}^{k-1}d_i^{e_i(\Hskel)}$,
\item $\Prob(\SLICE_{\cL}\mid\CROSS_{\tilde\Part})
= \big(1\pm\tfrac{\eps_k}{100}\big)\prod_{i=2}^{k-1}d_i^{e_i(\Hskel)}$, and
\item for any  $X \in \cX$ we have
$\Prob(\SLICE^*_{\cL[X]} \mid \SLICE^*_{\cL}) = \big(1\pm\tfrac{\eps_k}{20}\big)p(t_1^*)^{\ell-h}$. 
\end{enumerate}
Furthermore, for any slice $\tilde{\cJ}$ through $\tilde{\Part}^*$ we have $\Prob(\SLICE_{\tilde{\cJ}} \mid \CROSS_{\tilde{\Part}})
= \big(1\pm\tfrac{\eps_k}{100}\big)\prod_{i=2}^{k-1}d_i^{e_i(\Hskel)}$.
\end{claim}  

\begin{claimproof}
Identify the vertices of $\Hskel$ with the integers of $[h-\ell]$. Then for any $(h-\ell)$-tuple $\mathcal{C} = (C_1, \dots, C_{h-\ell})$ of distinct clusters of $\cL$ (i.e. parts of $\Part$), we say that a copy of $\Hskel$ in $\cL$ is \emph{$\mathcal{C}$-distributed} if for each $j \in [h-\ell]$ vertex $j$ of $\Hskel$ lies in cluster $C_j$ of $\cL$. Likewise, we say that $\psi$ is $\mathcal{C}$-distributed if $\psi(j) \in C_j$ for each $j \in [h-\ell]$, an event which we denote by $\DIST_\mathcal{C}$. Note that if we
condition on the event $\DIST_\mathcal{C}$, then for each $j \in [h - \ell]$ the image $\psi(j)$ is a uniformly-random vertex in $C_j$. It follows that $\Prob(\SLICE_\cL \mid \DIST_\mathcal{C})$ is equal to the number of $\mathcal{C}$-distributed labelled copies of $\Hskel$ in $\cL$ divided by $m^{h-\ell}$ (the number of possibilities for $\psi$ given that $\DIST_\mathcal{C}$ occurs).
So 
by Lemma~\ref{lem:count} we obtain the estimate 
\begin{equation} \label{simpreg:cslicegivendist} 
\Prob(\SLICE_{\cL} \mid \DIST_{\mathcal{C}})
=\big(1\pm\tfrac{\eps_k}{100}\big)\prod_{i=2}^{k-1}d_i^{e_i(\Hskel)}.
\end{equation} 
Since the events $\DIST_{\mathcal{C}}$ partition the event $\CROSS_\Part$, this implies (a). Similarly we obtain (b) since the event $\CROSS_{\tilde\Part}$ is partitioned by the events $\DIST_{\mathcal{C}}$ for those $(h-\ell)$-tuples $\mathcal{C}$ for which each $C_j$ is a subset of a distinct part of $\tilde{\Part}$. 
Now, fix $X \in \cX$, and let $\DIST_X$ be the event that $\psi$ maps each vertex of $\Hskel$ to a distinct member of $X$ (so $\DIST_X \subseteq \CROSS_{\tilde\Part}$).  Then the event $\DIST_X$ is partitioned by the events $\DIST_{\mathcal{C}}$ for $(h-\ell)$-tuples $\mathcal{C}$ of distinct sets in $X$, each of which is equally likely to occur. Summing~\eqref{simpreg:cslicegivendist} over all such $k$-tuples $\mathcal{C}$,we obtain 
\[\Prob(\SLICE^*_{\cL[X]} \mid \DIST_X) = \Prob(\SLICE^*_{\cL} \mid \DIST_X)
=\big(1\pm\tfrac{\eps_k}{100}\big)\prod_{i=2}^{k-1}d_i^{e_i(\Hskel)}\,.\]
Together with the fact that $\Prob(\DIST_X \mid \CROSS_{\tilde\Part}) = p(t_1^*)^{-(h-\ell)}$ (since $\psi$ was chosen uniformly at random), this gives 
\[\Prob(\SLICE^*_{\cL[X]}\mid\CROSS_{\tilde\Part})
=\big(1\pm\tfrac{\eps_k}{100}\big)p(t_1^*)^{-(h-\ell)}\prod_{i=2}^{k-1}d_i^{e_i(\Hskel)}\,.\]
Since $\SLICE^*_{\cL[X]} \subseteq \SLICE^*_{\cL}$, this equation and part (a) together prove (c). Finally, 
the proof of the final statement is almost identical to the proof of (a); the only differences are that we instead consider $k$-tuples $\mathcal{C}$ of parts of $\tilde\Part$, and that the term $m^{h-\ell}$ is replaced by $(n/\tilde{t})^{h-\ell}$ (but this term is, as above, cancelled by the output from Lemma~\ref{lem:count}).
\end{claimproof}

Note that $\Prob(\CROSS_\Part \cap \SLICE_\cJ) = \Prob(\SLICE_\cJ \mid \CROSS_\Part) \Prob(\CROSS_\Part) \leq \Prob(\SLICE_\cJ \mid \CROSS_\Part)$, and Claim~\ref{simpreg:cclaim} part~(a) gives an approximation for this probability. Similarly we have $\Prob(\CROSS_{\tilde\Part} \cap \SLICE_\cJ) = 
\Prob(\SLICE_\cJ \mid \CROSS_{\tilde\Part}) \Prob(\CROSS_{\tilde\Part})$, which we can approximate by Claim~\ref{simpreg:cclaim} part~(b) and~\eqref{simpreg:ccrossc}. So, since $\CROSS_{\tilde{\Part}} \subseteq \CROSS_\Part$, we obtain
\begin{align*}
\Prob(\CROSS_{\tilde{\Part}} \mid \CROSS_\Part \cap \SLICE_\cJ) & = \frac{\Prob (\CROSS_{\tilde{\Part}} \cap \SLICE_\cJ)}{\Prob(\CROSS_\Part \cap \SLICE_\cJ)} \\
&\geq \frac{\big(1 - \tfrac{\eps_k}{100}\big)\prod_{i=2}^{k-1}d_i^{e_i(\Hskel)} (1-\tfrac{\eps_k}{100})}
{\big(1 + \tfrac{\eps_k}{100}\big)\prod_{i=2}^{k-1}d_i^{e_i(\Hskel)}} \\
&\geq 1 - \tfrac{\eps_k}{20}.
\end{align*}
We may therefore apply Proposition~\ref{prop:condprob} with $\HOM, \CROSS_\Part \cap \SLICE_\cJ$ and $\CROSS_{\tilde{\Part}}$ in place of $E$, $S$ and $C$ respectively to obtain
\begin{equation*} 
\Prob(\HOM \mid \CROSS_\Part \cap \SLICE_\cJ) = \Prob(\HOM \mid \CROSS_{\tilde{\Part}} \cap \SLICE_{\cJ}) \pm \tfrac{\eps_k}{10}.
\end{equation*}
Combining this equation with~(\ref{simpreg:cdens}),~(\ref{simpreg:cdensJ}),~(\ref{simpreg:capprox}) and the definition of $\SLICE^*_{\cJ}$ as $\SLICE_{\cJ} \cap \CROSS_{\tilde{\Part}}$, we conclude that it is sufficient to prove that
  \begin{align}\label{simpreg:desired} 
\Prob(\HOM \mid \CROSS_{\tilde{\Part}}) 
& = \Prob(\HOM \mid \SLICE^*_{\cJ}) 
\pm \tfrac{\eps_k}{2}\,. 
  \end{align}

We begin with the left hand side of~\eqref{simpreg:desired}.
Recall that there are precisely $\prod_{j=2}^{k-1}d_j^{-\binom{\tilde{t}}{j}}$
slices $\tilde{\cJ}$ through $\tilde{\Part}^*$, and observe that
if the event $\CROSS_{\tilde{\Part}}$ occurs then by Claim~\ref{claim:numcell} there are precisely 
$\prod_{j=2}^{k-1}d_j^{e_j(\Hskel)-\binom{\tilde{t}}{j}}$
slices $\tilde{\cJ}$ through $\tilde{\Part}^*$ for which the event $\SLICE^*_{\tilde{\cJ}}$ occurs. 
By definition we have $\SLICE^*_{\tilde{\cJ}} \subseteq \CROSS_{\tilde{\Part}}$ for any such slice, so
summing over all slices $\tilde{\cJ}$ through $\tilde{\Part}^*$ we obtain
  \begin{multline}\label{simpreg:homtoavg}
   \Prob(\HOM\mid\CROSS_{\tilde{\Part}})
   =\prod_{j=2}^{k-1}d_j^{\binom{\tilde{t}}{j} - e_j(\Hskel)}
   \sum_{\tilde{\cJ}}\Prob(\HOM\mid\SLICE^*_{\tilde{\cJ}})\Prob(\SLICE^*_{\tilde{\cJ}}\mid\CROSS_{\tilde{\Part}})\, \\
= (1 \pm \tfrac{\eps_k}{100}) 
    \prod_{j=2}^{k-1}d_j^{\binom{\tilde{t}}{j}}  \sum_{\tilde{\cJ}}\Prob(\HOM\mid\SLICE^*_{\tilde{\cJ}})\,,
  \end{multline}  
where the final equality holds by the final part of Claim~\ref{simpreg:cclaim}.
  \smallskip  
 
 We next show that for any slice $\tilde{\cJ}$ through $\tilde{\Part^*}$ and any $X \in \cX$ the probabilities $\Prob(\HOM\mid\SLICE^*_{\tilde{\cJ}})$ and $\Prob(\HOM\mid\SLICE^*_{\tilde{\cJ}[X]})$ are roughly equal. To do this, fix any $X \in \cX$, and for any set $\tilde{Y}$ of $h-\ell$ parts of $\tilde{\Part}$, let $Y$ be the (unique) $\tilde{Y}$-consistent subset of $X$. Then for any choice of $\tilde{Y}$, Lemma~\ref{lem:strReg}\ref{strReg:spread} states that we have 
\[\Prob(\HOM \mid \SLICE^*_{\tilde{\cJ}[\tilde{Y}]}) = \Prob(\HOM\mid\SLICE^*_{\tilde{\cJ}[Y]}) \pm\eps^*_k\,.\]
Since the events $\SLICE^*_{\tilde{\cJ}[\tilde{Y}]}$ for $\tilde{Y} \in \binom{\tilde\Part}{h-\ell}$ partition the event $\SLICE^*_{\tilde{\cJ}}$, the events $\SLICE^*_{\tilde{\cJ}[Y]}$ partition the event $\SLICE^*_{\tilde{\cJ}[X]}$, and $\eps^*_k \leq \tfrac{\eps_k}{100}$ by \eqref{simpreg:epsks}, we may sum over all $\tilde{Y} \in \binom{\tilde{\Part}}{h-\ell}$ to obtain 
  \begin{equation}\label{simpreg:JtoJx}
   \Prob(\HOM\mid\SLICE^*_{\tilde{\cJ}}) = \Prob(\HOM\mid\SLICE^*_{\tilde{\cJ}[X]}) \pm \tfrac{\eps_k}{100}.
  \end{equation}
 
  \smallskip 
Turning to the right hand side of~\eqref{simpreg:desired}, we next show that $\Prob(\HOM\mid\SLICE^*_{\cJ})$ is close to the average of $\Prob(\HOM\mid\SLICE^*_{\cJ[X]})$ over all of the $p(t_1^*)^{\tilde{t}}$ sets $X \in \cX$.
Observe that the events $\SLICE^*_{\cJ[X]}$ for $X \in \cX$ cover
  $\SLICE^*_{\cJ}$, with each $\psi\in\SLICE^*_{\cJ}$
  covered exactly $p(t_1^*)^{\tilde{t}-h+\ell}$ times. We thus obtain
  \begin{multline}\label{simpreg:avgX}
      \Prob(\HOM\mid\SLICE^*_{\cJ}) 
    = p(t_1^*)^{-\tilde{t}+h-\ell} \sum_{X \in \cX} 
    \Prob(\HOM\mid\SLICE^*_{\cJ[X]})
    \Prob(\SLICE^*_{\cJ[X]} \mid \SLICE^*_{\cJ})\,\\
   =\Big(1 \pm \tfrac{\eps_k}{20}\Big) \Big( p(t_1^*)^{-\tilde{t}}
\sum_{X \in \cX}
\Prob(\HOM\mid\SLICE^*_{\cJ[X]})
   \Big) 
   \,. 
  \end{multline}
where the second equality holds by Claim~\ref{simpreg:cclaim} part~(c). 
     
  Finally, we want to connect the averages in~\eqref{simpreg:homtoavg}
  and~\eqref{simpreg:avgX}. Observe first that for any $X \in \cX$ there is precisely one slice $\tilde{\cJ}$ through $\tilde{\Part}^*$ such that $\cJ[X] \subseteq \tilde\cJ$, and for this $\tilde{\cJ}$ we have $\cJ[X] = \tilde\cJ[X]$, which implies
\begin{equation} \label{simpreg:cremovex}
\Prob(\HOM\mid\SLICE^*_{\cJ[X]})
=\Prob(\HOM\mid\SLICE^*_{\tilde{\cJ}[X]})
\eqByRef{simpreg:JtoJx} \Prob(\HOM\mid\SLICE^*_{\tilde{\cJ}})\pm \tfrac{\eps_k}{100}
\end{equation}
Since the last term does not depend on $X$,
 we only need to know, for each $\tilde{\cJ}$, how many sets $X \in \cX$ satisfy $\cJ[X] \subseteq \tilde{\cJ}$.  If $\cJ$
  satisfies the good event of Claim~\ref{clm:equalchoose}, then
the answer
  is given by~\eqref{equalchoose}, and so by summing~\eqref{simpreg:cremovex} over all $X \in \cX$ we obtain
\begin{multline*} 
\sum_{X \in \cX} \Prob(\HOM\mid\SLICE^*_{\cJ[X]}) 
= \sum_{\tilde{\cJ}} \sum_{X \in \cX :\,\cJ[X] \subseteq \tilde{\cJ}}
 \Big(\Prob(\HOM\mid\SLICE^*_{\tilde{\cJ}})\pm \tfrac{\eps_k}{100}\Big) \\
\eqByRef{equalchoose} \sum_{\tilde{\cJ}} (1\pm\tfrac{\eps_k}{16})p(t_1^*)^{\tilde{t}}\prod_{j=2}^{k-1}d_j^{\binom{\tilde{t}}{j}}  \big(\Prob(\HOM\mid\SLICE^*_{\tilde{\cJ}})\pm \tfrac{\eps_k}{100}\big)\,,
\end{multline*} 
where the sum is over all slices $\tilde{\cJ}$ through $\tilde{\Part^*}$. 
Together with~\eqref{simpreg:avgX}, this implies that
\begin{align*}
\Prob(\HOM \mid \SLICE^*_{\cJ}) \pm\tfrac{\eps_k}{20}
&= \sum_{\tilde{\cJ}} (1\pm\tfrac{\eps_k}{16}) 
\prod_{j=2}^{k-1}d_j^{\binom{\tilde{t}}{j}}  
\big(\Prob(\HOM\mid\SLICE^*_{\tilde{\cJ}}) \pm \tfrac{\eps_k}{100}\big) \\
&=
\Big(\prod_{j=2}^{k-1}d_j^{\binom{\tilde{t}}{j}}
\sum_{\tilde{\cJ}}
\Prob(\HOM\mid\SLICE^*_{\tilde{\cJ}}) \Big)\pm\tfrac{\eps_k}{8}\,,
\end{align*}
which together with~\eqref{simpreg:homtoavg}
implies~\eqref{simpreg:desired}. This completes the proof that if $\cJ$ has the good event of Claim~\ref{clm:equalchoose} then it satisfies property~\ref{simpreg:c} of Lemma~\ref{lem:simpreg}.

  \smallskip 
  
Finally, we conclude that the probability that $\cJ$ satisfies all the
conclusions of 
Lemma~\ref{lem:simpreg} is at least $1/2 -
s2^{-t}-t^{k-1}
2^{-\sqrt{t}}-2^{-t}$.
Since $t \geq t_0^*$, by choice of
$t_0^*$ in~\eqref{simpreg:etas} this probability is strictly greater than zero, and so some $\cJ$ exists as required.
\end{proof}

\section{Embedding tight cycles} \label{sec:embedding} 

Our aim in this section is to prove Lemma~\ref{lem:emb}. We start with some definitions.

Recall the definition of a tight walk $W$ from Section~\ref{sec:emb}. As we did for tight paths, we define the length $\ell(W)$ of $W$ to be the number of edges in $W$ (i.e.~the number of consecutive $k$-sets of vertices). We refer to the first $s$ vertices of $W$, ordered as they appear in $W$, as the \emph{initial $s$-tuple} of $W$, and similarly to the final $s$ vertices as the \emph{terminal $s$-tuple} of $W$. Initial and terminal $(k-1)$-tuples have particular importance, as given tight walks $W$ and $W'$ for which the terminal $(k-1)$-tuple of $W$ is identical to the initial $(k-1)$-tuple of $W'$, we may \emph{concatenate} $W$ and $W'$ to form a new tight walk, which we denote $W + W'$; when doing so, we always include the common $(k-1)$-tuple only once, so the edges of $W + W'$ are precisely the edges of $W$ followed by the edges of $W'$. In particular, we have $\ell(W + W') = \ell(W) + \ell(W')$. Note that this definition includes the case where $W$ and $W'$ are tight paths, in which case $W + W'$ is also a tight path provided that $W$ and $W'$ have no common vertices outside this common $(k-1)$-tuple.

Before giving the full proof of Lemma~\ref{lem:emb}, we briefly sketch our approach. We show that
for any $k$-tuple $X = (X_1, \dots, X_k)$ of clusters of $\cJ$ which forms an
edge of $R := R_{d_k}(G)$, we can find a long tight path `winding around' the $k$ clusters in $G[\bigcup_{i \in [k]}
X_i]$ by repeated use of the Extension Lemma. Indeed, we proceed in steps, at
each time $j \geq 1$ keeping track of a tight path $P^{(j)}$ and a set
$\paths^{(j)}$ of `possible extensions' of $P^{(j)}$: the latter are short tight
paths whose initial $(k-1)$-tuples are equal to the terminal $(k-1)$-tuple of
$P^{(j)}$, and whose terminal $(k-1)$-tuples are all distinct. The
Extension Lemma tells us that at each time step $j$ there must
exist some $P \in \paths^{(j)}$ so that, taking $P^{(j+1)}$ to be $P^{(j)} + P$, there
is a family $\paths^{(j+1)}$ which are possible extensions of $P^{(j+1)}$ (this
argument is formalised in Claim~\ref{clm:emb}). This procedure, which we refer to as `filling the edge $X$', can be continued until only few vertices remain in each cluster of $X$.

Of course, for large $\ell$ we need to cover more vertices than are contained in
any $k$ clusters; this is the point at which the connectedness of the fractional
matching comes into play. By assumption we can find a tight walk $W$ in
$R$ which visits every edge $e \in R$ whose weight $w_e$ is non-zero. Starting
at some such edge $e$, we proceed to `fill' this edge as described above,
stopping when our tight path $P^{(j)}$ covers around a $w_e$-proportion of the
vertices of each cluster. We then extend $P^{(j)}$ by `traversing the walk $W$' to the next edge $e'$ of non-zero weight. Since this extension of $P^{(j)}$ is short, few
vertices are used in total in walk-traversing steps. So the final proportion of vertices
covered by our tight path in any cluster $X_i$ is approximately the sum of the
weights of edges of $R$ containing $X_i$. Overall, this gives a path covering
sufficiently many vertices for the bound given in the lemma. To obtain a shorter
tight path, we simply stop `filling' each edge at an earlier stage.

Finally, it remains to `join the ends' of our tight path to form a tight cycle.
For this, at the very start of the argument we set aside some large subsets
$Z_1, \dots, Z_k$ of some clusters $X_1, \dots, X_k$ which form an edge of $R$.
We then choose our first path $P^{(0)}$ so that, as well as there being many
extensions $\paths^{(0)}$ suitable for use to form $P^{(1)}$, there are also
many $(k-1)$-tuples $f$ in $Z_1, \dots, Z_{k-1}$ which can be extended to the
initial $(k-1)$-tuple of $P^{(0)}$.

We now present the full details of the proof.

\begin{proof}[Proof of Lemma~\ref{lem:emb}]
We set
\begin{equation}\label{emb:constants}
  \alpha=\tfrac{\psi}{5}\quad\text{ and }\quad\beta=\tfrac{1}{200}\,.
\end{equation}
We will want to apply Lemma~\ref{restriction} to $k$-partite $k$-complexes
(i.e.~with $s=k$) with $\alpha$ as given in~\eqref{emb:constants} and with 
$\eps_k$ playing the same role here as there. Also, we will
want to apply Lemma~\ref{lem:ext} to both $k$-partite and
$(2k-1)$-partite $k$-complexes (i.e.~with $s=k$ and with $s=2k-1$) whose
top levels have relative density at least $d_k$, with each choice of $1\le b'<b\le 3k$, 
with $\beta$ as given in~\eqref{emb:constants}, and with $\sqrt{\eps_k}$ in place of $\eps_k$. 
We require $\eps_k > 0$ to be small enough for each of these applications.

Given $d_2,\ldots,d_{k-1}$, the various applications of Lemma~\ref{lem:ext}
mentioned above require various sufficiently small positive values for $c$. We take $c>0$ to be the minimum of these values (we will have a bounded number of choices of parameters, hence this minimum is well-defined).

We now require $\eps<c^2$ to be small enough for each of the applications
of Lemmas~\ref{restriction} and~\ref{lem:ext} mentioned above to
$k$-complexes whose underlying $(k-1)$-complex is
$(\cdot,\cdot,\sqrt{\eps})$-equitable with density vector
$\mathbf{d}=(d_{k-1},\ldots,d_2)$. In addition, we require 
$\eps\le\tfrac12\prod_{i=2}^{k-1}d_i^{\binom{k-1}{i}}$ to be
small enough that we can apply Fact~\ref{fact:densecount} with $\tfrac{1}{2}$ in place of $\beta$ to
$(\cdot,\cdot,\sqrt{\eps})$-equitable $k$-partite $k$-complexes
with density vector $\mathbf{d}$.
We require $r$ to be large enough for the above mentioned applications of
Lemma~\ref{lem:ext}. We also choose $m_0\ge 16(k-1)/\eps$ to be large enough so that any $m\ge
\alpha m_0$ is acceptable for all of these applications.

Given $t$, we set
\begin{equation}\label{emb:n0}
  n_0=t\cdot\max\Big(m_0,\frac{200k}{\eps},\frac{8k}{\alpha\sqrt{\eps}},\frac{2(k+1)t^{2k}}{\alpha}\Big)\,.
\end{equation}

Now let $G$ be an $n$-vertex $k$-graph, where $n \geq n_0$, and let $\cJ$ be a $(\cdot,\cdot,\eps,\eps_k,r)$-regular slice for $G$ with $t$ clusters and density vector $\mathbf{d}$. Let $R:=R_{d_k}(G)$, and let $m := n/t$, so each cluster of $\cJ$ has size $m$. We
write $\cG$ for the $k$-complex obtained from $\cJ$ by adding all edges of $G$ supported on
$\cJ^{(k-1)}$ as the `$k$-th level' of $\cG$. So for any edge $X \in R$,
$\cG[\bigcup_{X' \in X} X']$ is a $(\reld(X), d_{k-1}, \dots, d_2, \eps, \eps_k,
r)$-regular $k$-partite
$k$-complex with $\reld(X) \geq d_k$. Furthermore, for convenience of notation, for any $s$-tuple $X = (X_1,
\dots, X_s)$ of clusters of $\cJ$ and any subsets $Y_j \subseteq X_j$ for $j \in
[s]$  we write $\cG(Y_1, \dots, Y_s)$ for the $s$-partite $s$-graph
$\cG_X[\bigcup_{j \in [s]} Y_j]$, that is, whose edges are the edges of
$\cG^{(s)}$ with one vertex in each $Y_j$. In addition, we say that an
$s$-tuple $(v_1, \dots, v_s)$ of vertices of $G$ is an \emph{ordered edge of} $\cG(Y_1, \dots, Y_s)$
if $\{v_1, \dots, v_s\}$ is an edge of $\cG$ and $v_j \in Y_j$ for each $j \in
[s]$.

Since $\cJ$ is a regular slice for $G$, for any set $X$ of $k$ clusters $X_1, \dots, X_k$ 
of $\cJ$ the $k$-partite $(k-1)$-complex $\cJ[\bigcup_{j \in [k]} X_j]$ is 
$(\bf{d}, \eps, \eps, 1)$-regular. By adding all sets of $k$ vertices supported on the polyad $\hat{\cJ}_X$ as 
a `$k$-th level', we may obtain a $(1, d_{k-1}, \dots, d_2, \eps, \eps_k, r)$-regular 
$k$-partite $k$-complex, whose restriction to any subsets $Y_j \subseteq X_j$ of size 
$|Y_j| = \alpha m$ for each $j \in [k]$ is then 
$(1, d_{k-1}, \dots, d_2, \sqrt{\eps}, \sqrt{\eps_k}, r)$ by Lemma~\ref{restriction}. 
We conclude by Fact~\ref{fact:densecount} that for any subsets $Y_1,\ldots,Y_{k-1}$ of 
distinct clusters of $\cJ$, each of size $\alpha m$, we have
\begin{equation}\label{eq:emb:numedge}
e\big(\cG(Y_1,\ldots,Y_{k-1})\big) \geq \eps m^{k-1}\,.
\end{equation}

The heart of our embedding lemma is the following claim. We will use it in steps when we fill an edge of $R$ with $i=0$, and in walk-traversing steps with $i=1$.

 \begin{claim}\label{clm:emb} Let $\{X_1, \dots, X_k\}$ be an edge of $R$, and
 choose any $Y_j \subseteq X_j$ for each $j \in [k]$ so that $|Y_1| = \dots =
 |Y_k| = \alpha m$. Let $\paths$ be a collection of at least $\tfrac12\,e(\cG(Y_1,
 \dots, Y_{k-1}))$ tight paths in $G$ (not necessarily contained in $\bigcup_{j
 \in [k]} Y_j$) each of length at most $2k+1$ and whose terminal $(k-1)$-tuples
 are distinct members of $\cG(Y_1,\ldots,Y_{k-1})$. Then for each $i \in \{0,
 1\}$ there is a path $P\in\paths$ and a collection $\paths'$ of
 $\tfrac{9}{10}\,e(\cG(Y_{i+1},\ldots,Y_{i + k-1}))$ tight paths in $G$, each of length
 $k+i$, all of whose initial $(k-1)$-tuples are the terminal $(k-1)$-tuple of
 $P$, whose terminal $(k-1)$-tuples are distinct members of
 $\cG(Y_{i+1},\ldots,Y_{i+k-1})$, and where the $j$th vertex of each path in
 $\paths'$ lies in $Y_j$ and, if $j\ge k$, is not contained in $P$.
 \end{claim}
 
 One might expect to prove this claim by trying to extend the terminal $(k-1)$-tuple of a path in $\paths$ in many different ways using Lemma~\ref{lem:ext}. But it turns out to be hard to show that these many different ways really go to many different terminal $(k-1)$-tuples, and so what we actually do is show the stronger statement that most pairs of disjoint $(k-1)$-tuples are joined by many paths. That is, we apply Lemma~\ref{lem:ext} with $\mathcal{H}$ the $k$-complex formed by a tight path and $\mathcal{H}'$ the subcomplex induced by its initial \emph{and} terminal $(k-1)$-tuples.
 
 \begin{claimproof}
  We take $\mathcal{H}$ to be the $k$-complex generated by the down-closure of a
  tight path of length $k+i$ (so $\mathcal{H}$ has $2k-1+i$ vertices), and
  consider its $k$-partition in which the $i$th vertex of the path lies in the
  vertex class $V_j$ with $j = i$ modulo $k$. We take $\mathcal{H'}$ to be the
  subcomplex of $\mathcal{H}$ (with empty $k$th level) induced by the first and last $k-1$
  vertices of $\mathcal{H}$. By our assumptions on our various constants and
  by Lemma~\ref{restriction}, $\cG[\bigcup_{j \in [k]} Y_j]$ satisfies the
  conditions to apply Lemma~\ref{lem:ext}. Consider the pairs $(e, f)$, where
  $e$ is an ordered $(k-1)$-edge of $\cG(Y_1,\ldots,Y_{k-1})$ and $f$ is an ordered $(k-1)$-edge of $\cG(Y_{i+1},\ldots,Y_{i + k-1})$. For any such ordered edge $e$ there
  are at most $k m^{k-2}$ such ordered edges $f$ which intersect $e$, so by~\eqref{emb:n0} and~\eqref{eq:emb:numedge}, at most
  a $1/200$-proportion of the pairs $(e, f)$ are not disjoint. On the other hand,
  if $e$ and $f$ are disjoint, then (the down-closure of) the pair $(e, f)$ forms a labelled copy of
  $\mathcal{H'}$ in $\cG[\bigcup_{j\in[k]} Y_j]$, so by Lemma~\ref{lem:ext} with $b=2k-1+i$ and $b'=2k-2$, for all but at most a
  $1/200$-proportion of the disjoint pairs $(e, f)$ there are at least $c(\alpha
  m)^{i+1}\ge\sqrt{\eps} (\alpha  m)^{i+1}$ extensions to copies of
  $\mathcal{H}$ in $\cG[\bigcup_{j\in[k]} Y_j]$. Each such copy of $\mathcal{H}$ corresponds to a 
  tight path in $G$ of length $k+i$ with all vertices in the
  desired clusters. We conclude that at least a $99/100$-proportion of all pairs
  $(e, f)$ of ordered edges are disjoint and are linked by at least this many tight paths in $G$ 
  of the desired type; we call these pairs \emph{extensible}.

Let us call an ordered edge $e\in
  \cG(Y_{1},\ldots,Y_{k-1})$ \emph{good} if at most one-twentieth of the ordered edges
  $f \in \cG(Y_{i+1},\ldots,Y_{i + k-1})$ do not make an extensible pair with
  $e$. Then at most one-fifth of the ordered edges in $\cG(Y_1,\ldots,Y_{k-1})$ are
  not good. In particular, there must exist a path $P\in\paths$ whose terminal
  $(k-1)$-tuple is a good ordered edge $e$. Fix such a $P$ and $e$. 
Given any ordered edge
  $f$ in $\cG(Y_{i+1},\ldots,Y_{i + k-1})$ which is disjoint from $P$, suppose
  the pair $(e,f)$ is an extensible pair. By definition there are at least
  $\sqrt{\eps} (\alpha m)^{i+1}$ tight paths in $G$ from $e$ to $f$ where the $j$th vertex of each path
  lies in $Y_j$. We claim that at least one of these paths has the further
  property that if $j\ge k$, then the $j$th vertex is not contained in $P$ (and
  we can therefore put this path in $\paths'$). Indeed, as $f$ is disjoint from $P$, if $i=0$ it suffices to show that one of these paths has the property that its $k$th vertex is in $Y_k\setminus V(P)$. This is true because there are only $V(P)\le 2k+1<\sqrt{\eps}(\alpha m)$ (where the last inequality is by~\eqref{emb:n0}) paths which do not have this property. If on the other hand $i=1$, then we need a path whose $k$th and $(k+1)$st vertices are not in $V(P)$, which is possible since $2(2k+1)(\alpha m)<\sqrt{\eps}(\alpha m)^2$ by~\eqref{emb:n0}.

Finally, consider the ordered edges $f \in \cG(Y_{i+1},\ldots,Y_{i + k-1})$. Since by~(\ref{emb:n0}) and~(\ref{eq:emb:numedge}) we have $20 |P| (\alpha m)^{k-2} \leq \eps m^{k-1} \leq e(\cG(Y_{i+1},\ldots,Y_{i + k-1}))$, at most one-twentieth of these edges $f$ intersect $P$, and by choice of $e$ at most one-twentieth of these edges $f$ are such that $(e, f)$ is not extensible. This leaves at least nine-tenths of the edges $f$ remaining; choosing a tight path for each such $f$ as described above gives the desired set~$\paths'$.
 \end{claimproof} 

 Now, let $e_1,\ldots,e_s$ be the edges of non-zero weight in our fractional
 matching in $R$, and let $w_1, \ldots, w_s$ be the corresponding weights. For
 each $i \in [s]$ let $n_i$ be any integer with
 \begin{equation}\label{eq:emb:ni}
  0 \leq n_i \leq (1-3\alpha) w_i m\,.
 \end{equation}
 We next construct tight walks $W_i$ in $R$ for `traversing' from $e_i$ to $e_{i+1}$ for each $1\le i\le s-1$.
 Since our fractional matching is tightly connected, we may choose a minimum
 length tight walk $W_1$ from $e_1$ to $e_2$. 
Then
 for each $2 \le i \le s-1$, we may take a minimum length tight walk $W_i$ from
 $e_i$ to $e_{i+1}$ whose initial $(k-1)$-tuple is the terminal $(k-1)$-tuple of
 $W_{i-1}$.
 
 Finally, we construct a tight walk $W_s$ from $e_s$ to $e_1$ whose initial $(k-1)$-tuple is the terminal $(k-1)$-tuple of $W_{s-1}$ and whose terminal $(k-1)$-tuple is the initial $(k-1)$-tuple of $W_1$. We do this construction differently in order to ensure that our final cycles have length divisible by $k$. Indeed, let $W'$ be the concatenation $W_1 + \dots + W_{s-1}$, so $W'=(A_1,\ldots,A_{\ell'})$ is a
 tight walk in $R$ from $e_1$ to $e_s$. Now we let $W_s$ be the sequence of clusters
 given by writing down the terminal $(k-1)$-tuple of $W'$, that is, $A_{\ell'-k+2},A_{\ell'-k+3},\dots,A_{\ell'}$, followed by the
 penultimate $(k-1)$-tuple $A_{\ell'-k+1},A_{\ell'-k+2},\dots,A_{\ell'-1}$, and so on until we eventually write the initial
 $(k-1)$-tuple. Now any $k$ consecutive vertices of $W_s$ come from two $(k-1)$-tuples in $W'$ which share $k-2$ vertices. Since $W'$ is a tight walk in $R$, the vertices of the two $(k-1)$-tuples make an edge of $R$. In other words, $W_s$ is a tight walk in $R$. Note that by construction the terminal $(k-1)$-tuple of $W_s$ is the initial $(k-1)$-tuple of $W_1$, and also that we have
 \[\ell(W_s) = (k-1) \ell(W') = (k-1) \sum_{i = 1}^{s-1} \ell(W_i)\,.\]
 Note that
 any given $(k-1)$-tuple can appear at most once (consecutively) in a minimum length tight walk between two edges of $R$, or else we
 could contract the walk. It follows that
 each of the walks $W_i$ has length $\ell(W_i) \leq t^{k-1}$. Since
 $R$ has at most $\binom{t}{k}$ edges we conclude that $\sum_{i=1}^s
 \ell(W_i) = k \sum_{i = 1}^{s-1} \ell(W_i)$ is a multiple of $k$ and is at most $t^{2k}$.

Let $(X_1, \dots, X_k)$ be the initial $k$-tuple of $W_1$ (so $X_1, \dots, X_k$ are the clusters of $e_1$ in the order in which they appear in $W_1$). Given any subsets $Y_j \subseteq X_j$ of size $|Y_j| = \alpha m$ for $j \in [k]$, we say that an edge $e \in \cG(X_1, \dots, X_{k-1})$ is \emph{well-connected to $(Y_1, \dots, Y_{k-1})$ via $Y_k$} if for at least nine-tenths of the $(k-1)$-tuples $f$ in $\cG(Y_1,\ldots, Y_{k-1})$ there exist distinct vertices $u, v \in Y_k$ such that the concatenations $e + (u) + f$ and $e + (v) + f$ are tight paths in $G$ of length $k$. 
Now fix a subset $Z_j \subseteq X_j$ of size $\alpha m$ for each $j \in [k]$, and write
$Z = \bigcup_{j \in [k]} Z_j$. We reserve the vertices of $Z$ for joining
together the ends of the tight path we will construct to obtain a cycle; the following claim establishes the properties we will need to do this.

\begin{claim} \label{clm:emb2}
For any subsets $Y_j \subseteq X_j$ of size $\alpha m$ for $j \in [k]$ such that each $Y_j$ is disjoint from $Z_j$ the following statements hold. \begin{enumerate}[label=\abc]
\item\label{emb2:a} At least nine-tenths of the $(k-1)$-tuples $e$ in $\cG(Z_1,\ldots,Z_{k-1})$ are well-connected to $(Z_1, \dots, Z_{k-1})$ via $Z_k$.
\item\label{emb2:b} At least nine-tenths of the $(k-1)$-tuples $e$ in $\cG(Z_1,\ldots,Z_{k-1})$ are well-connected to $(Y_1, \dots, Y_{k-1})$ via $Y_k$.
\item\label{emb2:c} At least nine-tenths of the $(k-1)$-tuples $e$ in $\cG(Y_1,\ldots,Y_{k-1})$ are well-connected to $(Z_1, \dots, Z_{k-1})$ via $Y_k$. 
\end{enumerate}
\end{claim}
  
\begin{claimproof}
Observe that~\ref{emb2:a} was in fact proved in the proof of Claim~\ref{clm:emb} (in the case $i=0$), with $Y_1, \dots, Y_{k}$ there corresponding to $Z_1, \dots, Z_{k}$ here. We here modify this argument to prove~\ref{emb2:b}; a near-identical argument proves~\ref{emb2:c}.
 We apply Lemma~\ref{lem:ext} as in Claim~\ref{clm:emb}, with 
$\mathcal{H}$ being the $k$-complex generated by the down-closure of a tight path of length $k$ (that is, with $2k-1$ vertices), and $\mathcal{H}'$ the subcomplex
 induced by its initial and terminal $(k-1)$-tuples. However, we now regard $\mathcal{H}$ as a $(2k-1)$-partite $k$-complex, with one vertex in each vertex class. The role of $\cG$ in Lemma~\ref{lem:ext} is played by the $(2k-1)$-partite subcomplex of $\cG$ with vertex classes $Z_1, \dots, Z_{k-1}, Y_k, Y_1, \dots, Y_{k-1}$; the first vertex of $\mathcal{H}$ is to be embedded in $Z_1$, the second in $Z_2$, and so forth. So by Lemmas~\ref{restriction} and~\ref{lem:ext} the proportion of pairs $(e, f)$ for which there is no path as in~\ref{emb2:b} is at most $1/200$, and the remainder of the argument then follows exactly as in Claim~\ref{clm:emb}.
 \end{claimproof}
  
 We are now ready to construct our cycle. Arbitrarily choose a subset $X^{(0)}_j \subseteq X_j$ of size $\alpha m$ which is disjoint from $Z_j$ for each $j \in [k]$. By Claim~\ref{clm:emb2}\ref{emb2:a} and~\ref{emb2:b} we may fix a $(k-1)$-tuple 
 $e \in \cG(Z_1,\ldots,Z_{k-1})$ such that $e$ is both well-connected to $(Z_1, \dots, Z_{k-1})$ via $Z_k$ and well-connected to $(X^{(0)}_1, \dots, X^{(0)}_{k-1})$ via $X^{(0)}_k$. Set $P^{(0)}$ to be the tight path with no 
 $k$-edges consisting simply of the vertices of $e$ in their given order. By choice of $e$ there is a set 
 $\paths^{(0)}$ of tight paths of the form $e + (v) + f$ for $v \in X^{(0)}_k$ and $f \in \cG(X_1^{(0)}, \dots, X_{k-1}^{(0)})$ for which the terminal $(k-1)$-tuples of paths in $\paths^{(0)}$ are all distinct and constitute at least half of the ordered edges of $\cG(X_1^{(0)}, \dots, X_{k-1}^{(0)})$. We now describe the algorithm we use to construct the desired cycle.
 Set the initial state to be `filling the edge $e_1$'. We proceed for each time $j \geq 1$ as follows, maintaining the following property $(\dagger)$.

\medskip
\noindent $(\dagger)$ The terminal $(k-1)$-tuples of the paths $\paths^{(j)}$ constitute at least half of the ordered edges $\cG(X^{(j)}_1,\ldots,X^{(j)}_{k-1})$.

\medskip
Suppose first that our current state is `filling the edge $e_i$' for some~$i$. If we have previously completed $n_i$ steps in this state, then we do nothing, and immediately change state to `position $1$ in traversing the walk $W_i$'. 
Otherwise, since $(\dagger)$ holds for $j-1$, 
we may apply Claim~\ref{clm:emb} with
 $i=0$ to obtain a path $P\in\paths^{(j-1)}$ and a collection $\paths^{(j)}$ of $\tfrac{9}{10}\,e(\cG(X^{(j-1)}_{1},\ldots,X^{(j-1)}_{k-1}))$ tight paths of length $k$, all of whose initial $(k-1)$-tuples are the terminal $(k-1)$-tuple of $P$,
 whose terminal $(k-1)$-tuples are distinct members of
 $\cG(X^{(j-1)}_{1},\ldots,X^{(j-1)}_{k-1})$ and are disjoint from $V(P)$, and whose remaining vertex lies in $X^{(j-1)}_k \setminus V(P)$. 
We define $P^{(j)}$ to be the concatenation $P^{(j-1)} + P$. For each $1\le p\le k$ we  generate $X^{(j)}_p$ from $X^{(j-1)}_p$ by removing the (at most two) vertices  of $P^{(j)}$ in $X^{(j-1)}_p$ and replacing them by vertices from the same cluster which do not lie in $Z$ or in $P^{(j)}$. We will prove in Claim~\ref{clm:emb3} that this is possible and that $(\dagger)$ is maintained.

Now suppose that our current state is `position $q$ in traversing the walk $W_i$' for some $i$. Since $(\dagger)$ holds for $j-1$, 
we may apply Claim~\ref{clm:emb} with $i=1$,
which returns a path $P\in\paths^{(j-1)}$ and a collection $\paths^{(j)}$ of
 $\tfrac{9}{10}\,e(\cG(X^{(j-1)}_{2},\ldots,X^{(j-1)}_{k}))$ tight paths of length $k+1$, all of whose initial $(k-1)$-tuples are the terminal $(k-1)$-tuple of $P$,
 whose terminal $(k-1)$-tuples are distinct members of
 $\cG(X^{(j-1)}_{2},\ldots,X^{(j-1)}_{k})$ which are disjoint from $V(P)$, and whose two remaining vertices lie in $X^{(j-1)}_k \setminus V(P)$ and $X^{(j-1)}_1 \setminus V(P)$ respectively.  Exactly as before we define $P^{(j)}$ to be the concatenation $P^{(j-1)} + P$. We also form $X^{(j)}_p$ from $X^{(j-1)}_{p+1}$ for each $1\le p\le k-1$ by removing the vertices of $P^{(j-1)}$ in $X_{p+1}^{(j-1)}$ and replacing them by vertices from the same cluster which do not lie in $Z$ or $P^{(j)}$. 
If we have now reached the end of $W_i$, meaning that the $(k-1)$-tuple of clusters containing $X^{(j)}_1, \dots, X^{(j)}_{k-1}$ (in that order) is the terminal $(k-1)$-tuple of $W_i$, then we choose the set $X^{(j)}_k$ as follows. If $i < s$ then we let $X^{(j)}_k$ be a subset of the remaining cluster of $e_{i+1}$ (that is, the cluster not included in the terminal $(k-1)$-tuple) which has size $\alpha m$ and is disjoint from $P^{(j)}\cup Z$. If $i=s$ we let $X^{(j)}_k$ be a subset of $X_k$ (which is the remaining cluster of $e_1$) of size $\alpha m$ disjoint from $P^{(j)}\cup Z$. We then change our state to `filling the edge $e_{i+1}$' if $i < s$, or `completing the cycle' if $i = s$. Alternatively, if we have not yet reached the end of $W_i$, we instead choose  $X^{(j)}_k$ to be a subset of size $\alpha m$ of the cluster at position $q+k$ in the sequence $W_i$, again chosen so that $X^{(j)}_k$ does not intersect $P^{(j)}$ or $Z$. This ensures that for each $p \in [k]$, $X^{(j)}_p$ is a subset of the cluster at position $p + q$ in the sequence $W_i$; in particular, these clusters form an edge of $R$ since $W_i$ is a tight walk. In this case we now change our state to `position $q+1$ in traversing
 $W_i$'. Again, we prove in Claim~\ref{clm:emb3} that these choices are all possible and that $(\dagger)$ is maintained.

Finally, if our state is `completing the cycle' then $X^{(j-1)}_1, \dots, X^{(j-1)}_{k}$ must be subsets of $X_1, \dots, X_{k}$ respectively. So by $(\dagger)$ and Claim~\ref{clm:emb2}\ref{emb2:c} we may choose a path $P \in \paths^{(j-1)}$ such that the terminal $(k-1)$-tuple $f \in \cG(X^{(j-1)}_1, \dots, X^{(j-1)}_{k-1})$ of $P$ is well-connected to $(Z_1, \dots, Z_{k-1})$ via~$X^{(j-1)}_k$. Let $P^{(j)}$ be the concatenation $P^{(j-1)} + P$, and recall that we chose $e$, the initial $(k-1)$-tuple of $P^{(j-1)}$, to be well-connected to $(Z_1, \dots, Z_{k-1})$ via $Z_k$. Together with the well-connectedness of $f$, and the fact that $V(P) \cap X_k = 1$, this implies that we may choose a $(k-1)$-tuple $e'$ in $\cG(Z_1, \ldots, Z_{k-1})$ and vertices $v \in X_k^{(j-1)} \setminus V(P)$ and $v' \in Z_k$  such that $e'$ is disjoint from~$e$ and both $Q := f + (v) + e'$ and $Q' := e' + (v') + e$ are tight paths in $G$. Return $P^{(j)} + Q + Q'$ as the output tight cycle in $G$.

 \begin{claim}\label{clm:emb3}
The algorithm described above is well-defined (that is, it is always possible to
 construct the sets $X^{(j)}_p$), maintains~$(\dagger)$, and returns a tight cycle of length 
$$\left(3 + \sum_{i \in [s]} n_i\right) \cdot k + \left(\sum_{i \in [s]} \ell(W_i)\right) \cdot (k+1).$$ In particular, this length is divisible by $k$.
\end{claim}

\begin{claimproof}
To see that the output is indeed a tight cycle, recall that we always chose $X_1^{(j)}, \dots, X_k^{(j)}$ to be disjoint from $Z$ and $P^{(j)}$, and the vertices added to $P^{(j)}$ to form $P^{(j+1)}$ are always taken from $\bigcup_{p \in [k]} X_p^{(j-1)}$. So by construction, the final $P^{(j)}$ is a tight path which only meets $Z$ in its initial $(k-1)$-tuple $e$ and does not contain the vertex $v$ used in the `completing the cycle' step. Since when completing the cycle $e'$ is chosen to be disjoint from $e$, no vertices are repeated, and so the output is indeed a tight cycle.

 To see that $(\dagger)$ is maintained, observe that applying~\eqref{eq:emb:numedge} we have $e\big(\cG(X_1^{(j)},\ldots,X_{k-1}^{(j)})\big)\ge\eps m^{k-1}$ for each $j$. Fix some $j$. By construction, for either $A_p := X^{(j-1)}_p$ or $A_p := X^{(j-1)}_{p+1}$ (according to our current state) we obtain sets $A_1,\ldots,A_{k-1}$ each of size $\alpha m$ such that the terminal $(k-1)$-tuples of $\paths^{(j)}$ constitute at least nine-tenths of the ordered edges of $\cG(A_1,\ldots,A_{k-1})$, and for each $1\le i\le k-1$ the set $X_i^{(j)}$ is formed from $A_i$ by removing at most two vertices and replacing them with the same number of vertices. Since each vertex is in at most $m^{k-2}$ ordered $(k-1)$-edges of either $\cG(A_1,\ldots,A_{k-1})$ or $\cG(X_1^{(j)},\ldots,X_{k-1}^{(j)})$, we conclude that the fraction of ordered $(k-1)$-edges of $\cG(X_1^{(j)},\ldots,X_{k-1}^{(j)})$ which are terminal $(k-1)$-tuples of paths in $\paths^{(j)}$ is at least
 \begin{multline*}
  \frac{\frac{9}{10}\,e\big(\cG(A_1,\ldots,A_{k-1})\big)-2(k-1) m^{k-2}}{e\big(\cG(X_1^{(j)},\ldots,X_{k-1}^{(j)})\big)}\\
  \ge \frac{\frac{9}{10}\big(e\big(\cG(X_1^{(j)},\ldots,X_{k-1}^{(j)})\big)-2(k-1) m^{k-2}\big)-2(k-1) m^{k-2}}{e\big(\cG(X_1^{(j)},\ldots,X_{k-1}^{(j)})\big)}\\
  \ge \frac{9}{10}-\frac{4(k-1) m^{k-2}}{\eps m^{k-1}}\ge \frac12\,,
 \end{multline*}
 where the final inequality is because we have $m\ge m_0\ge 16(k-1)/\eps$. Thus $(\dagger)$ holds for $j$ as desired.

 To see that we can always
 construct the sets $X^{(j)}_p$, observe that it is enough to check that at
 termination every cluster still has at least $2\alpha m$ vertices not in
 $P^{(j)}$, as then there are at least $\alpha m$ such vertices outside $Z$. 
 Observe that at each step we choose a set of paths $\paths^{(j)}$, precisely one member of which is then used to extend $P^{(j)}$ to $P^{(j+1)}$ in the subsequent step. 
 In each walk-traversing step each path in $\paths^{(j)}$ contains precisely $k+1$ new vertices (i.e. vertices outside $P^{(j)}$), and the total number of walk-traversing steps is precisely $\sum_{i \in [s]} \ell(W_i)$. Recalling that this number is at most $t^{2k}$, and using $t \leq t_1$, by~\eqref{emb:n0} we have $(k+1)t^{2k}
< \alpha m/2$, so in particular fewer than $\alpha m/2$ vertices are added to $P^{(j)}$ as a result of walk-traversing steps. On the
 other hand, we remain in the state `filling the edge $e_i$' for precisely
 $n_i$ steps, and in each of these steps each path in $\paths^{(j)}$ contains
 precisely $k$ new vertices, one from each cluster of $e_i$. So for any cluster~$C$, the number of vertices of $C$ which are added to $P^{(j)}$ as a result of edge-filling steps is
 $\sum_{i :~C \in e_i} n_i \leq \sum_{i :~C \in e_i} (1-3\alpha) w_i m \leq
 (1-3\alpha) m$, where the first inequality holds by~\eqref{eq:emb:ni} and the
 second by definition of a fractional matching. Together with the $k-1$ vertices of $P^{(0)}$, and the $k$ vertices of the chosen path in $\paths^{(0)}$, we conclude that in total at
 most $(1-2\alpha) m$ vertices of any cluster are contained in the path $P^{(j)}$ at termination, as desired.
 
Finally, the total length of the cycle is equal to the number of vertices it contains, which we can calculate similarly. Recall that our initial $P^{(0)}$ contained $k-1$ vertices, to which $k$ vertices were added from some member of $\paths^{(0)}$ in the first step to form $P^{1}$. Each of the $\sum_{i \in [s]} n_i$ edge-filling steps resulted in $k$ new vertices being added to $P^{(j)}$ in the subsequent step, and each of the $\sum_{i \in [s]} \ell(W_i)$ walk-traversing steps resulted in $k+1$ new vertices being added to $P^{(j)}$ in the subsequent step. Finally, when completing the cycle we used $k+1$ vertices which were not contained in our final path $P^{(j)}$ (namely $v, v'$ and the vertices of $e'$). In summation, the cycle formed has length
\[(k-1)+k + \left(\sum_{i \in [s]} n_i\right) \cdot k + \left(\sum_{i \in [s]} \ell(W_i)\right) \cdot (k+1) + (k+1)\,,\]
giving the claimed expression. Since $\sum_{i \in [s]} \ell(W_i)$ is divisible by $k$, the same is true of this length.
\end{claimproof}

Recall that $\sum_{i \in [s]} \ell(W_i) \leq t^{2k}$. So if we take $n_i = 0$
for every $i \in [s]$ then we obtain a tight cycle of length at most $2k
t^{2k}$. On the other hand, if we take the $n_i$ to be as large as permitted,
so $n_i = (1-3\alpha)w_i m$ for each $i \in [s]$, then, since $\sum_{i=1}^{s} w_i = \mu$, we obtain a tight cycle of length at least $(1-3\alpha) \mu k m \geq
(1-\psi) k \mu n/t$. Clearly by choosing the $n_i$ appropriately we may
obtain tight cycles of any length between these two extremes which is divisible
by $k$. So it remains only to prove the lemma for cycle lengths $\ell < 2k
t^{2k}$. Note that $2k t^{2k} < m$, so the number of vertices in such a
cycle is fewer than the number of vertices in any cluster. Hence, we need only use
one edge of the reduced graph to find the desired tight cycle in this case.

More precisely, fix any edge $e_1 = \{X_1, \dots, X_k\}$ of $R$. If $\ell \geq 3k$, then we may
proceed as before with $n_1 := \ell/k - 3$. That is, we choose a $(k-1)$-tuple
$e$, disjoint sets $Z_j, X^{(0)}_j \subseteq X_j$ of size $\alpha m$, and a set of extensions $\paths^{(0)}$ of $P^{(0)} = e$ as previously, then enter state `filling the edge $e_1$', where we remain
for $n_1$ steps. Following this, we move directly to the state `completing the
cycle' (since there is no walk to traverse), which proceeds as before. By
similar arguments as before, this process gives a tight cycle of length
$kn_1 + 3k = \ell$, as required.

Finally, if $\ell = 2k$, we simply apply Lemma~\ref{lem:ext}
 with $\mathcal{H}$ the complex generated by the down-closure of a tight cycle of length $2k$, and
 $\mathcal{H}'$ the subcomplex induced by any $k-1$ consecutive vertices of this cycle, to 
 $\cG(X_1, \dots, X_k)$, where $X_1, \dots, X_k$ are the clusters of any edge of $R$. 
 This gives a cycle of length $2k$ in $G$.
\end{proof}

To conclude this section, we note that by a similar approach it is possible to find paths of specified
lengths whose initial and terminal $(k-1)$-tuples lie in specified sets of $(k-1)$-tuples. We will not need this
result here, but we state it for future convenience. 

\begin{lemma}\label{lem:fracembrest}
  Let $k,r,n_0,t, B$ be positive integers, and
  $\psi,d_2,\ldots,d_k,\eps,\eps_k, \nu$ be positive constants such that
  $1/d_i\in\NATS$ for each $2\le i\le k-1$, and such
  that $1/n_0 \ll 1/t$,
  \[\frac{1}{n_0}, \frac{1}{B}\ll
  \frac{1}{r},\eps\ll\eps_k,d_2,\ldots,d_{k-1}\quad\text{and}\quad\eps_k\ll
  \psi,d_k, \nu, \frac{1}{k}\,.\]
  Then the following holds for all integers $n\ge n_0$.

  Let~$G$ be a $k$-graph
  on~$n$ vertices, and $\cJ$ be a $(\cdot,\cdot,\eps,\eps_k,r)$-regular slice
  for $G$ with $t$ clusters and density vector $(d_{k-1},\ldots,d_2)$. Let
  $M$ be a tightly connected fractional matching of weight $\mu$ in $R :=
  R_{d_k}(G)$. Also let $X$ and $Y$ be $(k-1)$-tuples of clusters which lie in
  the same tight component of $R$ as $M$, and let $S_X$ and $S_Y$ be subsets of
  $\cJ_X$ and $\cJ_Y$ of sizes at least $\nu|\cJ_X|$ and $\nu|\cJ_Y|$ respectively. Finally, let $W$ be a tight walk in $R$ from $X$ to $Y$ of length at most $t^{2k}$, and let $p=\ell(W)$.

  Then for any
  $\ell$ divisible by $k$ with $3k \leq \ell \leq (1-\psi)k\mu n/t$
  there is a tight path $P$ in $G$ of length $\ell+p(k+1)$ whose initial
  $(k-1)$-tuple forms an edge of $S_X$ and whose terminal $(k-1)$-tuple forms an edge of
  $S_Y$, where the orders of these tuples are given by the orders of $X$ and $Y$ respectively. Furthermore $P$ uses at most $\mu(C)n/t + B$ vertices from any
  cluster $C$, where $\mu(C)$ denotes the total weight of edges of $M$ containing $C$.
\end{lemma}

\begin{proof}
We need to be able to apply Lemma~\ref{lem:ext} with
$\beta=\min(\nu/40,1/100)$ rather than $1/100$, and we need $\alpha\ll\nu$, but
the remaining constants only have to satisfy the given order of magnitude hierarchy.
We use the same notation as in the proof of Lemma~\ref{lem:emb}. So, writing $X
= (X_1, \dots, X_{k-1})$ and $Y = (Y_1, \dots, Y_{k-1})$ we can rewrite our
assumption on $S_X$ and $S_Y$ in a more familiar form: that $S_X$ constitutes at
least a $\nu$-proportion of $\cG(X_1, \dots, X_{k-1})$ and $S_Y$ constitutes at
least a $\nu$-proportion of $\cG(Y_1, \dots, Y_{k-1})$. Let $X_k$ be the cluster
following $X$ in $W$, and let $Y_k$ be the cluster preceding $Y$ in $W$, so
$\{X_1, \dots, X_k\}$ and $\{Y_1, \dots, Y_k\}$ are edges of $R$. Then we choose
subsets $Z_j \subseteq Y_j$ and $X^{(0)}_j \subseteq X_j$ of size $\alpha n$ for
each $j \in [k]$. A similar argument as in the proof of Lemma~\ref{lem:emb}
implies that we may choose $e \in S_Y$ so that for at least half the members $f$
of $\cG(Z_1, \dots, Z_k)$ there is a tight path of length $k$ in $G$ from $f$ to
$e$ whose remaining vertex lies in $Z_k$. Similarly, we may choose a
$(k-1)$-tuple $P^{(0)}$ in $S_X$ and a set $\paths^{(0)}$ of tight paths of the
form $P^{(0)} + v + e'$ with $v \in Y_k$ and $e' \in \cG(Y_1, \dots, Y_{k-1})$
so that the members of $\paths^{(0)}$ have distinct terminal $(k-1)$-tuples
which together occupy at least nine-tenths of $\cG(Y_1, \dots, Y_{k-1})$. We
then proceed by exactly the same algorithm as in the proof of
Lemma~\ref{lem:emb} to repeatedly extend $P^{(j)}$ whilst avoiding $Z_1, \dots,
Z_{k}$ and $e$. The only difference is that now in the `completing the cycle'
state, when we identify tight paths $Q$ and $Q'$ which together connect our
final $P^{(j)}$ to $e$, this does not yield a tight cycle but rather a tight
path whose initial $(k-1)$-tuple lies in $S_X$ and whose final $(k-1)$-tuple
lies in $S_Y$, as claimed.

It remains only to check the lengths of tight paths which can be obtained in
this way. As in Lemma~\ref{lem:emb} the shortest tight path is achieved by never
entering the state of `filling an edge', in which case we obtain a tight path of
length $3k + p(k+1)$. On the other hand, exactly as for Lemma~\ref{lem:emb}, by
extending $W$ to include all edges of $R$ of non-zero weight before we implement
the algorithm, we can obtain tight paths of length up to $(1-\psi)k\mu
n/t$, but without using more than $\mu(C)n/t + B$ vertices from any
cluster, where $B=B(t,k)$ does not depend on $n$.
\end{proof}

\section{Concluding remarks}
\label{sec:conclude}

\paragraph{\bf Towards an extremal theorem for tight paths and cycles}

We mentioned in the introduction that our result is an approximate analogue
of the Erd\H{o}s-Gallai Theorem. The exact analogue is the following, which we
conjecture.

\begin{conjecture}\label{conj:EG} For any~$\ell$, all $n$-vertex
  $k$-graphs with more than $\tfrac{\ell-k}{k}\binom{n}{k-1}$ edges
  contain a
  tight path on $\ell$ vertices, and all $n$-vertex
  $k$-graphs with more than $\tfrac{\ell-1}{k}\binom{n-1}{k-1}$ edges have
  a tight cycle of length at least $\ell$. 
\end{conjecture}

As we saw in Section~\ref{sec:LB}, for any fixed $\ell$ this conjecture is
sharp for $n$ satisfying certain divisibility conditions. We also observed
there that if $p\ll n^{-(\ell-k)/(\ell-1)}$
then a random hypergraph of density $p$ has fewer $\ell$-vertex tight
cycles than edges, so we can easily delete all short cycles without
significantly altering the density. Hence we cannot ask for the existence
of cycles of lengths up to $\ell$.

Recall that Gy\H{o}ri, Katona and Lemons~\cite{GyKatLem} proved that more
than $(\ell-k)\binom{n}{k-1}$ edges suffices to guarantee the existence of
an $\ell$-vertex tight path, which is weaker than the
conjecture both in that it only deals with tight paths and in that a factor $k$
more edges than conjectured are required. However their result holds for
all $\ell\le n$. Our result is off by a factor only
$\big(1+o(1)\big)$ from the conjectured number of edges, but we require $\ell$
to be linear in $n$. Possibly one could even prove the conjecture exactly in
our range of $\ell$ using the Stability Method, but we do not believe an
attempt to do so is worthwhile, for the following reason.

For large cycles we do not believe Conjecture~\ref{conj:EG} is best
possible. It is easy to check that if $\alpha\gg n^{-1/2}$ then no designs
(or even set systems with nearly as many edges as a design) with sets of
size $\alpha n$ and without pairwise intersections of size two or greater
exist: The lines of $\mathbb{F}^2_p$ for prime $p\approx n^{1/2}$ form a
design with all pairs covered exactly once, so this is optimal.

The best lower bound we know of in this
range of~$\alpha$ is the simple construction presented in Section~\ref{sec:LB} in which we take all edges
meeting a fixed $(\alpha n/k-1)$-vertex set. There is a substantial gap
between this lower bound and Conjecture~\ref{conj:EG}. 
It would be very interesting---but also, we suspect, very difficult---to
close this gap. In fact, we think that it is already difficult to solve the
following problem for some fixed $\alpha\in(0,1)$.

\begin{problem}
  Determine the limit of the maximum edge density of $n$-vertex $3$-graphs
  which do not contain $\alpha n$-vertex tight paths.
\end{problem}

We have no suggestions of good candidates for extremal structures, which
would also be interesting to obtain.  For $\alpha$ close to $1$, it even
seems possible that the construction presented in Section~\ref{sec:LB} in which we take an $(\alpha n-1)$-vertex clique and all further
edges which meet it in less than $k-1$ vertices could be extremal.

\smallskip

\paragraph{\bf Spanning structures and stability} 
Let us support our claim that Lemma~\ref{lem:simpreg} property~\ref{simpreg:c}
can be useful for proving results involving spanning subgraphs. For this
purpose we first briefly sketch a proof of R\"odl,
Ruci\'nski and Szemer\'edi~\cite{RRSHam} giving the following Dirac-type condition
for tight Hamilton cycles in $k$-graphs, and then explain how their
approach can be simplified with the help of Lemma~\ref{lem:simpreg}.

\begin{theorem}[\cite{RRSHam}, Theorem 1.1]\label{thm:rrsham}
  For each $k\ge 3$ and $\gamma>0$ there exists $n_0$ such that for each $n\ge
  n_0$ the following holds. If $G$ is a $k$-graph on $n$ vertices and each
  $(k-1)$-set of vertices of $G$ is contained in at least $(1/2+\gamma)n$ edges of $G$,
  then $G$ has a tight Hamilton cycle.
\end{theorem}

The strategy of R\"odl, Ruci\'nski and Szemer\'edi is as follows. First,
they show that the codegree condition implies that for each vertex~$u$
of~$G$, there are $\Theta\big(n^{2k-2}\big)$ `absorbing structures for
$u$', that is, vertex tuples $(v_1,\ldots,v_{2k-2})$ such that both of the
tuples $(v_1,\ldots,v_{2k-2})$ and
$(v_1,\ldots,v_{k-1},u,v_k,\ldots,v_{2k-2})$ form tight paths in
$G$. Second, they use a probabilistic argument to show that there is a
collection of pairwise-disjoint $(2k-2)$-vertex tight paths which cover a
small fraction of $V(G)$ and which have the property that for each $u\in
V(G)$, $\Theta(n)$ of these tight paths are absorbing structures for
$u$. Third, they establish a Connecting Lemma which allows them to connect
the collection of $(2k-2)$-vertex tight paths into one `absorbing path',
which is a tight path $P$ such that for any not too large set $S\subset
V(G)\setminus V(P)$, there is a tight path on the vertices $V(P)\cup S$
with the same endpoints as $P$. Fourth, they show that $P$ can be extended
to an almost-spanning tight cycle. Then setting $S=V(G)\setminus V(C)$ and
using the absorbing property of $P$ completes their proof.

In their approach, the Regularity Lemma is used only to complete the fourth
step, and most of the work is in proving the Connecting Lemma. We can use
Lemma~\ref{lem:simpreg} to simplify this approach. Property~\ref{simpreg:c}
implies that since each $u\in V(G)$ has many absorbing structures, so each
$u$ also has many absorbing structures supported on the regular slice $\cJ$
returned by Lemma~\ref{lem:simpreg}, and the same probabilistic argument as
in their second step then gives a collection of $(2k-2)$-vertex tight paths
supported on $\cJ$ with otherwise the same properties. But now their third
step can be replaced by showing that the reduced graph $R(G)$ is (after
deleting a few vertices) tightly connected. This, however, is easy and can
be obtained as in our proof of Theorem~\ref{thm:PartiteCycle}. Their fourth
step can be carried out in more or less the same way on $\cJ$, and the
proof is then completed exactly as before.

Besides simplifying the proof, we also believe it would be much easier to
prove a `stability version' of Theorem~\ref{thm:rrsham} using our approach,
because the tight connectivity of $R(G)$ we use is a comparatively easy
concept to work with. It appears more challenging to prove a stability version
of the Connecting Lemma in their approach. 

Furthermore, it would be interesting to prove similar results for more
general `tight-path-like' hypergraphs. To do this one will need a suitable
`embedding lemma' which can complete the fourth step (in either approach)
and this does not currently exist. Given such an `embedding lemma', it
would be easy to modify our approach to complete the proof of such a
result.

\smallskip

\paragraph{\bf Entropy preserving regular slices}

We note that one can ask for the slice $\cJ$ provided by Lemma~\ref{lem:simpreg} to satisfy further properties. In particular, one property which is useful in enumeration is that entropy is preserved.

The \emph{binary entropy} of a number $x\in(0,1)$ is defined to be
$\entropy(x)=-x\log_2 x-(1-x)\log_2 (1-x)$, and we define
$\entropy(0)=\entropy(1)=0$. Given a weighted reduced $k$-graph $R$, we
define the binary entropy $\entropy(R)$ of $R$ to be the average over all
$e\in\binom{V(R)}{k}$ of $\entropy\big(\reld(e)\big)$.  Then we can ask for
the following additional property of Lemma~\ref{lem:simpreg}. The number of
$n$-vertex $k$-graphs $G$ whose reduced graph $R(G)$ given by
Lemma~\ref{lem:simpreg} has $\entropy\big(R(G)\big)\le x$ is at most
$2^{x\binom{n}{k}+\eps_k n^k}$.

We define the binary entropy $\entropy(G,\Part^*)$ of $G$ with respect to a
family of partitions $\Part^*$ to be the average of $\entropy\big(d(G\mid
\hat{P}(Q))\big)$ over the polyads $\hat{P}(Q)$ of
$\Part^*$ (where $d(G\mid
\hat{P}(Q))$ is the relative density of $G$ with respect to $\hat{P(Q)}$ as defined in Section~\ref{subsec:regdef}).  Then the method used in the proof of
Lemma~\ref{lem:simpreg}\ref{simpreg:a} can be used to show that
$\entropy\big(R(G)\big)$ is with high probability close to
$\entropy(G,\Part^*)$, where $\Part^*$ is as given in the proof of
Lemma~\ref{lem:simpreg}. Hence we may ask that Lemma~\ref{lem:simpreg}
returns a regular slice $\cJ$ such that $\entropy\big(R(G)\big)\approx
\entropy(G,\Part^*)$. Then a standard and easy enumeration argument yields
the desired conclusion.

\smallskip

\paragraph{\bf Regular slices and the Hypergraph Blow-up Lemma}
Recently, Keevash~\cite{HypBlow} provided a new major tool for extremal
hypergraph theory, the Hypergraph Blow-up Lemma. However, while our Regular
Slice Lemma, Lemma~\ref{lem:simpreg}, is based on the Strong Hypergraph
Regularity Lemma, Lemma~\ref{reglemma}, the Hypergraph Blow-up Lemma
requires the stronger regularity properties given by the Regular
Approximation Lemma, Lemma~\ref{lem:RAL}. Hence
Lemma~\ref{lem:simpreg} is not suitable for an application together with
the Hypergraph Blow-up Lemma. 

We remark though that Lemma~\ref{lem:simpreg} can
be modified appropriately to allow for such an application. The proof of
such a modified version is almost identical to the proof of
Lemma~\ref{lem:simpreg}, but uses a version of the Strengthened Regularity
Lemma, Lemma~\ref{lem:strReg}, which
gives regularity properties of comparable strength to those in
Lemma~\ref{lem:RAL}. It is easy to modify the proof of
Lemma~\ref{lem:strReg} to obtain this.

\appendix

\section{Derivation of Lemmas~\ref{lem:count} and~\ref{lem:ext}}\label{app:count}

We first describe the differences between Lemma~\ref{lem:count} and the
Lemma provided by Cooley, Fountoulakis, K\"uhn, and
Osthus~\cite[Lemma~4]{CFKO}.  Firstly, their version also allows to count
copies of $H$ in which multiple vertices of $H$ may be embedded within the
same cluster. We do not need this strengthening, so we omit it for
notational convenience. 

Secondly, their lemma includes an additional constant $d_k$ with $\eps_k
\ll d_k$ and $1/d_k \in \NATS$, and requires for any edge $e \in H$ that
$G$ is $(d_k,\eps_k, r)$-regular with respect to $\hat{\cJ_X}$, where $X =
\{V_j : j \in e\}$, whereas our lemma only requires that $G$ is $(d,\eps_k,
r)$-regular with respect to $\hat{\cJ_X}$ for some $d$ which may depend on
$X$. We now describe how this apparently stronger result may be derived.

Introduce a constant $d_k$ with $d_k \ll \beta$, and also a constant
$\gamma$ with $\eps \ll \gamma \ll \eps_k, d_2, \dots, d_{k-1}$. Let~$X$ be
any $k$-set of clusters of $\cJ$. The version of the Dense Extension
Lemma~\cite[Lemma~7]{CFKO}, applied with $\cJ$ in place of $\cG$ and
$\mathcal{H} \setminus \mathcal{H}^{(k-1)}$ in place of $\mathcal{H}$,
implies that at most $\gamma |K_i(\hat{\cJ}_X)|$ edges of $G_X$ are
contained in more than $\prod_{i = 2}^{k-1} d_i^{e_i(\mathcal{H}) -
  \binom{k}{i}}$ copies of $\mathcal{H}$. Also, the Slicing
Lemma~\cite[Lemma~8]{CFKO} implies that we can partition $G_X$ into
subgraphs $G^0_X, G^1_X, \dots, G^{p_X}_X$ for some integer $p_X$, so that
$|G^0_X| \leq d_k |K_i(\hat{\cJ}_X)|$ and $G^i_X$ is $(d_k,2\eps_k,
r)$-regular with respect to $\hat{\cJ_X}$ for each $i \geq 1$.  However we
choose $q_X \in [p_X]$ for each $X$, we may then apply~\cite[Lemma~4]{CFKO}
to establish that there are $(1 \pm \beta/3) m^s \prod_{i=2}^{k}
d_i^{e_i(\Hy)}$ copies of $\mathcal{H}$ in $\bigcup_X G^{q_X}_X$.
Different choices of $q_X$ give distinct copies of $\mathcal{H}$, so
summing over all possible choices we obtain $(\prod_{e \in \mathcal{H}}
\reld(e) \pm 2\beta/3) m^s \prod_{i=2}^{k-1} d_i^{e_i(\Hy)}$ copies of
$\mathcal{H}$ in $G$, having counted all copies except for those which
contain an edge from some $G^0_X$. Fact~\ref{fact:densecount} tells us that
for any $X$ there are at most $d_k |K_i(\hat{\cJ}_X)| \leq
d_k\prod_{i=2}^{k-1} d_i^{\binom{k}{i}} m^k$ such edges, each of which lies
in at most $\prod_{i=2}^{k-1} d_i^{e_i(\mathcal{H}) - \binom{k}{i}}
m^{s-k}$ copies of $\mathcal{H}$, except for at most $\gamma
|K_i(\hat{\cJ}_X)| \leq d_k \prod_{i=2}^{k-1} d_i^{e_i(\Hy)} m^k$ edges,
which each lie in up to $m^{s-k}$ copies of $\mathcal{H}$. Overall this
adds at most $(\beta/3) m^s \prod_{i=2}^{k-1} d_i^{e_i(\Hy)}$ further
copies of $\mathcal{H}$ in $G$, giving Lemma~\ref{lem:count}.

\smallskip

Similarly, \cite[Lemma~4.6]{lcycles} differs from our Lemma~\ref{lem:ext}
in that it requires that for each $A\in \binom{[s]}{k}$ such that there is an edge of
$\Hy$ with index~$A$, the graph $\cG^{(k)}[V_A]$ is regular with density
exactly $d_k$ with respect to $\cG^{(k-1)}[A]$, whereas we require density
at least $d_k$. Moreover, \cite[Lemma~4.6]{lcycles} also gives a
formula for the typical number of extensions, whereas we give only a lower
bound. This latter difference makes it elementary to reduce our lemma to theirs by applying the Slicing
Lemma~\cite[Lemma~8]{CFKO}.

\section{Proof of Lemma~\ref{lem:strReg}.}
\label{app:strReg}

Let us first justify why we need Lemma~\ref{lem:strReg}.  As we saw in the
proof of Lemma~\ref{lem:simpreg}, Lemma~\ref{lem:strReg} is used to ensure
that the link of \emph{every} vertex of $G$ is represented in the regular
slice. It may not be clear why Lemma~\ref{reglemma} is not good enough for
this purpose. But observe that if $G$ is a $k$-graph and $\Part^*$ a family
of partitions obtained by applying Lemma~\ref{reglemma}, then we can
construct a $k$-graph $G'$ from $G$ by associating to each slice $\cJ$
through $\Part^*$ a vertex $v_{\cJ}$ of $V(G)$, and then removing all edges
of $G$ of the form $\{v_{\cJ}\}\cup e$ where $e\in\cJ^{(k-1)}$. Then $G'$
is also regular with respect to $\Part^*$, but whatever slice $\cJ$ through
$\Part^*$ is chosen, the link of $v_{\cJ}$ in $G'$ is not represented on
$\cJ$, so $\cJ$ does not satisfy Lemma~\ref{lem:simpreg}\ref{simpreg:c}.

\medskip

Now we turn to the proof of Lemma~\ref{lem:strReg}.  The basic idea is simple. We will apply the
Regular Approximation Lemma, Lemma~\ref{lem:RAL}, to our $k$-graph $G$ to
generate a family of partitions $\tilde{\Part}^*$ with $\tilde{t}\le t_1^*$
clusters. We will then take a random equipartition of each cluster into
$p(t_1^*)$ parts to obtain a partition $\Part$, and let $\Part^*$ be a
refinement of $\tilde{\Part}^*$ generated by $\Part$. We will show that
deterministically $G$ has the desired regularity with respect to
$\Part^*$, and that with high probability, for any set $\tilde{X}$ of clusters of $\tilde{\Part}^*$ and any $\tilde{X}$-consistent set $X$ of clusters of $\Part^*$, the  
rooted $H$-densities on $X$ of any slice through $\Part^*$ are close to the densities on $\tilde{X}$ of the corresponding slice through $\tilde{\Part}^*$. 

Let us explain why it is necessary in this proof to apply
Lemma~\ref{lem:RAL} rather than Lemma~\ref{reglemma} to obtain
$\tilde{\Part}^*$. The reason is that typically $p(\cdot)$ grows fast and
we must therefore split each cluster into parts which are too small to
control using any regularity obtainable from Lemma~\ref{reglemma}. 

To analyse the random equipartition we need the following
standard concentration inequality for the hypergeometric distribution.

\begin{theorem}[see~{\cite[Theorem 2.10]{JaLuRu:Book}}]\label{thm:hypgeom}
Given a set $V$ and a subset $W$ of $V$, let $V'$ be chosen uniformly at random from all $\ell$-subsets of $V$. We have
$\Prob\Big(|V'\cap W|=\tfrac{\ell|W|}{|V|}\pm t\Big)\ge 1-2\exp\Big(-\tfrac{t^2}{2|V|}\Big)$.
\qed
\end{theorem}

\begin{proof}[Proof of Lemma~\ref{lem:strReg}]
Assume we are given integers $q$, $t_0$ and $s$,
a constant $\eps_k$, functions $r:\NATS\to\NATS$ and $\eps:\NATS\to(0,1]$ and a monotone increasing function
$p:\NATS\to\NATS$.
Without loss of generality, we may assume that $r$ is also monotone increasing
while $\eps$ is monotone decreasing.
We set
\[ \nu=\eps_k^3/(16k!)\quad\text{and}\quad t_0^*=\max\big(t_0,4k^2/\eps_k\big)\,.\]
We let $\eps^*:\NATS\to(0,1]$ be a monotone decreasing function such that for each $x$ we have
\begin{equation}\label{eq:strReg:epss}\eps^*(x)\le \min\Big(\eps\big(p(x)x\big)^2,
\frac{\eps_k^4}{16r\big(p(x)x\big)^2}\Big)\,.
\end{equation}
We further require $\eps^*(\cdot)$ to be small enough that we can apply
Lemma~\ref{restriction} with $\alpha=1/p(x)$, and Lemma~\ref{lem:count}
with $\beta=\tfrac{\eps_k}{10}$ and any $s\le1/\eps_k$, to
$(\mathbf{d},\eps^*(x),\eps^*(x),1)$-regular $k$-complexes, provided that
each $d_i$ is at least $1/x$.

Let $t_1^*$ and $n_0^*$ be returned by Lemma~\ref{lem:RAL} with inputs $q$,
$t_0^*$, $s$, $\nu$ and $\eps^*(\cdot)$. Let $t_1 :=p(t_1^*)t_1^*$, $\eps^*:=\eps^*(t_1^*)$, $\eps:=\eps(t_1)$ and $r:=r(t_1)$, so in particular we have $\eps^* \leq \eps^2$ by definition of $\eps^*$. Let $m_0\ge
n_0^*$ be large enough for the above applications of Lemmas~\ref{lem:count}
and~\ref{restriction} with $x=t_1^*$ and clusters of size $m\ge m_0$. We then
choose $n_0\ge m_0t_1$ sufficiently large for the union bound at the end of
the proof.

Let $V$ be a set of $n \geq n_0$ vertices, where $n$ is divisible by $t_1!$,
and let $\Qart$ partition $V$ into at most $q$ parts of equal size. Let
$G_1, \dots, G_s$ be edge-disjoint $k$-graphs on the vertex set $V$. We
start by applying Lemma~\ref{lem:RAL} (with the inputs stated above), which yields
$k$-graphs $G'_1,\ldots,G'_s$ with $|G_i\Delta G'_i| \leq \nu n^k$ for
each $1\le i\le s$, and a $(k-1)$-family of partitions $\tilde{\Part}^*$ on $V$ which is
$(t_0^*,t_1^*,\eps^*)$-equitable, whose ground
partition $\tilde{\Part}$ refines $\Qart$, and which is such that each $G'_i$ is perfectly
$(\eps^*,1)$-regular with respect to $\tilde{\Part}^*$. Let the density vector of
$\tilde{\Part}^*$ be $\mathbf{d}=(d_2,\ldots,d_{k-1})$. By definition of `equitable', the entries of $\mathbf{d}$ are all at least $1/t_1^*$.

We let the partition $\Part$ of $V$ be obtained by partitioning each cluster $X$ of $\tilde{\Part}^*$ uniformly at random into~$p(t_1^*)$ parts of equal size. We let $\tilde{t}$ be the number of
clusters of $\tilde{\Part}$, and $t$ be the number of clusters of
$\Part$. Thus we have $t=p(t_1^*)\tilde{t}$, which is
property~\ref{strReg:numP}. Since $\tilde{\Part}$ refines $\Qart$, so also
$\Part$ refines $\Qart$, giving property~\ref{strReg:refQ}. 

We now obtain a
family of partitions $\Part^*$ with ground partition $\Part$ as follows. For
each $2\le i\le k-1$ and each $i$-cell $C$ in $\tilde{\Part}^*$ on the
$i$ clusters $\mathcal{C}$ of $\tilde{\Part}$, we put into $\Part^*$ each
of the $i$-uniform induced subgraphs of $C$ obtained by choosing $i$ clusters of
$\Part$, one in each cluster of $\mathcal{C}$.
We then need to add further $2$-cells between pairs of clusters in $\Part$
which were both contained in one cluster of $\tilde{\Part}$. We do this by
choosing for each such pair $X_1,X_2$ of clusters in~$\Part$ an arbitrary
equipartition of the complete bipartite graph with vertex classes $X_1$ and $X_2$ into $1/d_2$ cells which are each
$(d_2,\sqrt{\eps^*},1)$-regular with respect to the $1$-graph formed by the vertices of $X_1$ and $X_2$ (for example, it is not hard to show that a random partition has this property with high probability). We then need to choose all those $3$-cells whose supporting
clusters do not lie in distinct parts of $\tilde{\Part}$, and so on. We do this in the same
way, while also ensuring that we are consistent with the $2$-cells we just
chose, and so on, so that we obtain a family of partitions. By
construction $\Part^*$ is generated from $\tilde{\Part}^*$ by $\Part$,
giving property~\ref{strReg:genP}. Now consider any $\tilde{\Part}$-partite set $Q \in \binom{V}{k}$, let $\tilde{\cJ}(Q)$ be the $k$-partite $(k-1)$-complex whose edge set is the union of the cells $\tilde{C}_Q$ of $\tilde{\Part}^*$ which contain proper subsets of $Q$, and similarly let $\cJ(Q)$ be the $k$-partite $(k-1)$-complex whose edge set is the union of the cells $C_Q$ of $\Part^*$ which contain proper subsets of $Q$. The fact that $\tilde{\Part}^*$ is $(t_0^*, t_1^*, \eps^*)$-equitable implies that $\tilde{\cJ}(Q)$ is $(\bf{d}, \eps^*, \eps^*, 1)$-regular, whereupon Lemma~\ref{restriction} implies that $\cJ(Q)$ is $(\bf{d}, \sqrt{\eps^*}, \sqrt{\eps^*}, 1)$-regular. It follows that for any $2 \leq i \leq k-1$ and any $i$-cell $C$ of $\Part^*$, if $C$ was obtained as an induced subgraph of an
$i$-cell of $\tilde{\Part}^*$, then $C$ is $(d_i,\sqrt{\eps^*},1)$-regular with respect to the $(i-1)$-cells of $\Part^*$ on which it is supported. On the other hand, if $C$ was not obtained in this manner, then by construction $C$ is $(d_i,\sqrt{\eps^*},1)$-regular with respect to these $(i-1)$-cells. Since $t_0^*\ge t_0$, $t_1=p(t_1^*)t_1^*\ge t_1^*$ and
$\sqrt{\eps^*}\le\eps$, we conclude that $\Part^*$ is
$(t_0,t_1,\eps)$-equitable with density vector $\mathbf{d}$. This
establishes property~\ref{strReg:equit}.

We next want to establish property~\ref{strReg:regP}. Fix a graph $G=G_i$, and let $G'=G'_i$. The next claim shows that  if $G$ is not $(\eps_k,r)$-regular with respect to some $\tilde{\Part}$-partite polyad of $\Part^*$, then $G\Delta G'$ is dense on that polyad.

\begin{claim}\label{cl:regP}
For any $\tilde{\Part}$-partite $k$-set $Q \in \binom{V}{k}$, if $G$ is not $(\eps_k,r)$-regular with respect to the polyad $\hat{P} = \hat{P}(Q; \Part^*)$ of $\Part^*$, then $d\big(G'\Delta G\mid \hat{P}\big)>(\eps_k/2)^2$.
\end{claim}
\begin{claimproof}
  Let $\mathbf{W}=(W_1,\ldots,W_{r})$ be any collection of~$r$ subgraphs
  of~$\hat{P}$ such that $|K_k(\mathbf{W})|\ge\eps_k |K_k(\hat{P})|$.  We
  split these $r$ subgraphs into \emph{large} subgraphs $W_j$ with
  $|K_k(W_j)|\ge \sqrt{\eps^*}|K_k(\hat{P})|$, and the remaining
  \emph{small} subgraphs.

Since $Q$ is $\tilde{\Part}$-partite, Lemma~\ref{restriction} implies 
that $G'$ is $(\sqrt{\eps^*},1)$-regular with respect to $\hat{P}$, so on
  any large $W_j$ we have $d(G'|W_j)=d(G'|\hat{P})\pm\sqrt{\eps^*}$. On the
  other hand the total number of $k$-cliques supported on small subgraphs
  $W_{j}$ is by definition of `small' less than
  $r\sqrt{\eps^*}|K_k(\hat{P})|$. It follows that
  \[d(G'|\mathbf{W})=d(G'|\hat{P})\pm \tfrac{2r\sqrt{\eps^*}}{\eps_k}\,.\]

  If $d(G'\Delta G|\hat{P})\le(\eps_k/2)^2$, then by definition we have
  $d(G'|\hat{P})=d(G|\hat{P})\pm(\eps_k/2)^2$. Furthermore, since we
  assumed $|K_k(\mathbf{W})|\ge\eps_k |K_k(\hat{P})|$, we have
  \[
  d(G'|\mathbf{W})
  =d(G|\mathbf{W})\pm d(G'\Delta G|\hat{P})\frac{|K_k(\hat{P})|}{|K_k(\mathbf{W})|}
  =d(G|\mathbf{W})\pm\eps_k/4\,. 
  \]
  Putting these together we have
  \[d(G|\mathbf{W})=d(G|\hat{P})\pm\big(\tfrac{\eps_k^2}{4}+\tfrac{2r\sqrt{\eps^*}}{\eps_k}+\tfrac{\eps_k}{4}\big)\eqByRef{eq:strReg:epss}d(G|\hat{P})\pm\eps_k\,,\]
  hence $G$ is $(\eps_k,r)$-regular with respect to $\hat{P}$.
\end{claimproof}

Let $\mathcal{B}$ be the set of all polyads $\hat{P} = \hat{P}(Q; \Part^*)$ of $\Part^*$ for which $Q \in \binom{V}{k}$ is $\tilde{\Part}$-partite and $G$ is not $(\eps_k, r)$-regular with respect to $\hat{P}$. 
Now using the fact that $e(G\Delta G')\le \nu n^k$ and Claim~\ref{cl:regP} we have
\[\nu n^k\ge\sum_{\hat{P} \in \mathcal{B}} d(G\Delta G'\mid\hat{P})\big|K_k(\hat{P})\big|> (\eps_k/2)^2\sum_{\hat{P} \in \mathcal{B}}\big|K_k(\hat{P})\big|\,,\]
and hence the number of $\tilde{\Part}$-partite sets $Q \in \binom{V}{k}$ such that $G$ is not $(\eps_k,r)$-regular with respect to $\hat{P}(Q; \Part^*)$ is at most $\nu n^k(\eps_k/2)^{-2}\le\tfrac12\eps_k\binom{n}{k}$, where the inequality is by choice of $\nu$ and $n_0$. Moreover, the number of $k$-sets $Q \in \binom{V}{k}$ which are not $\tilde\Part$-partite is at most $\binom{n}{k-1}\tfrac{(k-1)n}{t^*_0}$, which by choice of $t_0^*$ is at most $\tfrac{1}{2}\eps_k\binom{n}{k}$. Putting these together we see that there are at most $\eps_k\binom{n}{k}$ $\Part$-partite sets $Q \in \binom{V}{k}$ such that $G$ is not $(\eps_k,r)$-regular with respect to $\hat{P}(Q; \Part^*)$. In other words, $G$ is $(\eps_k,r)$-regular with respect to $\Part^*$, proving property~\ref{strReg:regP}.

It remains to show that property~\ref{strReg:spread} holds with positive probability. To that end, fix a $k$-graph $G=G_i$, $1\le\ell\le
1/\eps_k$, a $k$-graph $H$ equipped with roots $x_1,\ldots,x_\ell$ such that
$v(H) \le 1/\eps_k$, vertices $v_1,\ldots,v_\ell$ in $V$, a slice $\tilde{\cJ}$ through
$\tilde{\Part}^*$ and a $v(\Hskel)$-tuple of clusters $\tilde{X}$ of $\tilde{\Part}^*$, and let $h := v(\Hskel)$. Now for any permutation $\phi$ of $1, \dots, h$, the rooted copies
of $H$ in $G$ supported on $\tilde{\cJ}[\tilde{X}]$, with the $i$th vertex of $\Hskel$ in the $\phi(i)$th cluster of $\tilde{X}$, correspond to a $h$-uniform
$\tilde{\Part}$-partite hypergraph $F_\phi$.
Let $X$ be an $\tilde{X}$-consistent $h$-tuple of clusters of
$\Part^*$. Then we want to show that for any fixed slice $\cJ$ through $\Part^*$ such that $\cJ[X]\subset\tilde{\cJ}[\tilde{X}]$, with high probability we have
\[d_H(G;v_1,\ldots,v_\ell,\tilde{\cJ}[\tilde{X}])=d_H(G;v_1,\ldots,v_\ell,\cJ[X])\pm\eps_k\,.\]
By Lemma~\ref{lem:count} the number of labelled $\Hskel$-copies in
$\tilde{\cJ}[\tilde{X}]$ for which no two vertices of $\Hskel$ lie in the same cluster of $\tilde{\cJ}$ is
$$ n'_{\Hskel}(\tilde{\cJ}[\tilde{X}])= h! \big(1\pm\tfrac{\eps_k}{10}\big)\big(\tfrac{n}{\tilde{t}}\big)^{v(\Hskel)}\prod_{i=2}^{k-1}d_i^{e_i(\Hskel)}\,,$$
where the $h!$ term is the number of possible allocations of vertices of $\Hskel$ to clusters of $\tilde{J}$.
Similarly
$$ n'_{\Hskel}(\cJ[X])= h! \big(1\pm\tfrac{\eps_k}{10}\big)\big(\tfrac{n}{t}\big)^{v(\Hskel)}\prod_{i=2}^{k-1}d_i^{e_i(\Hskel)}\,.$$
It therefore suffices to show that with high probability, for each of the $h!$ permutations $\phi$ we have
\begin{equation}\label{eq:strReg:desired}d(F_\phi)=d(F_\phi[X])\pm \frac{\eps_k}{2(h!)}\prod_{i=2}^{k-1}d_i^{e_i(\Hskel)}\,.
\end{equation}

Fix any permutation $\phi$, and for each $0\le j\le h$, let $F_j$ denote the subgraph of $F_\phi$ given by taking the first $j$ clusters from $X$ and the last $h-j$ clusters from $\tilde{X}$. Thus we have $F_0=F_\phi$ and $F_{h}=F_\phi[X]$. Observe that for each $j$, the distribution of the $j$th cluster of $X$ is the uniform distribution over all $n/t$-subsets of the $j$th cluster of $\tilde{X}$. We can thus use the following claim to show that for each $1\le j\le h$, with high probability the densities $d(F_{j-1})$ and $d(F_j)$ are close.

\begin{claim}\label{clm:equipart}
  For each $s,p\in\NATS$ and each $\delta>0$, let $B$ be an
  $s$-partite $s$-uniform hypergraph with parts $V_1,\ldots,V_s$. Choose uniformly at random a subset $V' \subseteq V_1$ of size $|V_1|/p$, and let $B'=B[V',V_2,\ldots,V_s]$. Then with
  probability at least \[1- \tfrac{4}{\delta}\exp\big(-\tfrac{\delta^4|V_1|}{32p^2}\big)\,,\] we have $d(B')=d(B)\pm\delta$.
\end{claim}
\begin{claimproof}
  We partition the vertices $V_1$ into sets $W_0, W_1,\ldots,W_{2/\delta}$, with the
  property that all vertices in $W_j$ have degree in $B$ in the interval
  \[\Big(\frac{(j-1)\delta |V_2|\cdots|V_s|}{2},\frac{j\delta
    |V_2|\cdots|V_s|}{2}\Big]\,.\] Let $V'$ be a random subset of $V_1$ of
  size $\ell=|V_1|/p$, and $B'=B[V',V_2,\ldots,V_s]$. By
  Theorem~\ref{thm:hypgeom}, for each $1\le j\le 2/\delta$ we have
  \[\Prob\big(|V'\cap W_j|=\tfrac{|W_j|}{p}\pm t\big)\ge 1-2\exp\big(-\tfrac{t^2}{2|V_1|}\big)\,.\]
  Taking $t=\delta^2|V_1|/(4p)$ and using a union bound, we conclude that we have
  \[|V'\cap W_j|=\tfrac{|W_j|}{p}\pm\tfrac{\delta^2
  |V_1|}{4p}\]
  for each $1\le j\le 2/\delta$ with probability at least
  \[1- \tfrac{4}{\delta}\exp\big(-\tfrac{\delta^4|V_1|}{32p^2}\big)\,.\]
  Conditioning on this likely event, we have
  \begin{align*}
   e(B')&=\sum_{j=1}^{2/\delta}|V'\cap W_j|\cdot\big(j-\tfrac12\pm\tfrac12\big)\tfrac{\delta|V_2|\cdots|V_s|}{2}\\
   &=\sum_{j=1}^{2/\delta}\Big(\tfrac{|W_j|}{p}\pm\tfrac{\delta^2
   |V_1|}{4p}\Big)\cdot\big(j-\tfrac12\pm\tfrac12\big)\tfrac{\delta|V_2|\cdots|V_s|}{2}\\
   &=\sum_{j=1}^{2/\delta}\tfrac{|W_j|}{p}\cdot\big(j-\tfrac12\pm\tfrac12\big)\tfrac{\delta|V_2|\cdots|V_s|}{2}\pm\tfrac{\delta}{2p}|V_1|\cdots|V_s|\\
   &=\tfrac{1}{p}e(B)\pm\Big( \sum_{j=1}^{2/\delta}\tfrac{|W_j|}{p}\cdot\tfrac{\delta|V_2|\cdots|V_s|}{2}\Big)\pm\tfrac{\delta}{2p}|V_1|\cdots|V_s|\\
   &=\tfrac{1}{p}e(B)\pm\tfrac{\delta}{p}|V_1|\cdots|V_s|\,,
  \end{align*}
  which gives the desired conclusion.    
\end{claimproof}

Applying Claim~\ref{clm:equipart} to $B=F_{j-1}$ and $B'=F_j$ for each $1\le j\le h$, with
\[p=p(t_1^*)\quad\text{ and }\quad\delta=\frac{1}{2h(h!)}\eps_k\prod_{i=2}^{k-1}d_i^{e_i(\Hskel)}\,,\]
we conclude
that with probability at least
\[1- \tfrac{4h}{\delta}\exp\big(-\tfrac{\delta^4n}{32p^2t_1^*}\big)\,,\] 
we have $d(F_0)=d(F_{h})\pm \delta h$. By choice of $\delta$ this implies that the equation~\eqref{eq:strReg:desired} holds for the permutation $\phi$.

Now we take a union bound over the $s$ choices of $G=G_i$, the $1/\eps_k$ choices of $\ell$, the at most $2^{\binom{1/\eps_k}{k}}$ choices of rooted $H$, the at most $n^{1/\eps_k}$ choices of $v_1,\ldots,v_\ell$, the $\prod_{j=2}^{k-1}d_j^{-\binom{\tilde{t}}{j}}$ choices of $\tilde{\cJ}$, the at most $\binom{\tilde{t}}{1/\eps_k}$ choices of $\tilde{X}$, the $\prod_{j=2}^{k-1}d_j^{-\binom{t}{j}}$ choices of $\cJ$, the at most $\binom{t}{1/\eps_k}$ choices of $X$, and the $h! \leq (1/\eps_k)!$ permutations $\phi$ of $1, \dots, h$. Since we have the lower bound $d_j\ge 1/t_1^*$ for each $j$, the total number of events over which we take a union bound is polynomial in $n$. The failure probability of each good event is exponentially small in $n$, so if $n_0$ was chosen sufficiently large, with positive probability all good events hold. In other words, with positive probability~\ref{strReg:spread} holds, which completes the proof.
\end{proof}

 
\bibliographystyle{amsplain_yk}
\bibliography{HyperErdosGallai}

\providecommand{\bysame}{\leavevmode\hbox to3em{\hrulefill}\thinspace}
\providecommand{\MR}{\relax\ifhmode\unskip\space\fi MR }
\providecommand{\MRhref}[2]{%
  \href{http://www.ams.org/mathscinet-getitem?mr=#1}{#2}
}
\providecommand{\href}[2]{#2}
\def\MR#1{\relax}
\begin{thebibliography}{10}

\bibitem{Allen}
P.~Allen, \emph{Minimum degree conditions for cycles}, in preparation.

\bibitem{Alon}
N.~Alon, \emph{The largest cycle of a graph with a large minimal degree}, J.
  Graph Theory \textbf{10} (1986), 123--127.

\bibitem{bollobas}
B.~Bollob\'as, \emph{Combinatorics}, Cambridge University Press, 1986.

\bibitem{Bondy}
J.~A. Bondy, \emph{Pancyclic graphs.~{I}}, J. Combinatorial Theory Ser. B
  \textbf{11} (1971), 80--84.

\bibitem{CheFauGouJacLes}
G.~Chen, R.~J. Faudree, R.~J. Gould, M.~S. Jacobson, and L.~Lesniak,
  \emph{Hamiltonicity in balanced {$k$}-partite graphs}, Graphs Combin.
  \textbf{11} (1995), no.~3, 221--231.

\bibitem{FracMatch}
D.~Christofides, J.~Hladk{\'y}, and A.~M{\'a}th{\'e}, \emph{Hamilton cycles in
  dense vertex-transitive graphs}, J. Combin. Theory Ser. B \textbf{109}
  (2014), 34--72. \MR{3269902}

\bibitem{WeakReg}
F.~R.~K. Chung, \emph{Regularity lemmas for hypergraphs and quasi-randomness},
  Random Structures Algorithms \textbf{2} (1991), no.~2, 241--252. \MR{1099803
  (92d:05117)}

\bibitem{CFKO}
O.~Cooley, N.~Fountoulakis, D.~K{\"u}hn, and D.~Osthus, \emph{Embeddings and
  {R}amsey numbers of sparse {$k$}-uniform hypergraphs}, Combinatorica
  \textbf{29} (2009), no.~3, 263--297. \MR{2520273 (2010g:05241)}

\bibitem{Dirac}
G.~A. Dirac, \emph{Some theorems on abstract graphs}, Proc. London Math. Soc.
  \textbf{2} (1952), 69--81.

\bibitem{ErdGal}
P.~Erd{\H{o}}s and T.~Gallai, \emph{On maximal paths and circuits of graphs},
  Acta Math. Acad. Sci. Hungar \textbf{10} (1959), 337--356. \MR{0114772 (22
  \#5591)}

\bibitem{Gowers}
W.~T. Gowers, \emph{Hypergraph regularity and the multidimensional
  {S}zemer\'edi theorem}, Ann. of Math. (2) \textbf{166} (2007), no.~3,
  897--946. \MR{2373376 (2009d:05250)}

\bibitem{GyKatLem}
E.~Gy\H{o}ri, G.~Y. Katona, and N.~Lemons, \emph{Hypergraph {E}xtensions of the
  {E}rd{\H{o}}s-{G}allai {T}heorem}, in preparation.

\bibitem{3CycleRamsey}
P.~E. Haxell, T.~{\L}uczak, Y.~Peng, V.~R{\"o}dl, A.~Ruci{\'n}ski, and
  J.~Skokan, \emph{The {R}amsey number for 3-uniform tight hypergraph cycles},
  Combin. Probab. Comput. \textbf{18} (2009), no.~1-2, 165--203. \MR{2497379
  (2010k:05195)}

\bibitem{HuangLohSud}
H.~Huang, P.-S. Loh, and B.~Sudakov, \emph{The size of a hypergraph and its
  matching number}, Combin. Probab. Comput. \textbf{21} (2012), no.~3,
  442--450.

\bibitem{JaLuRu:Book}
S.~Janson, T.~{\L}uczak, and A.~Ruci{\'n}ski, \emph{Random graphs},
  Wiley-Interscience, New York, 2000.

\bibitem{KeeDesign}
P.~Keevash, \emph{The existence of designs}, submitted (arXiv:1401.3665).

\bibitem{HypBlow}
P.~Keevash, \emph{A hypergraph blow-up lemma}, Random Structures Algorithms
  \textbf{39} (2011), no.~3, 275--376. \MR{2816936}

\bibitem{KeeMyc}
P.~Keevash and R.~Mycroft, \emph{A geometric theory for hypergraph matching},
  Mem. Amer. Math. Soc. \textbf{233} (2014), no.~1098.

\bibitem{WeakCount}
Y.~Kohayakawa, B.~Nagle, V.~R{\"o}dl, and M.~Schacht, \emph{Weak hypergraph
  regularity and linear hypergraphs}, J. Combin. Theory Ser. B \textbf{100}
  (2010), no.~2, 151--160. \MR{2595699 (2011g:05215)}

\bibitem{KRS}
Y.~Kohayakawa, V.~R{\"o}dl, and J.~Skokan, \emph{Hypergraphs, quasi-randomness,
  and conditions for regularity}, J. Combin. Theory Ser. A \textbf{97} (2002),
  no.~2, 307--352.

\bibitem{lcycles}
D.~K{\"u}hn, R.~Mycroft, and D.~Osthus, \emph{Hamilton {$\ell$}-cycles in
  uniform hypergraphs}, J. Combin. Theory Ser. A \textbf{117} (2010), no.~7,
  910--927.

\bibitem{LucRamsey}
T.~{\L}uczak, \emph{{$R(C\sb n,C\sb n,C\sb n)\leq(4+o(1))n$}}, J. Combin.
  Theory Ser. B \textbf{75} (1999), no.~2, 174--187. \MR{1676887 (2000b:05096)}

\bibitem{McDiarmid}
C.~McDiarmid, \emph{On the method of bounded differences}, Surveys in
  Combinatorics \textbf{141} (1989), 148--188.

\bibitem{MooMos}
J.~Moon and L.~Moser, \emph{On {H}amiltonian bipartite graphs}, Israel J. Math.
  \textbf{1} (1963), 163--165.

\bibitem{NRS}
B.~Nagle, V.~R{\"o}dl, and M.~Schacht, \emph{The counting lemma for regular
  $k$-uniform hypergraphs}, Random Structures Algorithms \textbf{28} (2006),
  no.~2, 113--179.

\bibitem{NikiSchelp}
V.~Nikiforov and R.~H. Schelp, \emph{Cycle lengths in graphs with large minimum
  degree}, J. Graph Theory \textbf{52} (2006), no.~2, 157--170.

\bibitem{RodlEH}
V.~R{\"o}dl, \emph{On a packing and covering problem}, European J. Combin.
  \textbf{6} (1985), 69--78.

\bibitem{RSHamSurvey}
V.~R{\"o}dl and A.~Ruci{\'n}ski, \emph{Dirac-type questions for hypergraphs---a
  survey (or more problems for {E}ndre to solve)}, An irregular mind, Bolyai
  Soc. Math. Stud., vol.~21, J\'anos Bolyai Math. Soc., Budapest, 2010,
  pp.~561--590. \MR{2815614}

\bibitem{RRSHam}
V.~R{\"o}dl, A.~Ruci{\'n}ski, and E.~Szemer{\'e}di, \emph{An approximate
  {D}irac-type theorem for {$k$}-uniform hypergraphs}, Combinatorica
  \textbf{28} (2008), no.~2, 229--260. \MR{2399020 (2009a:05146)}

\bibitem{RRS3unif}
V.~R{\"o}dl, A.~Ruci{\'n}ski, and E.~Szemer{\'e}di, \emph{Dirac-type conditions
  for {H}amiltonian paths and cycles in 3-uniform hypergraphs}, Adv. Math.
  \textbf{227} (2011), no.~3, 1225--1299. \MR{2799606 (2012d:05213)}

\bibitem{RS2}
V.~R{\"o}dl and M.~Schacht, \emph{Regular partitions of hypergraphs: counting
  lemmas}, Combin. Probab. Comput. \textbf{16} (2007), no.~6, 887--901.
  \MR{2351689 (2008j:05238)}

\bibitem{RS}
\bysame, \emph{Regular partitions of hypergraphs: regularity lemmas}, Combin.
  Probab. Comput. \textbf{16} (2007), no.~6, 833--885. \MR{2351688
  (2008h:05083)}

\bibitem{RSk}
V.~R{\"o}dl and J.~Skokan, \emph{Regularity lemma for {$k$}-uniform
  hypergraphs}, Random Structures Algorithms \textbf{25} (2004), no.~1, 1--42.
  \MR{2069663 (2005d:05144)}

\bibitem{SchMit}
E.~Schmeichel and J.~Mitchem, \emph{Bipartite graphs with cycles of all even
  lengths}, J. Graph Theory \textbf{6} (1982), 429--439.

\bibitem{SzeReg}
E.~Szemer{\'e}di, \emph{Regular partitions of graphs}, Probl\`emes
  combinatoires et th\'eorie des graphes (Orsay, 1976), Colloques
  Internationaux CNRS, vol. 260, CNRS, 1978, pp.~399--401.

\end{thebibliography}
 
\end{document}